\documentclass[10pt,reqno, twoside]{article}
\usepackage[letterpaper,margin=1in,footskip=0.30in]{geometry}
\usepackage{microtype}
\usepackage{amsmath,amssymb,amsthm}
\usepackage{mathtools}
\usepackage[dvipsnames]{xcolor}
\usepackage{enumitem}
\setlist{noitemsep}
\usepackage{mathrsfs}
\usepackage{tikz-cd}
\usepackage{caption}
\usepackage{array}
\usepackage[pdfusetitle,colorlinks]{hyperref}
\hypersetup{citecolor=black,linkcolor=black,urlcolor=black}
\usepackage[capitalise,noabbrev]{cleveref}
\crefformat{equation}{\ensuremath{(#2#1#3)}}
\crefmultiformat{equation}{\ensuremath{(#2#1#3)}}{ and~\ensuremath{(#2#1#3)}}{, \ensuremath{(#2#1#3)}}{, and~\ensuremath{(#2#1#3)}}
\numberwithin{figure}{subsection}
\numberwithin{equation}{subsection}
\newtheorem{theorem}[figure]{Theorem}
\newtheorem{lemma}[figure]{Lemma}

\newtheorem{proposition}[figure]{Proposition}
\theoremstyle{definition}
\newtheorem{definition}[figure]{Definition}
\newtheorem{notation}[figure]{Notation}
\theoremstyle{definition}
\newtheorem{remark}[figure]{Remark}
\theoremstyle{definition}
\newtheorem{example}[figure]{Example}
\theoremstyle{definition}
\newtheorem{construction}[figure]{Construction}
\theoremstyle{cited}

\usepackage{lipsum}
\usepackage{fancyhdr}
\fancypagestyle{mypagestyle}{
\fancyhead[OC]{\footnotesize SHUBHODIP MONDAL}
\fancyhead[EC]{\footnotesize  $\mathbb G _a^{\mathrm{perf}}$-MODULES AND DE RHAM COHOMOLOGY}
\fancyfoot[C]{\thepage}

\vspace{-10mm}
}

\newcommand{\w}{\text}

\newcolumntype{C}[1]{>{\centering\arraybackslash}p{#1}}

\title{\large \textbf{$\mathbb G _a^{\mathrm{perf}}$-MODULES AND DE RHAM COHOMOLOGY}}
\author{\small SHUBHODIP MONDAL}
\date{}

\begin{document}

\maketitle
\begin{abstract}
We prove that algebraic de Rham cohomology as a functor defined on smooth $\mathbb{F}_p$-algebras is formally \'etale in a precise sense. This result shows that given de Rham cohomology, one automatically obtains the theory of crystalline cohomology as its \textit{unique} functorial deformation. To prove this, we define and study the notion of a pointed $\mathbb{G}_a^{\mathrm{perf}}$-module and its refinement which we call a quasi-ideal in $\mathbb{G}_a^{\mathrm{perf}}$ -- following Drinfeld's terminology. Our main constructions show that there is a way to ``unwind" \textit{any} pointed $\mathbb{G}_a^{\text{perf}}$-module and define a notion of a cohomology theory for algebraic varieties. We use this machine to redefine de Rham cohomology theory and deduce its formal \'etalness and a few other properties.
\end{abstract}
\tableofcontents
\clearpage
\section{Introduction}
\subsection{Overview of the results}
Let $X$ be a scheme over a field $k.$ Grothendieck defined the algebraic de Rham cohomology of $X$ to be the hypercohomology of the algebraic de Rham complex $\Omega_X^*$ \cite{Gro66}. When $k$ is a field of characteristic zero, de Rham cohomology forms a Weil cohomology theory for smooth proper varieties over $k.$ But when $k$ has positive characteristic, for example $k = \mathbb{F}_p$, then the theory of de Rham cohomology does not form a Weil cohomology theory. In particular, the de Rham cohomology groups are killed by $p$. To rectify this situation, Grothendieck \cite{Gro68} and Berthelot \cite{Ber74} devised the theory of crystalline cohomology. For a smooth algebraic variety $X$ over $\mathbb{F}_p,$ its ($n$-truncated) crystalline cohomology $R \Gamma_{\mathrm{crys}}(X/\mathbb{Z}/p^n)$ is a deformation of de Rham cohomology; in the sense that $R \Gamma_{\mathrm{crys}}(X/\mathbb{Z}/p^n) \otimes ^{L}_{\mathbb{Z}/p^n \mathbb{Z}} \mathbb{F}_p \simeq R \Gamma(X,\Omega_X^*).$ However, potentially there could exist some other cohomology theory which is also a deformation of de Rham cohomology. Our goal is to show that this does not happen. In particular, we show that de Rham cohomology theory for varieties over $\mathbb{F}_p$ is \textit{formally \'etale}. Thus, given the theory of de Rham cohomology, one can realize crystalline cohomology as its \textit{unique} deformation. To make this precise, we fix some notations. We let $\mathrm{Alg}^{\mathrm{sm}}_{\mathbb{F}_p}$ denote the category of smooth $\mathbb{F}_p$-algebras and $\mathrm{CAlg}(D(A))$ denote the $\infty$-category of commutative algebra objects (in the sense of \cite[2.1.3]{Lur17}) in the derived $\infty$-category $D(A)$ of an Artinian local ring $A$ with residue field $\mathbb{F}_p.$ In other words, $\w{CAlg}(D(A))$ is the $\infty$-category of $\mathbb{E}_{\infty}$-algebras over $A.$ We show the following 

\begin{theorem}[\cref{mainthm}]\label{inhosp3}Let $$\mathrm{dR} : \mathrm{Alg}^{\mathrm{sm}}_{\mathbb F _p} \to \mathrm{CAlg}(D(\mathbb{F}_p))$$
be the algebraic de Rham cohomology functor. Given an Artinian local ring $(A, \mathfrak{m})$ with residue field $\mathbb{F}_p$, the functor $\mathrm{dR}$ admits a unique deformation $$\mathrm{dR}':  \mathrm{Alg}^{\mathrm{sm}}_{\mathbb F _p} \to \mathrm{CAlg}(D(A)) .$$ Further, the deformation $\mathrm{dR}'$ is unique up to unique isomorphism (More precisely, the space of deformations of $\mathrm{dR}$ is contractible, see \cref{bbleave}). \end{theorem}

\begin{remark}\label{hun}
Roughly speaking, \cref{inhosp3} proves that the theory of crystalline cohomology is the \textit{unique} functorial deformation of de Rham cohomology theory. Thus, it offers a simple new characterization of crystalline cohomology. More precisely, when $A= \mathbb{Z}/p^n$, the ($n$-truncated) crystalline cohomology functor $R\Gamma_{\w{crys}}((\,\cdot\,)/\mathbb{Z}/p^n)$ is uniquely isomorphic to $\w{dR}',$ where the latter is as characterized by \cref{inhosp3}.
\end{remark}{}

 The special analogue of \cref{inhosp3} for $A = \mathbb{Z}_p$ (instead of an arbitrary Artinian local ring) was a result of Bhatt, Lurie and Mathew \cite[Thm. 10.1.2]{BLM20}. Our \cref{inhosp3} gives a generalization of their result, which works for arbitrary Artinian local rings $A;$ the case when $A = \mathbb{Z}_p = \varprojlim \mathbb{Z}/p^n$ can now be deduced via a limit argument. Since \cref{inhosp3} works with arbitrary Artinian local rings, it establishes that the de Rham cohomology functor for smooth varieties over $\mathbb{F}_p$ is ``formally \'etale." The proof of \cite[Thm. 10.1.2]{BLM20} due to Bhatt, Lurie and Mathew crucially uses that $\mathbb{Z}_p$ and other relevant rings appearing in their work are $p$-torsion free. However, Artinian local rings with residue field $\mathbb{F}_p$ are always $p$-torsion, which presents major difficulties in approaching \cref{inhosp3} in a similar fashion.

\vspace{2mm}
We use a very different approach to prove \cref{inhosp3}. In fact, we develop a new approach to the theory of algebraic de Rham cohomology, by compressing its ``essence" in a simpler algebro-geometric structure that we introduce, which we call a \textit{pointed} $\mathbb{G}_a^{\w{perf}}$-\textit{module} (\cref{def11}); these objects are closely related to some classical constructions in $p$-adic Hodge theory (see \cref{perfect}). We study properties of pointed $\mathbb{G}_a^{\w{perf}}$-modules and its closely related variant called pointed $\mathbb{G}_a$-modules in detail. Then we develop a machine called \textit{unwinding}: for \textit{any} pointed $\mathbb{G}_a^{\w{perf}}$-module $X,$ we build a functor $\w{Un}(X)$ by unwinding $X,$ which can be regarded as a cohomology theory for algebraic varieties. We show that de Rham and crystalline cohomology theory can be rebuilt by unwinding specific pointed $\mathbb{G}_a^{\w{perf}}$-modules (see \cref{though}). After establishing good formal properties of the unwinding construction, we use it to approach \cref{inhosp3}. An outline of the proof of \cref{inhosp3} will be explained following the statement of \cref{def12} below.\vspace{2mm}

Having briefly mentioned the key new players in the proof of \cref{inhosp3}, let us now take a slightly more technical perspective and explain some of the relevant definitions and how they enter the picture. The de Rham cohomology functor takes values in coconnective commutative algebra objects
in the derived category $D(\mathbb{F}_p).$  In order to avoid talking about deformation theory in such a context, it would be convenient for us if we could work with discrete rings instead. In order to do that, instead of working with de Rham cohomology theory, we work with \textit{derived} de Rham cohomology theory as defined and studied in \cite{Ill72} and \cite{Bha12}. We will write $\mathrm{dR}$ to denote derived de Rham cohomology as well; it agrees with the usual algebraic de Rham cohomology for smooth schemes, so the notation is consistent. For our purposes, it has the technical advantage that derived de Rham cohomology theory can be completely understood by its values on a certain class of rings introduced by Bhatt, Morrow and Scholze \cite{BMS19} called \textit{quasiregular semiperfect} (QRSP) algebras. We point out that somewhat similar class of rings appeared in the work of Fontaine and Jannsen as well \cite{FJ13}. If $S$ is a QRSP algebra, its derived de Rham cohomology $\mathrm{dR}(S)$ is then a discrete ring. Therefore, we are equivalently led to the study of $\mathrm{dR}$ as a functor from QRSP algebras to discrete $\mathbb{F}_p$-algebras. In fact, after some reductions that are carried out in \cref{section5.1}, \cref{inhosp3} follows from the following statement formulated in purely $1$-categorical language. Below, $\mathrm{QRSP}$ denotes the category of QRSP algebras and $\mathrm{Alg}_A$ denotes the category of discrete $A$-algebras. We show the following

\begin{theorem}\label{why?}Let $\mathrm{dR} :\mathrm{QRSP} \to \mathrm{Alg}_{\mathbb{F}_p}$ be the derived de Rham cohomology functor. Given an Artinian local ring $(A, \mathfrak{m})$ with residue field $\mathbb{F}_p,$ the functor $\mathrm{dR}$ admits a deformation $\mathrm{dR}':\mathrm{QRSP} \to \mathrm{Alg}_A$ which is unique up to unique isomorphism (\textit{cf. \cref{section5.1}}).
\end{theorem}{}

In \cref{section3}, more generally, we study the category $\mathrm{Fun}(\mathrm{QRSP}, \mathrm{Alg}_{\mathbb{F}_p}),$ where $\mathrm{QRSP}$ denotes the category of QRSP algebras. We show that a certain class of functors, which includes the de Rham cohomology functor, can be realized as some kind of ``unwinding" (\textit{cf.} \cref{athosp}) of a much smaller and more tractable structure which we call a \textit{pointed $\mathbb{G}_a^{\mathrm{perf}}$-module.} In order to make sure that the process of ``unwinding" is well-behaved, we will need to study a special class of pointed $\mathbb{G}_a^{\mathrm{perf}}$-modules, which we call \textit{quasi-ideals} following Drinfeld \cite[Def. 3.1.3]{Dri20}.

\begin{definition}The functor $\mathrm{Alg}_{\mathbb{F}_p} \to \mathrm{Alg}_{\mathbb{F}_p}$ that sends $S \mapsto S^\flat:= \varprojlim_{x \mapsto x^p } S$ can be represented by an affine ring scheme which we denote as $\mathbb{G}_a^{\mathrm{perf}}.$ The underlying affine scheme is given by $\mathrm{Spec}\, \mathbb{F}_p [x^{1/p^\infty}].$
\end{definition}{}

\begin{definition}\label{def11}A pointed $\mathbb{G}_a^{\mathrm{perf}}$-module is the data of a $\mathbb{G}_a^{\mathrm{perf}}$-module scheme $X$ equipped with a map of $\mathbb{G}_a^{\mathrm{perf}}$-module schemes $X \to \mathbb{G}_a^{\mathrm{perf}}.$ The data of the map $X \to \mathbb{G}_a^{\mathrm{perf}}$ will be referred to as a \textit{point} (\textit{cf}. \cref{sec2.2}). In \cref{deg1}, we give some justifications for the terminology ``point" in this context.
\end{definition}{}

\begin{definition}A pointed $\mathbb{G}_a^{\mathrm{perf}}$-module is called a quasi-ideal in 
$\mathbb{G}_a^{\mathrm{perf}}$ if the data of the point denoted as $d: X \to \mathbb{G}_a^{\mathrm{perf}}$ sits in a commutative diagram as below (\textit{cf.} \cref{quasi-ideal}).

\begin{center}
\begin{tikzcd}
X \times X \arrow[r, "\mathrm{id}\times d "] \arrow[d, "d \times \mathrm{id} "] & X \times \mathbb G_a^{\mathrm{perf}} \arrow[d, "\mathrm{action}"] \\
\mathbb G_a^{\mathrm{perf}} \times X \arrow[r, "\mathrm{action}"] & X                             
\end{tikzcd}
\end{center}
The commutativity of the above diagram ensures that for an algebra $R,$ the map $X(R) \to \mathbb{G}_a^{\mathrm{perf}} (R)$ viewed as a complex, where $\mathbb{G}_a^{\mathrm{perf}} (R)$ sits in degree zero, has the structure of a differentially graded algebra \cite[Remark 3.1.2]{Dri20}.
\end{definition}{}

\begin{remark}Later on, we will need to work with $\mathbb{G}_a^{\mathrm{perf}}$-module schemes defined over an Artinian local base ring $A$ with residue field $\mathbb{F}_p$. Most of our constructions are also defined in this generality. However, for the overview, we assume that $A$ is always $\mathbb{F}_p.$
\end{remark}{}

\begin{example}\label{introex1}Let $W$ denote the ring scheme of $p$-typical Witt vectors. Then $W$ has an endomorphism $F$ which is called the Frobenius on $W.$ Note that for any algebra $S,$ the ring of Witt vectors $W(S)$ has an additive endomorphism $V$ (called the Verschiebung), which induces an operator also denoted as $V$ on the group scheme underlying $W.$ For $x, y \in W(S),$ one has $V(x) \cdot y= V(x \cdot F(y)).$ Therefore, if $F(y)=0,$ we must have $V(x) \cdot y=0.$ Further, note that $VW(S)$ is an ideal of the ring $W(S)$ and there is a natural isomorphism $W(S)/ VW(S) \simeq S \simeq \mathbb{G}_a(S).$ These observations imply that the group scheme underlying the kernel of $F$ on the ring scheme $W,$ written as $W[F]$ naturally has the structure of 
a $\mathbb{G}_a$-module (see \cref{g_amod}).\vspace{2mm}

We note that there is also a natural map $W[F] \to \mathbb G_a$ of $\mathbb G_a$-module schemes. Pulling $W[F]$ back along the map $u:\mathbb{G}_a^{\mathrm{perf}} \to \mathbb{G}_a$ of ring schemes (\cref{pullback}) produces a $\mathbb{G}_a^{\mathrm{perf}}$-module scheme which we call $u^* W[F].$ Then $u^* W[F]$ can be equipped with the structure of a pointed $\mathbb{G}_a^{\mathrm{perf}}$-module scheme which is further also a quasi-ideal in $\mathbb{G}_a^{\mathrm{perf}}.$ We point out that the group scheme $W[F]$ is isomorphic to the divided power completion of the additive group scheme $\mathbb{G}_a$, which is denoted as $\mathbb{G}_a^\sharp$ in \cite{Dri20}. This isomorphism is also proven in \cite[Lemma 3.2.6]{Dri20}.
\end{example}{}

Our goal is to use the data of a pointed $\mathbb{G}_a^{\mathrm{perf}}$-module to produce a functor such as de Rham cohomology in a lossless manner. Note that there is a natural functor $\mathfrak{G}: \mathrm{QRSP} \to \mathrm{Alg}_{\mathbb{F}_p}$ which sends $S \mapsto S^\flat$, where $S^\flat$ denotes the tilt of $S$ defined as $S^\flat := \varprojlim_{x \mapsto x^p} S.$ In \cref{athosp}, we construct the (contravariant) unwinding functor denoted by ${\mathrm{Un}}$ which takes in the data of a pointed $\mathbb{G}_a^{\mathrm{perf}}$-module as input and produces a functor from $\mathrm{QRSP} \to \mathrm{Alg}_{\mathbb{F}_p}.$ As a basic example, we note that the functor $\mathfrak{G}$ is the unwinding of the pointed $\mathbb{G}_a^{\mathrm{perf}}$-module given by $\mathbb{G}_a^{\mathrm{perf}}$ itself. Other examples are noted in \cref{quiz} and \cref{though} below. Restricting our attention to quasi-ideals satisfying a particular property, which we call \textit{nilpotent quasi-ideals} (\cref{nilpotent}), we obtain the following.

\begin{theorem}[\cref{nil}]There is a fully faithful (contravariant) embedding of the category of nilpotent quasi-ideals in $\mathbb{G}_a^{\mathrm{perf}}$ inside $\mathrm{Fun}(\mathrm{QRSP}, \mathrm{Alg}_{\mathbb{F}_p})_{\mathfrak{G}/}$ given by the unwinding functor ${\mathrm{Un}}.$
\end{theorem}{}

\begin{example}\label{quiz}We note that $\mathrm{Spec\,} \mathbb{F}_p [x^{1/p^\infty}]/x $ can be equipped with the structure of a pointed $\mathbb{G}_a^{\mathrm{perf}}$-module which we denote as $\alpha^{\natural}.$ Another way to describe $\alpha^{\natural}$ is to say that it is the pointed $\mathbb{G}_a^\mathrm{perf}$-module underlying the kernel of the map $u: \mathbb{G}_a^\mathrm{perf} \to \mathbb{G}_a.$ It is also the same as $u^* \mathrm{Spec} \, \mathbb{F}_p$ where $\mathrm{Spec}\, \mathbb{F}_p$ is the pointed $\mathbb{G}_a$-module underlying the zero group scheme. Applying the unwinding functor to $\alpha^{\natural}$ gives the functor $\mathrm{QRSP} \mapsto \mathrm{Alg}_{\mathbb{F}_p}$ that sends $S \mapsto S.$
\end{example}

\begin{theorem}\label{though}Derived de Rham cohomology naturally viewed as an object $\mathrm{dR}$ of $\mathrm{Fun}(\mathrm{QRSP}, \mathrm{Alg}_{\mathbb{F}_p})_{\mathfrak{G}/}$ is naturally isomorphic to the unwinding of the nilpotent quasi-ideal given by $u^*W[F].$
\end{theorem}{}

The above results indicate that properties of certain objects of $\mathrm{Fun}(\mathrm{QRSP}, \mathrm{Alg}_{\mathbb{F}_p})_{\mathfrak{G}/}$ can be deduced by studying nilpotent quasi-ideals or more generally pointed $\mathbb{G}_a^{\mathrm{perf}}$-modules which is the subject of \cref{section2}. For example, we define a full subcategory of pointed $\mathbb{G}_a^{\mathrm{perf}}$-modules which we call \textit{fractional rank-$1$} pointed $\mathbb{G}_a^{\mathrm{perf}}$-module (\textit{cf.} \cref{satu3})
which has an initial object given by $\alpha^{\natural}.$ By applying the unwinding functor, using \cref{quiz} and the universal property of $\alpha^{\natural}$ mentioned before, one gets the following result. 

\begin{theorem}[\cref{inhosp2}]The natural transformation $\mathrm{gr}^0 : \mathrm{dR} \to \mathrm{id}$ coming from $\mathrm{gr}^0$ of the Hodge filtration in derived de Rham cohomology is the unique natural transformation between $\mathrm{dR}$ and $\mathrm{id}$ viewed as objects of the category $\mathrm{Fun}(\mathrm{QRSP}, \mathrm{Alg}_{\mathbb{F}_p})_{\mathfrak{G}/}.$
\end{theorem}{}

We study a more refined class of objects which we call \textit{pure fractional rank-$1$} pointed $\mathbb{G}_a^{\mathrm{perf}}$-module (\textit{cf.} \cref{pure}). The full subcategory of 
pure fractional rank-$1$ pointed $\mathbb{G}_a^{\mathrm{perf}}$-module has an initial object given by $u^*W[F].$ By applying the unwinding functor, one gets a universal property of the de Rham cohomology functor which we loosely state below.

\begin{theorem}[Universal property of $\mathrm{dR}$]\label{lunch5}Derived de Rham cohomology is a \textit{final object} of a certain full subcategory of $\mathrm{Fun}(\mathrm{QRSP}, \mathrm{Alg}_{\mathbb{F}_p})_{\mathfrak{G}/}$ (\textit{cf.} \cref{univprop}).
\end{theorem}{}

As an application of the universal property, we can deduce the following result \cite[Prop. 10.3.1]{BLM20}.

\begin{theorem}[Bhatt-Lurie-Mathew]\label{def12}If we consider algebraic de Rham cohomology as a functor defined on smooth $\mathbb{F}_p$-algebras denoted as $\mathrm{dR}$, then any endomorphism of $\mathrm{dR}$ that commutes with the $\mathrm{gr}^0$ map of the Hodge filtration $\mathrm{dR} \to \mathrm{id}$ is identity.
\end{theorem}{}

\noindent
\textbf{Outline of the proof of \cref{inhosp3}.}
Once we have developed the properties of the unwinding functor ${\mathrm{Un}}$ in \cref{section3}, we try to use it to prove \cref{inhosp3} in 
\cref{section5}. We have noted that $\mathrm{dR}$ is essentially the data of the quasi-ideal $u^*W[F].$ Our strategy is the following.\vspace{2mm}

\noindent
1. We reduce the problem to the case where the base Artinian local ring $A$ is $\mathbb{F}_p[\epsilon]/\epsilon^2.$ \vspace{1mm}

\noindent
2. Given any deformation $\mathrm{dR}'$ of $\mathrm{dR},$ we extract a quasi-ideal from $\mathrm{dR}'$ denoted as $r (\mathrm{dR}')$ which is a deformation of $u^*W[F].$\vspace{1mm}

\noindent
3. We show that $\mathrm{dR}'$ is essentially determined by the quasi-ideal $r (\mathrm{dR}').$\vspace{1mm}

\noindent
4. We show that any deformation of $u^*W[F]$ to $\mathbb{F}_p[\epsilon]/\epsilon^2$ as a pointed $\mathbb{G}_a^{\mathrm{perf}}$-module is uniquely isomorphic to the trivial deformation obtained by base change. This is proven in \cref{mainthm1}. Therefore $r (\mathrm{dR}')$ is necessarily the trivial deformation of $u^*W[F]$ and by 3, $\mathrm{dR}'$ is necessarily the trivial deformation $\mathrm{dR} \otimes \mathbb{F}_p[\epsilon]/\epsilon^2$ as well.

\vspace{2mm}

\noindent
\textbf{Other approaches to \cref{inhosp3}.}\label{other} Our approach to \cref{inhosp3} uses QRSP algebras in an essential way in order to \textit{not} deal with deformation theory of coconnective $\mathbb{E}_\infty$-rings. Our construction of the unwinding functor ${\mathrm{Un}}$ is also devised in a way to work with the category $\mathrm{Fun}(\mathrm{QRSP}, \mathrm{Alg}_{\mathbb{F}_p}).$ However, in principle, this should not be absolutely necessary. Below we attempt to loosely explain other possible approaches that could be seen as more natural.\vspace{2mm}

By the reduction in \cref{section5.1}, it is equivalent to address the version of \cref{inhosp3} for the category $\mathrm{Poly}_{\mathbb F_p}$ of finitely generated polynomial algebras over $\mathbb{F}_p$ instead of all smooth algebras. Instead of studying the category $\mathrm{Fun}(\mathrm{QRSP}, \mathrm{Alg}_{\mathbb{F}_p})$, we can study the category $\mathrm{Fun}(\mathrm{Poly}_{\mathbb F_p}, \mathrm{CAlg}(D(\mathbb{F}_p))).$ A functor $F \in \mathrm{Fun}(\mathrm{Poly}_{\mathbb F_p}, \mathrm{CAlg}(D(\mathbb{F}_p)))$ that preserves coproducts would provide an $\mathbb{F}_p$-coalgebra object structure on the $\mathbb{E}_\infty$-ring $F(\mathbb{F}_p[x])$ coming from the $\mathbb{F}_p$-coalgebra structure of $\mathbb{F}_p[x]$ as an object of $\mathrm{Poly}_{\mathbb F_p}.$ One can also try to reverse the situation, i.e., given an $\mathbb{E}_\infty$-ring $K$ with the extra structure of an $\mathbb{F}_p$-coalgebra object, one can try to build a functor ${\mathrm{Un}}_K: \mathrm{Poly}_{\mathbb F_p} \to \mathrm{CAlg}(D(\mathbb{F}_p))$ that would send $\mathbb{F}_p[x] \mapsto K$ and extend in a coproduct preserving way. This version of ``unwinding" is explained in \cref{introex} (in a $1$-categorical language). Assuming good properties of these constructions, in order to approach \cref{inhosp3}, we are led to studying the deformations of the $\mathbb{E}_\infty$-ring $\mathrm{dR}(\mathbb{F}_p[x])$ along with the extra structure of an $\mathbb{F}_p$-coalgebra object.\vspace{2mm}

Using the stacky approach to $p$-adic cohomology theories due to Bhatt--Lurie \cite{BL} and Drinfeld \cite{Dri18} \cite{Dri20}, one can ask a similar question regarding deformation of the $\mathbb{F}_p$-algebra stack ${(\mathbb{A}^1_{\mathbb{F}_p})}^{\mathrm{dR}}$ relevant to \cref{inhosp3}. This is a stack whose cohomology of the structure sheaf recovers $\mathrm{dR}(\mathbb{F}_p [x]).$ Deformations of ${(\mathbb{A}^1_{\mathbb{F}_p})}^{\mathrm{dR}}$ as an $\mathbb{F}_p$-algebra stack seems to be relevant to \cref{inhosp3}. Further, using \cite[Prop. 3.5.1]{Dri20}, ${(\mathbb{A}^1_{\mathbb{F}_p})}^{\mathrm{dR}}$ is the cone of the quasi-ideal given by $W[F].$ Therefore, deformations of ${(\mathbb{A}^1_{\mathbb{F}_p})}^{\mathrm{dR}}$ as an $\mathbb{F}_p$-algebra stack seem related to deformations of the quasi-ideal or the pointed $\mathbb{G}_a$-module given by $W[F]$ which is studied in \cref{section2} of our paper.\vspace{2mm}

In the approach we have taken in this paper (which uses QRSP algebras) we can avoid talking about higher categorical structures and obtain a purely $1$-categorical formulation as mentioned in \cref{why?}. Further, the notion of a pointed $\mathbb{G}_a$-module or a quasi-ideal comes out quite naturally (\textit{cf.} \cref{smoke1}). As a downside, the construction of ``unwinding" seems more convoluted for \text{QRSP} than what it would be for $\mathrm{Poly}_{\mathbb F_p}$. We use quasisyntomic descendability and left Kan extensions to switch between $\mathrm{QRSP}$ and $\mathrm{Poly}_{\mathbb F_p},$ which could potentially be avoided in the other approaches outlined above.\vspace{2mm}

In any case, we point out that a precise formulation of the deformation problems involving the $\mathbb{E}_\infty$-ring $\mathrm{dR}(\mathbb{F}_p[x])$ or the $\mathbb{F}_p$-algebra stack $(\mathbb{A}^1_{\mathbb{F}_p})^{\mathrm{dR}}$ would likely be equivalent to \cref{inhosp3} and therefore they are  answered \textit{a posteriori} after proving \cref{inhosp3}. Also, a comparison of these approaches can lead to other questions as well. For example, motivated by \cref{lunch5}, one can attempt to formulate a universal property for the stack $(\mathbb{A}_{\mathbb{F}_p}^1)^{\mathrm{dR}}$ in the category of $\mathbb{F}_p$-algebra stacks. In \cref{table}, we explain a rough comparison between the stacky approach and the approach taken in our paper.

\subsection{Motivations and related work}
In this section we describe the motivations behind the constructions appearing in this paper and other related work. Our starting point was to approach \cref{inhosp3} which asks about deformations of a functor (instead of a single object) which we regard as somewhat difficult to approach. The strategy of the proof outlined above is vaguely inspired by some constructions from chromatic homotopy theory. Given a complex oriented multiplicative cohomology theory $E^*$, one can extract a formal group law from it by looking at $E^* (\mathbb{C}P^\infty)$ and using the multiplication $\mathbb{C}P^\infty \times \mathbb{C}P^\infty \to \mathbb{C}P^\infty.$ Further, given a formal group law, the Conner-Floyd construction \cite{CF66} defines a ``cohomology theory" associated to it. Motivated by this picture, one can ask the following naive question in our context.\vspace{2mm}

\noindent
\textbf{Question.} Is there a way to extract a ``group like object" from de Rham cohomology (or its deformations)? Further, is the theory of de Rham cohomology (and its deformations) determined by this ``group like object"?\vspace{2mm}

By the de Rham-crystalline comparison theorem \cite[Thm. V.2.3.2]{Ber74}, the theory of de Rham cohomology is essentially determined by the theory of divided power structures. This can be seen more concretely by using the work of Bhatt on derived de Rham cohomology \cite{Bha12}. Given a QRSP algebra $S$, by \cite[Prop. 8.12]{BMS19}, its derived de Rham cohomology $\mathrm{dR}(S)$ is naturally isomorphic to the divided power envelope $D_{S^\flat}(I)$ where $I:= \mathrm{Ker}(S^\flat \to S)$. Setting $S := \mathbb{F}_p[x^{1/p^\infty}]/x$ and considering $\mathrm{dR}(\mathbb{F}_p[x^{1/p^\infty}]/x)$, we get the ring of functions underlying $u^*W[F]$ from \cref{introex1}. Further, the Hopf stucture of $\mathbb{F}_p[x^{1/p^\infty}]/x$ provides a Hopf structure on $\mathrm{dR}(\mathbb{F}_p[x^{1/p^\infty}]/x)$ which is the same as the Hopf algebra underlying the ring of functions on $u^*W[F].$ This addresses the first half of our question above and extracts the ``group like object" $u^*W[F]$ from $\mathrm{dR}.$\vspace{2mm}

For the second half, one needs to build the de Rham cohomology functor from the object $u^*W[F].$ By the isomorphism $\mathrm{dR}(S) \simeq D_{S^\flat}(I)$ for a QRSP algebra $S$, it would be enough to build divided power envelopes from $u^*W[F].$ In \cite[Appendix 2]{BO78}, Berthelot-Ogus constructs the closely related divided power algebra $\Gamma_R(M)$ for any ring $R$ and an $R$-module $M$ by using a particular $R$-module called $\mathrm{exp}(R).$ We note that there is an isomorphism $\mathrm{exp}(R) \simeq W[F] (R),$ where the latter denotes the $R$-valued points of the group scheme $W[F].$ This suggests that in principle, it could be possible to build divided power envelopes from $u^*W[F].$ However, we need to equip the group scheme $u^*W[F]$ with more structure. This leads to the definition of a pointed $\mathbb{G}_a^{\mathrm{perf}}$-module, which is the framework for our ``group like object". In \cref{example3}, we see that using the unwinding functor, it is indeed possible to directly build divided power envelopes (and consequently derived de Rham cohomology) out of the pointed $\mathbb{G}_a^{\mathrm{perf}}$-module $u^*W[F].$ This addresses the second half of our question as well.\vspace{2mm}

Let us now mention some independent related work that appeared during the preparation of this paper. The connection between $u^*W[F]$ or $W[F]$ and de Rham cohomology also appears in the stacky approach to $p$-adic cohomology theories by Drinfeld \cite{Dri20}. The ``crystallization" of $\mathbb{A}^1 _{\mathbb{F}_p}$ is a stack that is obtained by taking the cone of the quasi-ideal $W[F]$ in $\mathbb{G}_a.$ The notion of a quasi-ideal also appeared in the work of Drinfeld and in general, a ring stack can be created out of a quasi-ideal by considering its cone. More details on these constructions can all be found in \cite{Dri20}. For us, a quasi-ideal is used as a special kind of a pointed $\mathbb{G}_a$ or a $\mathbb{G}_a^{\mathrm{perf}}$-module for which the unwinding functor is particularly well-behaved. In \cref{high11}, we show that the (opposite) category of quasi-ideals can be embedded in a certain naturally defined category. \vspace{2mm}

A connection between $W[F]$ and Hodge cohomology appears in the work of Moulinos, Robalo and To\"en on Hochschild homology \cite{MRT20}. In their context, Hodge cohomology appears as the associated graded object of the HKR filtration on Hochschild homology. The authors construct a filtered stack (over a $p$-adic base) which they call the filtered circle. The associated graded stack of the filtered circle is given by the classifying stack $B W[F].$ They show that Hochschild homology can be studied through this filtered circle where the filtration on the filtered circle induces 
the HKR filtration on Hochschild homology. Their work also gives a different way of thinking about the group scheme $W[F]$: the classifying stack $B W[F]$ is the affine stack corresponding to the cosimplicial ring given by the trivial square zero extension $\mathbb{F}_p \oplus \mathbb{F}_p[-1].$ The stack $B W[F]$ also appears in the work of To\"en in \cite{Toe20}, where it is used to define derived foliations on schemes.\vspace{2mm}

A universal property of the Hodge completed derived de Rham complex was recently obtained in \cite{Rak20} and motivated us to look for a universal property for $\mathrm{dR}$ from our perspective as in \cref{lunch5}.\vspace{2mm}

Finally, as previously noted, the analogue of \cref{inhosp3} for $A = \mathbb{Z}_p$ (instead of an Artinian local ring) was already known due to the work of Bhatt, Lurie and Mathew \cite[Thm. 10.1.2]{BLM20}, which they use to give a new proof of the de Rham Witt to crystalline comparison theorem of Illusie \cite[Thm. II.1.4]{Ill79}. \cref{inhosp3} in our paper also allows torsion base rings $A,$ and one can deduce \cite[Thm. 10.1.2]{BLM20} from it by a limit argument (see the discussion after \cref{hun}). A variant of questions regarding endomorphisms of the de Rham cohomology functor appears in the work of Li and Liu \cite{LL}.
\vspace{2mm}

The idea of controlling de Rham cohomology theory by a single object (with appropriate structure) originating from this paper has been pursued further in author's subsequent  joint work with Li in \cite{LM21}, where all the endomorphisms of de Rham cohomology theory as a functor has been classified in very general situations \cite[Thm 1.1]{LM21}. As a consequence of this classification, one can deduce Drinfeld's refinement of the Deligne--Illusie decomposition \cite[Thm 1.6]{LM21}.

\subsection{Acknowledgements}I am extremely grateful to my advisor Bhargav Bhatt for proposing the main question regarding deformations of de Rham cohomology studied in this paper as well as for his patience, interest and many suggestions throughout the project. Numerous discussions with him have improved many aspects of this paper. Much of this paper relies on his earlier work on $p$-adic derived de Rham cohomology; he also suggested the use of ``descendability" which is an important technical ingredient in some of the proofs. I am thankful to Attilio Castano, Andy Jiang and Emanuel Reinecke for many helpful discussions during the preparation of this paper. I would also like to thank Benjamin Antieau, Vladimir Drinfeld, Haoyang Guo, Luc Illusie, Arthur-C\'esar Le Bras, Shizhang Li, Akhil Mathew, Arpon Raksit and Bertrand Toën for comments on a draft version of the paper. Special thanks to the referee for many careful comments and valuable suggestions. 
\vspace{2mm}

This work is part of my PhD thesis at the University of Michigan and I am thankful to the mathematics department for its support. During the preparation of this paper, I was also partially supported by the Rackham International Students Fellowship, NSF grants DMS \#1801689 and \#1952399 through Bhargav Bhatt.

\newpage

\section{Modules over ring schemes}\label{section2}

In this section, we begin with the definition of a $\mathbb{G}_a$ and $\mathbb{G}_a^{\mathrm{perf}}$-module leading up to the notion of a \textit{pointed} $\mathbb{G}_a$ or $\mathbb{G}_a^{\mathrm{perf}}$-module. Our final goal is to study the deformation of the pointed $\mathbb{G}_a$-module $W[F],$ which is obtained from the kernel of Frobenius on Witt vectors and its closely related variant $u^*W[F],$ which is a pointed $\mathbb{G}_a^{\mathrm{perf}}$-module. This will in part be achieved via attaching universal properties to the objects $W[F]$ and $u^*W[F]$ as objects in certain categories. The construction of such categories leads to several refinements of the category of pointed $\mathbb{G}_a^{\mathrm{perf}}$-modules which we call pointed $\mathbb{G}_a^{\mathrm{perf}}$-modules of \textit{fractional rank 1} (\cref{satu3}), \textit{full of fractional rank $1$} (\cref{satu100}) and \textit{pure of fractional rank $1$} (\cref{pure}).

\begin{notation}
We let $\mathbb{N}$ denote the monoid of nonnegative integers. The set of positive integers will be denoted by $\mathbb{N}_{>0}.$ For a fixed prime $p,$ we let $\mathbb{N}[\frac 1 p]$ denote the monoid of nonnegative elements in $\mathbb{Z}[\frac 1 p] \subset \mathbb Q.$ The category of $A$-algebras will be denoted as $\mathrm{Alg}_A.$ Its opposite category, i.e., the category of affine scheme will be denoted by $\mathrm{Aff}_A.$ All schemes considered are affine schemes unless otherwise mentioned. The group schemes we consider are always assumed to be commutative. The notion of a $\mathbb{G}_a$-module is valid over any base ring $A.$ However, the notion of a $\mathbb{G}_a^{\mathrm{perf}}$-module will require us to fix a prime $p.$ In fact, $\mathbb{G}_a^{\mathrm{perf}}$-modules will only be defined over a base ring where $p$ is nilpotent. \end{notation}

\subsection{$\mathbb{G}_a$-modules}

Let $\mathcal{C}$ be any category which admits finite products. Many of the familiar concepts from algebra can be made sense of in the category $\mathcal{C}.$ For example, one may talk about any monoid $M$ acting on an object $c$ of $\mathcal{C},$ which is encoded by the data of a monoid homomorphism $M \to \w{Hom}_{\mathcal{C}}(c,c).$ One can also define the notion of a group object $\mathcal{G}$ of $\mathcal{C}$ and talk about $\mathcal{G}$ acting on an object of $\mathcal{C}.$ Further, one can talk about the notion of a ring object of $\mathcal{C}.$ If $\mathcal{R}$ is a ring object of $\mathcal{C},$ then one can define a notion of $\mathcal{R}$-module objects too. Many of the definitions we introduce in this section can be understood and defined in this generality. We will only spell out these definitions in the more concrete cases as required for our paper. However, our definitions will be based on the Yoneda lemma and they apply to the general situation of any category $\mathcal{C}$ with finite products.

\begin{definition}\label{def1} Let $R$ be an arbitrary ring and $\mathrm{Mod}_R$ be the category of $R$-modules. Let $F: \mathrm{Aff}_A^{\w{op}} \to \mathrm{Mod}_R$ be a functor. We will say that $F$ defines an $R$-module scheme over $A,$ if the set-valued presheaf underlying $F$ is representable by an affine scheme over $A.$ \vspace{2mm}

From a categorical perspective, one can also say that an $R$-module scheme is simply an $R$-module object in the category of affine schemes over $A.$
\end{definition}

\begin{remark}
Note that if $X$ is an $R$-module scheme, then by definition, $X$ is equipped with the structure of a commutative group scheme. Additionally, for every $r \in R,$ there is a map $m_r: X \to X$ of group schemes, which is the analogue of ``multiplication by $r$ map" in the case of usual rings and modules. These maps are required to satisfy certain conditions analogous to the usual ones in algebra that we do not spell out here. All of these data and conditions are captured by the functorial definition provided in \cref{def1}.
\end{remark}{}

\begin{remark}We point out that the ring $R$ in \cref{def1} is arbitrary and not required to be an $A$-algebra.
\end{remark}

\begin{example}Taking $X = \mathrm{Spec}\, A[x]$, we see that $X$ can be equipped with the structure of an $A$-module scheme.
\end{example}

\begin{remark} One can also similarly define a notion of $R$-module schemes that are not necessarily affine. However, such non affine examples will not be necessary for us in this paper; so in \cref{def1}, we restrict our definition to the affine case.
\end{remark}{}

\begin{example}\label{ex1} We note that $\mathrm{Spec}\, A[x]$ can be naturally equipped with the structure of a ring scheme over $A$. We will denote this ring scheme by $\mathbb{G}_a.$ It represents the functor that sends an affine scheme to its ring of global sections.
\end{example}{}

\begin{definition}[$\mathbb{G}_a$-module]\label{g_amod} Let us consider the category $\w{Aff}_A$ and the presheaf of rings on $\w{Aff}_A$ represented by the ring scheme $\mathbb{G}_a$. Let $F$ be a presheaf of modules over the presheaf of rings represented by $\mathbb{G}_a.$ We will say that $F$ is a $\mathbb{G}_a$-module over $A$, if the set-valued presheaf underlying $F$ is representable by an affine scheme over $A.$ Morphisms of $\mathbb{G}_a$-modules are defined as morphisms of presheaves of modules over the presheaf of rings represented by $\mathbb{G}_a.$
\vspace{2mm}

From a categorical perspective, one can simply say that a $\mathbb{G}_a$-module is a $\mathbb{G}_a$-module object in the category of affine schemes over $A.$
\end{definition}{}

\begin{remark}\label{hungry991}
Note that by definition, a $\mathbb{G}_a$-module $X$ has the structure of a commutative group scheme. Further, there is also the $\mathbb{G}_a$-action map $\mathrm{act}: X \times \mathbb{G}_a \to X.$ These data are subjected to the usual compatibilities which are all abstractly captured in \cref{g_amod}. For example, the $\mathbb{G}_a$-action must be compatible with the group operation $m_X: X \times X \to X.$ Further, the $\mathbb{G}_a$-action must also respect the multiplication map $m_{\mathbb{G}_a}: \mathbb{G}_a \times \mathbb{G}_a \to \mathbb{G}_a$ coming from the additive group scheme structure of $\mathbb{G}_a.$ We spell out the latter compatibility explicitly, which amounts to the following commutative diagram. Below, we let $\Delta: X \to X \times X$ denote the diagonal map.
\begin{center}
    \begin{equation}\label{hungry990}
    \begin{tikzcd}
\mathbb G_a \times X \times \mathbb G_a \arrow[rr, "\simeq "] \arrow[d, "\text{id}\times \Delta \times\text{id}"] &  & \mathbb G_a \times \mathbb G_a \times X \arrow[rr, "m_{\mathbb G_a} \times \text{id}"] &  & \mathbb G_a \times X \arrow[d, "\text{act}"'] \\
\mathbb G_a \times X \times X \times \mathbb G_a \arrow[rr, "\text{act}\times \text{act}"]                        &  & X \times X \arrow[rr, "m_X"]                                                           &  & X                                            
\end{tikzcd}
\end{equation}
\end{center}
\end{remark}{}

\begin{example}$\mathbb{G}_a$ itself can be equipped with the structure of a $\mathbb{G}_a$-module. If $A$ has char $p$, then $\mathrm{Spec}\, A[x]/ x^p$ can be equipped with the structure of a $\mathbb{G}_a$-module.
\end{example}{}

\begin{remark}\label{noise}We note that the affine scheme underlying a $\mathbb{G}_a$-module in particular has the action of $(\mathbb{G}_a, \cdot~)$, where the latter is considered to be a monoid scheme under multiplication; thus global sections on it gives a nonnegatively graded Hopf algebra. In other words, every $\mathbb{G}_a$-module has the structure of a nonnegatively graded group scheme. If the underlying affine scheme of a $\mathbb{G}_a$-module is written as $\w{Spec}\, B,$ then we have a direct sum decomposition $B = \bigoplus_{i \in \mathbb{N}} B_i$ coming from the grading, where $B_i$ denotes the summand of degree $i.$ We refer to \cite{MM65} for a study of graded Hopf algebras.
\end{remark}

\begin{remark}The notion of a $\mathbb{G}_a$-module extends to any scheme which is not \textit{a priori} assumed to be affine. However, we note that being a $\mathbb{G}_a$-module imposes strong topological restrictions on the underlying scheme. In fact, any scheme which can be equipped with the structure of a $\mathbb{G}_a$-module over a field is necessarily affine (see \cref{stillediting78}). Thus, there is not much of a loss of generality by defining the notion of $\mathbb{G}_a$-modules only in the affine case, as we do in our paper. We thank Drinfeld for pointing this out. Below, we prove a slightly more general proposition.
\end{remark}

\begin{proposition}\label{A^1hom}Let $G$ be a scheme over $k$ equipped with a $k$-rational point given by $c: \mathrm{Spec}\, k \to G$. Suppose that there is a map $F: G \times \mathbb{A}^{1}_k \to G$ such that the restriction map $G \times \left \{ 1 \right \} \to G$ is identity and the restriction map $G \times \left \{ 0 \right \} \to G$ is the composition $G \to \mathrm{Spec}\, k \to G$ (where the latter map comes from the chosen $k$-rational point). In this set up, if $G$ can be equipped with the structure of some group scheme, then $G$ must be an affine scheme.
\end{proposition}

\begin{proof} To prove this assertion, we can assume that $k$ is algebraically closed. By the assumptions on the map $F: G \times \mathbb{A}^1_k \to G$, it follows that $G$ must be connected. By a modification of Chevalley's theorem due to Perrin \cite[Cor. 4.2.9]{Per76}, there exists an exact sequence $0 \to H \to G \to A \to 0$ of group schemes in the fpqc topology, where $H$ is an affine group scheme and $A$ is an abelian variety. By hypothesis, we have a distinguished $k$-rational point $c$ of $G,$ whose image in $A$ will be denoted by $c' \in A(k).$  Our claim would follow if we prove that $A(k)= 0,$ where $A(k)$ denotes the $k$-valued points of $A.$ We let $t \in A(k).$ Since $G \to A$ is an fpqc surjection, we can find an algebraically closed field $K,$ which contains $k$ and such that there exists $t' \in G(K)$ which is mapped to the image of $t$ in $A(K).$ Using $t'$ and the map $F: G \times \mathbb{A}^1_{k} \to G$ supplied by our assumptions, we obtain a map $\mathbb{A}^1_K \to A_K$ such that $\left \{1 \right \} \in \mathbb{A}^1_K$ is mapped to the image of $t$ in $A(K)$ and $\left \{0 \right \} \in \mathbb{A}^1_K$ maps to image of $c'$ in $A(K).$ Any such map extends to a map $\mathbb{P}^1_K \to A_K$ and since $A_K$ is an abelian variety, any such map has to factor through the Jacobian of $\mathbb{P}^1_K$, which is a point. Thus the map $\mathbb{A}^1_K \to A_K$ is constant. By fpqc sheaf property, the map $A(k)\to A(K)$ is injective. This implies that $t=c' \in A(k).$ Since $t$ was arbitrary, it follows that $A(k)$ consists of a single point and thus $A(k) =0,$ which gives the claim.
\end{proof}

\begin{remark}\label{stillediting78}
We clarify that in \cref{A^1hom}, we do not assume that the zero section of the group scheme structure on $G$ is the same as the $k$-rational point $c.$ In the language of $\mathbb{A}^1$-homotopy theory, the hypothesis in \cref{A^1hom} means that the structure map $G \to \w{Spec}\, k$ is a strict $\mathbb{A}^1$-homotopy equivalence \cite[2.3]{MV99}. Further, we point out that the hypothesis in \cref{A^1hom} is satisfied for any $\mathbb{G}_a$-module $X,$ by considering the zero section $\mathrm{Spec}\, k \to X$ itself to be the rational point $c$ and by taking $F$ to be the $\mathbb{G}_a$-module action map $X \times \mathbb{G}_a \to X.$ This gives the conclusion that any $\mathbb{G}_a$-module over a field is affine.
\end{remark}{}

\begin{proposition}\label{omgwat}The forgetful functor from the category of $\mathbb{G}_a$-modules to the category of graded group schemes is fully faithful.
\end{proposition}{}

\begin{proof}Let $\mathrm{Spec}\,U $ and $\mathrm{Spec}\,V$ be two $\mathbb{G}_a$ modules and let $\mathrm{Spec}\, U \to \mathrm{Spec} \, V$ be a map of $\mathbb{G}_a$-modules. This is the data of a map $V \to U$ that is a Hopf algebra map and is equivariant with the $A[x]$-coaction, i.e., commutes with the $A[x]$-coaction maps $U \to U[x] $ and $V \to V[x].$ However, the latter compatiblity can be checked after composing along the injective maps $U[x] \to U[x^{\pm 1}]$ and $V[x] \to V[x^{\pm 1}]$ and thus it is enough to provide a map $V \to U$ which is compatible with the $A[x^{\pm 1}]$-coaction, i.e., a graded Hopf algebra map  $V \to U.$ \end{proof}{}

\begin{remark}
Let us give a much more abstract generalization of \cref{omgwat} inspired from a comment by the referee. If $\mathcal{X}$ is a topos and $\mathcal{R}$ is a ring object of $\mathcal{X},$ one can consider the category of $\mathcal{R}$-module objects. The units in $\mathcal{R}$ can be viewed as a group object of $\mathcal{X},$ which we will denote as $\mathcal{R}^\times.$ Then there is a forgetful functor $\varphi$ from the category of $\mathcal{R}$-modules to the category of $\mathcal{R}^\times$-representations. Note that the category of $\mathcal{R}^\times$-representations can also be viewed as the category of $\mathbb{Z}[\mathcal{R}^\times]$-modules of $\mathcal{X}.$ The forgetful functor $\varphi$ can simply be identified with the restriction of scalars along the natural map $\mathbb{Z}[\mathcal{R}^\times] \to \mathcal{R}.$ Therefore, by the adjunction between restriction and extension of scalars, $\varphi$ is fully faithful if and only if the counit is naturally isomorphic to identity. The latter is equivalent to the natural map $\mathcal{R} \otimes_{\mathbb{Z}[\mathcal{R}^\times]} \mathcal{R} \to \mathcal{R}$ being an isomorphism. Now, specializing to the case when $\mathcal{X}$ is the fpqc topos, $\mathcal{R}= \mathbb{G}_a$ and $\mathcal{R}^\times = \mathbb{G}_m,$ the condition that the natural map $\mathcal{R} \otimes_{\mathbb{Z}[\mathcal{R}^\times]} \mathcal{R} \to \mathcal{R}$ is an isomorphism is implied by the fact that $\mathbb{Z}[\mathbb{G}_m] \to \mathbb{G}_a$ is a surjection of sheaves. The last claim can be deduced by the observation that for any ring $S$ and an element $f \in S,$ $f$ is a sum of at most two units Zariski locally on $\w{Spec}\,S.$ Indeed, $\w{Spec}\, S_{f}$ and $\w{Spec}\, S_{1-f}$ cover $\w{Spec}\, S$; on $\w{Spec}\, S_{f},$ $f$ is already a unit and on $\w{Spec}\, S_{1-f}$, we have $f=1 + (f-1).$ This gives \cref{omgwat}.
\end{remark}{}

\begin{remark}We note that a graded group scheme being a $\mathbb{G}_a$-module is no extra data, but a condition. This condition is not always satisfied. For example, the Witt group scheme $W$ which represents the functor $\mathrm{Aff}_A \to \mathrm{Sets}$ given by $\mathrm{Spec}\,B \mapsto W(B)$ where $W(B)$ is the ring of $p$-typical Witt vectors of $B$ is a graded group scheme but not a $\mathbb{G}_a$-module. Indeed, for any $b \in B,$ multiplication by the Teichmüller lift $[b]= (b, 0, \ldots, 0, \ldots)$ of $b$ equips $W$ with the structure of a nonnegatively graded group scheme (\textit{cf}. \cref{noise}). However, for $\underline{t} \in W(B)$ and $b, b' \in B,$ in general, $[b+ b'] \cdot t \neq  [b] \cdot t + [b'] \cdot t. $ Therefore, the graded group scheme $W$ does not satisfy the condition of being a $\mathbb{G}_a$-module. \textit{cf.}~\cref{hungry991},~\cref{hungry990}.
\end{remark}

\begin{proposition}\label{conn}Let $\mathrm{Spec}\, B$ be a $\mathbb{G}_a$-module. Then as a graded algebra, the degree zero piece of $B$ is naturally isomorphic to $A$ as an $A$-module. In other words, as a graded Hopf algebra, $B$ is connected.
\end{proposition}{}

\begin{proof}First, by using the zero section of a group scheme, we note that the $A$-algebra structure map $A \to B$ is injective. The $\mathbb{G}_a$-module structure map is given by a map $B \to B[x].$ Killing $x$ produces a map $B \to B$ whose kernel is $I_{>0} := \bigoplus_{i>0} B_i$ where $B = \bigoplus_{i \ge 0} B_i.$ Further, using the fact that $\mathrm{Spec}\,B$ is a $\mathbb{G}_a$-module, we note that the map $B \to B$ obtained this way also has the property that it factors through $B \to A$, which is the zero map of the comultiplication. Since the map $A \to B$ is injective, this provides an $A$-algebra map $B \to A$ whose kernel is $I_{>0}.$ Thus $B_0$ is naturally isomorphic to $A$, as desired. 
\end{proof}{}

\begin{remark}\label{lastonetoday176}
Let $\w{Spec}\, B$ be a $\mathbb{G}_a$-module as above. As noted in \cref{noise}, $B$ has the structure of a nonnegatively graded Hopf algebra. Let $c: B \to B \otimes B$ be the comultiplication map. The proof of \cref{conn} shows that the (surjective) map $z: B \to A$ induced by the zero section has kernel equal to $I_{>0};$ this gives a natural isomorphism $A \simeq B_0.$ Therefore, for any $b \in B_i,$ such that $i>0,$ we have $c(b) = p\otimes 1 + 1 \otimes q + c_{+}(b),$ where $p,\,q \in B_i$ and $c_{+}(b) \in I_{>0} \otimes_A I_{>0}.$ We claim that $p=q=b.$ To see this, we note that the composite map $B \xrightarrow{c} B \otimes_A B \xrightarrow{\w{id} \otimes z} B \otimes_A A \simeq B $ is the identity on $B.$ Recalling the fact that the kernel of $z$ is the ideal $I_{>0},$ it follows that $p=b.$ Similarly one obtains $q=b.$ To summarize, we see that if $b \in B$ is a homogeneous element of degree $>0,$ then $c(b) = b\otimes 1 + 1 \otimes b + c_+ (b),$ where $c_+(b) \in I_{>0} \otimes_A I_{>0}.$ This observation will be used in the proof of \cref{classifyg_amodchar0}. As a special case of this observation, we note that since $\deg c_+(b) = \deg b = i$, the element $c_+(b)$ is necessarily zero if $i=1.$
\end{remark}{}

\begin{remark}\label{detai}
By unwrapping \cref{g_amod}, one sees that a $\mathbb{G}_a$-module over an arbitrary base ring $A$ is equivalent to the following: \vspace{2mm}

\noindent
1. For every $A$-algebra $B,$ a $B$-module scheme $\w{Spec}\, M_B$ over $B.$\vspace{1mm}

\noindent
2.~For every map $B \to B'$ of $A$-algebras, an isomorphism $\w{res}^{B'}_{B}: M_B \otimes_{B} B' \simeq M_{B'}.$ Further, in this isomorphism, the $B$-action on the left hand side is compatible with the restriction of the $B'$-action on the right hand side along the map $B \to B'.$ The latter is a condition and not extra data.\vspace{2mm}

Further, a morphism of $\mathbb{G}_a$-modules under this equivalence translates to the following:\vspace{2mm}

\noindent
1. For every $A$-algebra $B,$ a morphism $\Phi_B$ of $B$-module schemes over $B.$ \vspace{1mm}

\noindent
2. For every map $B \to B'$ of $A$-algebras, the maps $\Phi_B'$ and $\Phi_B$ are compatible with $\w{res}^{B'}_{B}.$
\end{remark}{}

\begin{definition}[Pointed $\mathbb{G}_a$-module]\label{def2}A $\mathbb{G}_a$-module scheme $X$ along with the data of a map $X \to \mathbb{G}_a$ of $\mathbb{G}_a$-modules will be called a \textit{pointed $\mathbb{G}_a$-module $X$.} We will follow the convention that the data of the map $X \to \mathbb{G}_a$ will be simply called a point. Maps between pointed $\mathbb{G}_a$-modules are maps of $\mathbb{G}_a$-modules that commute with the points. We denote the category of such objects by $\mathbb{G}_a \mathrm{-Mod}_*.$
\end{definition}{}

\begin{remark}[\textit{cf}.~\cref{high}]\label{deg1}If $X = \mathrm{Spec}\, B$ is a $\mathbb{G}_a$-module, then we note that giving a map $X \to \mathbb{G}_a$ of $\mathbb G_a$-modules is equivalent to choosing an element of degree $1$ in the graded algebra $B.$ This follows from the fact that if $x$ is an element of degree $1$ in $B$, then the comultiplication map $B \to B \otimes_A B$ sends $x \to x \otimes 1 + 1 \otimes x.$ Thus a pointed $\mathbb{G}_a$-module is the data of a $\mathbb{G}_a$-module $\mathrm{Spec}\, B$ and an element $x \in B_1$ (where $B_1$ is the degree $1$ piece of $B$ ). The choice of this element $x \in B_1$ is the reason we use the word ``point'' to talk about the map $X \to \mathbb{G}_a$; it is motivated by the terminology in topology where a space $Y$ and a choice of an element $y \in Y$ is called a pointed space. Using functor of points, in our case, this can also be interpreted as an $X$-valued point of $\mathbb{G}_a.$
\end{remark}{}

\begin{remark} $\mathbb{G}_a$ can be naturally equipped with the structure of a pointed $\mathbb{G}_a$-module using the identity map $\mathbb{G}_a \to \mathbb{G}_a.$ In fact, $\mathbb{G}_a$ is the final object of $\mathbb{G}_a\mathrm{-Mod}_*.$ The initial object of $\mathbb{G}_a\mathrm{-Mod}_*$ is given by the zero section $\mathrm{Spec}\, A \to \mathbb{G}_a$. \cref{dri} and \cref{rem7} records more examples of pointed $\mathbb G_a$-modules.
\end{remark}{}

\begin{example}
Let $A$ be the base ring fixed as before. If $M$ is an $A$-module, then $\w{Spec}\,(\w{Sym}_A (M))$ naturally has the structure of a $\mathbb{G}_a$-module over $A$ (\textit{cf.}~\cref{7up3}). Below, we will show that if $A$ is a $\mathbb{Q}$-algebra, then every $\mathbb{G}_a$-module is of the form described above. We thank the referee for bringing this to our attention. 
\end{example}

\begin{proposition}\label{classifyg_amodchar0}
 Let $A$ be a $\mathbb{Q}$-algebra and $X$ be a $\mathbb{G}_a$-module over $A.$ Then, there is an $A$-module $M$ such that we have an isomorphism $X \simeq \mathrm{Spec}\,(\mathrm{Sym}_A (M))$ of $\mathbb{G}_a$-modules.
\end{proposition}{}

\begin{proof}
Let $B := \Gamma (X, \mathcal{O}_X)$. As noted in \cref{noise}, we have natural a direct sum decomposition $B= \bigoplus_{i \in \mathbb{N}} B_i.$ Since $X$ is a $\mathbb{G}_a$-module, $B$ has the structure of a connected (\cref{conn}) nonnegatively graded Hopf algebra. We will start by recalling that (see~\cref{lastonetoday176})
 if $b \in B$ is a homogeneous element such that $\deg b >0,$ then the comultiplication map $B \to B \otimes_A B$ sends \begin{equation}\label{wolverine}
     b \mapsto b \otimes 1 + 1 \otimes b + \sum_{u} b'_u \otimes b''_u,
 \end{equation}where $b'_u,\,b''_u \in B$ are homogeneous elements of $B$ such that $\deg b'_u,\,\deg b''_u >0. $ In particular, \cref{wolverine} implies that if $ b \in B_1,$ then the comultiplication map $B \to B \otimes_A B$ sends $b \to b \otimes 1 + 1 \otimes b.$ This implies that the natural map 
$$F: \mathrm{Sym}_A (B_1) \to B$$ is a map of graded Hopf algebras. This constructs a map $X \to \mathrm{Spec}\, (\mathrm{Sym}_A (B_1))$ of graded group schemes, which is automatically a map of $\mathbb{G}_a$-modules by \cref{omgwat}. To prove the proposition, it would be enough to show that the map $F: \mathrm{Sym}_A (B_1) \to B$ is an isomorphism of $A$-algebras.\vspace{2mm}

First we show that $F$ is surjective. For this, we use the commutative diagram \cref{hungry990} that every $\mathbb{G}_a$-module $X$ is required to satisfy. Applying the global section functor to \cref{hungry990}, we obtain the following commutative diagram of $A$-algebras.

\begin{center}
\begin{tikzcd}
{A[t] \otimes_A B \otimes_A A[t]}                     &  & {B \otimes_A (A[t] \otimes_A A[t])} \arrow[ll, "\simeq"'] & {B \otimes_A A[t]} \arrow[l] \\
{A[t]\otimes_A B \otimes_A B\otimes_A A[t]} \arrow[u] &  & B \otimes_A B \arrow[ll]                                  & B \arrow[u] \arrow[l]       
\end{tikzcd}
\end{center}{}
We will write $A [t] \otimes_A B \otimes_A A[t] \simeq B[t_1, t_2],$ where the isomorphism sends $(t \otimes 1 \otimes 1) \mapsto t_1$ and $(1 \otimes 1 \otimes t) \mapsto t_2$. We note that in order to show that $F$ is surjective, by construction, it would be enough to show that $B$ is generated as an $A$-algebra by its degree $1$ elements. \vspace{2mm}

To this end, let $b \in B$ be a homogeneous element of degree $i>1.$ By \cref{wolverine}, the composite map $B \to B \otimes_A B \to B[t_1, t_2]$ from the above diagram sends $$b \mapsto t_{1}^i b + t_2^i b + \sum_{u} t_1^{(\deg b'_u) } t_2^{(\deg b''_u)}b'_u \cdot b''_u,$$ where $b'_u, b''_u \in B$ are as described in \cref{wolverine}. On the other hand, the composite map $B \to B \otimes_A A[t] \to B[t_1, t_2]$ sends $b \mapsto b (t_1 + t_2)^i. $ By the commutativity of the diagram, we have the relation 
\begin{equation}\label{hungry997}
    b (t_1 + t_2)^i = t_{1}^i b + t_2^i b + \sum_{u} t_1^{(\deg b'_u) } t_2^{(\deg b''_u)}b'_u \cdot b''_u
\end{equation}{}in the ring $B[t_1, t_2].$ Since $\deg b'_u,\,\deg b''_u >0$ and $\deg b =i,$ it follows that $\deg b'_u,\,\deg b''_u <i.$ Since the base ring $A$ is assumed to be a $\mathbb{Q}$-algebra, by comparing coefficients of $t_1^{i-1}t_2$ from both sides in the equation \cref{hungry997}, we see that $b = \sum_{v} c'_v c''_v$ for some homogeneous elements $c'_v, \, c''_v \in B$ such that $\deg c'_v, \, \deg c''_v < i.$ Since $b \in B$ was an arbitrary element of degree $>1,$ inductively we obtain that $B$ is generated as an algebra by its degree $1$ elements, which shows that $F$ is surjective, as desired.\vspace{2mm}

Now we show that $F$ is injective. By construction, $F$ is a graded map and it induces an isomorphism in degree $1.$ For an integer $r,$ let $F_r$ denote the induced map on the summands of degree $r.$ Let us assume for the sake of contradiction that $F$ is not injective. Let $n \ge 2$ be the minimal integer such that $F_n$ is not injective. Since we already know that $F$ is surjective, it follows that $F_r$ is an isomorphism for $ r < n.$ We note the following commutative diagram in the category of $A$-modules
\begin{center}
\begin{tikzcd}
\mathrm{Sym}_A^n (B_1) \arrow[d, "F_n"'] \arrow[rr] &  & \mathrm{Sym}_A(B_1) \otimes_A \mathrm{Sym}_A(B_1) \arrow[rr] \arrow[d, "F \otimes F"'] &  & \mathrm {Sym}_A^{n-1}(B_1) \otimes_A \mathrm{Sym}_A^1(B_1) \arrow[d, "F_{n-1} \otimes_A F_1"'] \\
B_n \arrow[rr]                                  &  & B \otimes_A B \arrow[rr]                                                         &  & B_{n-1} \otimes_A B_1 .              \end{tikzcd} 
\end{center}In the above, the horizontal maps are obtained from the graded Hopf algebra structure on $\mathrm{Sym}_A (B_1)$ and $B.$ More explicitly, the left horizontal maps are induced by restricting the comultiplication to the summand of degree $n$ and the right horizontal maps are the projection maps arising from the grading. We note that the upper horizontal composite map $\mathrm{Sym}^n_A (B_1) \to \mathrm{Sym}^{n-1}_A (B_1) \otimes_A \mathrm{Sym}^1_A (B_1)$ is injective. To see the latter claim, one notes that the composition $\mathrm{Sym}^n_A (B_1) \to \mathrm{Sym}^{n-1}_A (B_1) \otimes_A \mathrm{Sym}^1_A (B_1) \xrightarrow[]{\mathrm{mult}} \mathrm{Sym}^n_A(B_1)$ is multiplication by the integer $n$, which is an isomorphism since $A$ is a $\mathbb Q$-algebra; here the map $\mathrm{mult} \colon \mathrm{Sym}^{n-1}_A (B_1) \otimes_A \mathrm{Sym}^1_A (B_1) \to \mathrm{Sym}^n_A(B_1)$ is induced by the graded $A$-algebra structure on $\mathrm{Sym}_A(B_1).$ This implies that the composite map $$\mathrm{Sym}^n_A (B_1) \to \mathrm{Sym}^{n-1}_A (B_1) \otimes_A \mathrm{Sym}^1_A (B_1) \xrightarrow[]{F_{n-1} \otimes_A F_1} B_{n-1} \otimes_A B_1$$ is injective since $F_{n-1} \otimes_A F_1$ is an isomorphism by the induction hypothesis. The above commutative diagram now implies that $F_n$ must be injective, which finishes the proof.
\end{proof}{}

\subsection{$\mathbb{G}_a^{\mathrm{perf}}$-modules}\label{sec2.2}
Below we define the notion of a $\mathbb{G}_a^{\mathrm{perf}}$-module which will be defined over a fixed ring $A$ such that $p^n = 0$ in $A$ for some $n.$

\begin{proposition}\label{tilt}
The functor $(\,\cdot\,)^\flat: \mathrm{Aff}_A^{\mathrm{op}} \to \mathrm{Sets}$ given by $\mathrm{Spec}\, B \to B^\flat$, where $B^\flat := \lim_{x \to x^p} B$ is naturally valued in rings.
\end{proposition}

\begin{proof}This follows from the natural bijection $B^\flat \simeq \lim_{x \to x^p} B/p$, which holds since $B$, being an $A$-algebra, is $p$-adically complete.\end{proof}{}

\begin{definition}The functor $(\,\cdot\,)^\flat$ from \cref{tilt} is represented by $\mathrm{Spec}\,A[x^{1/p^\infty}]$, which can be naturally viewed as a ring scheme and will be denoted as $\mathbb{G}_a^{\mathrm{perf}}$ when equipped with this ring scheme structure.
\end{definition}{}

\begin{remark}\label{rmk}When $A$ is an $\mathbb{F}_p$-algebra, the comultiplication of the Hopf algebra underlying $A[x^{1/p^\infty}]$ can be described easily: it is given by the map $A[x^{1/p^\infty}] \to A[x^{1/p^\infty}] \otimes_A A[x^{1/p^\infty}]$ given by $x^{1/p^n} \to x^{1/p^n }\otimes 1 + 1 \otimes x^{1/p^n}$ for all $n.$ However, in general the comultiplication is less simple to write down and we need to trace through the bijection $B^\flat \simeq \lim_{x \to x^p} B/p.$ For example, when $A$ is a $\mathbb{Z}/p^2 \mathbb{Z}$ algebra, the comultiplication is given by $A[x^{1/p^\infty}] \to A[x^{1/p^\infty}] \otimes_A A[x^{1/p^\infty}]$ which sends $x^{1/p^n} \to x^{1/p^n}\otimes 1 + 1 \otimes x^{1/p^n} + \sum_{0<i < p} \binom{p}{i} x^{i/p^{n+1}} \otimes x^{p-i/p^{n+1}}.$
\end{remark}

\begin{remark}\label{ring} When $A$ is an $\mathbb{F}_p$-algebra, the $A$-algebra map $A[x] \to A[x^{1/p^\infty}]$ gives us a morphism of ring schemes $\mathbb{G}_a^{\mathrm{perf}} \to \mathbb{G}_a.$ At the level of functor of points, this morphism is induced by the natural map $B^\flat \to B,$ which is a ring homomorphism when $B$ is an $\mathbb{F}_p$-algebra. However, using \cref{rmk} we can see that if $p \ne 0$ in $A,$ the natural map $A[x] \to A[x^{1/p^\infty}]$ is not a map of Hopf algebras and hence does not give a morphism of ring schemes.
\end{remark}{}

\begin{definition}[$\mathbb{G}_a^{\mathrm{perf}}$-module] Let us consider the category $\w{Aff}_A$ and the presheaf of rings on $\w{Aff}_A$ represented by the ring scheme $\mathbb{G}_a^{\w{perf}}$. Let $F$ be a presheaf of modules over the presheaf of rings represented by $\mathbb{G}_a^{\w{perf}}.$ We will say that $F$ is a $\mathbb{G}_a^{\w{perf}}$-module over $A$, if the set-valued presheaf underlying $F$ is representable by an affine scheme over $A.$ Morphisms of $\mathbb{G}_a^{\w{perf}}$-modules are defined as morphisms of presheaves of modules over the presheaf of rings represented by $\mathbb{G}_a^{\w{perf}}.$
\vspace{2mm}

From a categorical perspective, one can simply say that a $\mathbb{G}_a^{\w{perf}}$-module is a $\mathbb{G}_a^{\w{perf}}$-module object in the category of affine schemes over $A.$
\end{definition}{}

\begin{remark}
Note that by definition, a $\mathbb{G}_a^{\w{perf}}$-module scheme $X$ has the structure of a commutative group scheme. Further, there is also the $\mathbb{G}_a^{\w{perf}}$-action map $X \times \mathbb{G}_a^{\w{perf}} \to X.$ To avoid confusion, we clarify that a $\mathbb{G}_a^{\w{perf}}$-module is not the same as a $\mathbb{G}_a^{\w{perf}}$-module object in the category of perfect rings (when $A= \mathbb{F}_p$). Further, the affine scheme underlying a $\mathbb{G}_a^{\w{perf}}$-module need not be a perfect scheme either. 
\end{remark}{}

\begin{example}\label{ex3}When $A$ has char. $p$, the scheme $\mathrm{Spec}\,A[x^{1/p^\infty}]/x$ can be equipped with the structure of a $\mathbb{G}_a^{\mathrm{perf}}$-module which we will denote as $\alpha^{\natural}.$ If $p \ne 0$ in $A$, the Hopf structure of $A[x^{1/p^\infty}]$ as described in \cref{rmk} does not induce a Hopf structure in the quotient $A[x^{1/p^\infty}]/x.$
\end{example}{}

\begin{remark}\label{gradingdet}A $\mathbb{G}_a^{\mathrm{perf}}$-module $\mathrm{Spec}\, B$ naturally provides us a $\mathbb{G}_m^{\mathrm{perf}}:= \mathrm{Spec}\, A[x^{\pm 1/p^\infty}]$-equivariant group scheme $\mathrm{Spec}\,B$, or, more specifically in our case, an $\mathbb{N}[1/p]$-graded group scheme structure on $\mathrm{Spec}\, B.$ The summand of $B$ of degree $i \in \mathbb{N}[1/p]$ will be denoted as $B_i,$ so that $B = \bigoplus_{i \in \mathbb{N}[1/p]} B_i.$ Further, the forgetful functor from $\mathbb{G}_a^{\mathrm{perf}}$-modules to $\mathbb{N}[1/p]$-graded (or $\mathbb{Z}[1/p]$-graded) group schemes is fully faithful. The proof follows in a way entirely similar to the proof of \cref{omgwat}.
\end{remark}{}

\begin{proposition}\label{conn1}Let $\mathrm{Spec}\, B$ be a $\mathbb{G}_a^{\mathrm{perf}}$-module. Then as a graded algebra, the degree zero piece of $B$ is isomorphic to $A$ as an $A$-module.
 \end{proposition}{}

\begin{proof}First, by using the zero section of a group scheme, we note that the $A$-algebra structure map $A \to B$ is injective. The $\mathbb{G}_a^{\mathrm{perf}}$-module structure map is given by a map $B \to B[x^{1/p^\infty}].$ Killing $x^{1/p^n}$ for all $n$ produces a map $B \to B$ whose kernel is $I_{>0} := \bigoplus_{i>0} B_i$ where $B = \bigoplus_{i \in \mathbb{N}[1/p]} B_i.$ Further, using the fact that $\mathrm{Spec}\,B$ is a $\mathbb{G}_a^{\mathrm{perf}}$-module, we note that the map $B \to B$ obtained this way also has the property that it factors through $B \to A$ which is the zero map of the comultiplication. Since the map $A \to B$ is injective, this provides an $A$-algebra map $B \to A$ whose kernel is $I_{>0}.$ Thus $B_0$ is naturally isomorphic to $A$, as desired. 
\end{proof}{}

\begin{remark}A $\mathbb{G}_a^{\mathrm{perf}}$-module over $\mathbb{F}_p$ amounts to the following data. For every $\mathbb{F}_p$-algebra $S,$ we have an $S^\flat$-module scheme $\mathrm{Spec}\,M_S$ over $S$ such that for a map $\varphi: S \to R$ of $\mathbb{F}_p$-algebras, we have isomorphisms $
\mathrm{res}_{S}^{R} : M_S \otimes_{S} R \simeq M_R$ of $R$-algebras. Further, the action of $ S^\flat$ on $M_S$ provides an endomorphism of $M_S$ for every $s^\flat \in S^\flat$ which under the isomorphism $\mathrm{res}_S^R$ corresponds to the endomorphism induces by $\varphi^\flat (s^\flat)$ on $M_R.$ Here $\varphi^\flat$ denotes the map $S^\flat \to R^\flat.$ Using the map $S^\flat \to S$ for an $\mathbb{F}_p$-algebra $S$, we see that it is enough to specify the same data only on perfect rings.
\end{remark}{}

\begin{proposition}\label{perfect}Let $(A, \mathfrak{m})$ be an Artinian local ring with residue field $\mathbb{F}_p.$ For every perfect ring $R$, let $W_A(R):= A \otimes_{\mathbb{Z}_p} W(R).$ Then a $\mathbb{G}_a^{\mathrm{perf}}$-module over $A$ is equivalent to the following data:\vspace{2mm}

\noindent
1. For every perfect ring $R$, an $R$-module scheme $\mathrm{Spec}\, M_R$ over $W_A(R).$\vspace{1mm}

\noindent
2. For every map $S \to R$ of perfect rings, an isomorphism $\mathrm{res}_S^R: M_S \otimes_{W_A(S)} W_A(R) \simeq M_R.$ Further, in this isomorphism, the $S$-action on the left hand side is compatible with the restriction of the $R$-action on the right hand side along the map $S \to R.$ The latter is a condition and not extra data.\vspace{2mm}

Further, a morphism of $\mathbb{G}_a^{\mathrm{perf}}$-module under this equivalence translates to the following data:\vspace{2mm}

\noindent
1. For every perfect ring $R$, a morphism $\Phi_R$ of $R$-module schemes over $W_A(R).$\vspace{1mm}

\noindent
2. For every map $S \to R$ of perfect rings, the maps $\Phi_R$ and $\Phi_S$ are compatible with $\mathrm{res}_S^R.$
\end{proposition}{}

\begin{proof}Let $R$ be a perfect ring. We note that $W_A(R)$ is an $A$-algebra. So given a $\mathbb{G}_a^{\mathrm{perf}}$-module scheme over $W_A(R)$, we obtain an $W_A(R)^\flat$-module scheme which will be denoted as $M_R.$ Thus one direction of the proposition will follow from the following lemma.

\begin{lemma}\label{lem1}In the above set up, $W_A(R)^\flat \simeq R.$
\end{lemma}{}

\begin{proof} This fact is rather classical and a proof can be found in \cite[Prop. 2.1.2]{FF18}. We will explain a proof in our set up for the convenience of the reader. We start by noting that there is a natural map $W_A(R)^\flat:= \lim_{x \to x^p} W_A(R)/p \to \lim_{x \to x^p}W_A(R)/\mathfrak{m}$ since $p \in \mathfrak{m}.$ Since $W_A(R)/\mathfrak{m}$ is isomorphic to $R$, which is a perfect ring, it will be enough to prove that the natural map is an isomorphism. We know that the natural map $\lim_{x \to x^p} W_A(R) \to W_A(R)^\flat $ is a set theoretic bijection since $W_A(R)$ is $p$-adically complete. Thus it is enough to show that the natural map $\lim_{x \to x^p} W_A(R) \to \lim_{x \to x^p} W_A(R)/\mathfrak{m}$ is a set theoretic bijection. First we check injectivity. Let $(a_n)$ and $(b_n)$ be two sequences in $\lim_{x \to x^p} W_A(R)$ such that $a_n = b_n \,\mathrm{mod}\, \mathfrak{m}.$ For every $k$, we have $a_{n+k}^{p^k }= a_n$ and $b_{n+k}^{p^k}= b_n.$ Since $a_{n+k} = b_{n+k} \,\mathrm{mod}\, \mathfrak{m},$ one inductively checks using $p \in \mathfrak{m}$ that  $a_n = b_n \, \mathrm{mod}\, \mathfrak{m}^{k+1}.$ Since $k$ was arbitrary, and the ideal $\mathfrak{m}$ is nilpotent, this checks the injectivity. For surjectivity, we fix $(\overline{a_n}) \in \lim_{x \to x^p} W_A(R)/\mathfrak{m}.$ We choose arbitrary lifts $a_n$ of $\overline{a_n}$ to $W_A(R).$ For every $k$, we have $a_{n+k+1}^p = a_{n+k} \, \mathrm{mod}\, \mathfrak{m}.$ Thus, $a_{n+k+1}^{p^{k+1}}= a_{n+k}^{p^k}\, \mathrm{mod}\, \mathfrak{m}^{k+1} .$ Since $\mathfrak{m}$ is nilpotent, the sequence $k \to a_{n+k}^{p^k}$ is eventually constant, and we define the limit element to be $b_n.$ Now it follows that $b_{n+1}^p = b_n$ and $b_n$ lifts $\overline{a}_n$, which proves the required surjectivity. 
\end{proof}{}

For the opposite direction, we are given with the data of an $R$-module scheme $\mathrm{Spec}\, M_R$ over $W_A(R)$ for every perfect ring $R$. In order to obtain a $\mathbb{G}_a^{\mathrm{perf}}$-module, we are required to provide the data of a $B^\flat$-module scheme $\mathrm{Spec}\, M_B$ over $B$ for every $A$-algebra $B.$  For this, we note the following lemma.

\begin{lemma}Let $B$ be an $A$-algebra. There is a natural map $W_A(B^\flat) \to B$ which induces an isomorphism $W_A(B^\flat)^\flat \to B^\flat.$
\end{lemma}{}

\begin{proof}There is a natural map $B^\flat \to B/p \to B/\mathfrak{m}.$ This gives a natural map of $A$-algebras $W_A(B^\flat) \to B^\flat \to B/\mathfrak{m}.$ We note that $W_A(B^\flat)$ is a flat $A$-algebra by definition. Since $B^\flat$ is perfect, we have $\mathbb{L}_{B^\flat/\mathbb{F}_p}=0$ implying $\mathbb{L}_{W_A(B^\flat)/A}=0.$ This implies that $W_A(B^\flat)$ is a formally \'etale $A$-algebra. Since the map $B \to B/\mathfrak{m}$ has nilpotent kernel, it follows that the map $W_A(B^\flat) \to B/\mathfrak{m}$ lifts uniquely to provide a map $W_A(B^\flat) \to B,$ as desired. The map $W_A(B^\flat)^\flat \to B^\flat$ is an isomorphism by \cref{lem1}.
\end{proof}{}

Now we can define $M_B:=M_{B^\flat}\otimes_{W_A(B^\flat)} B$. Then $\mathrm{Spec}\, M_B$ automatically has the structure of a $B^\flat$-module scheme. This data determines a $\mathbb{G}_a^{\mathrm{perf}}$-module.
\end{proof}{}

\begin{definition}[Pointed $\mathbb{G}_a^{\mathrm{perf}}$-module]A $\mathbb{G}_a^{\mathrm{perf}}$-module scheme $X$ over $A$ along with the data of a map $X \to \mathbb{G}_a^{\mathrm{perf}}$ of $\mathbb{G}_a^{\mathrm{perf}}$-modules will be called a \textit{pointed $\mathbb{G}_a^{\mathrm{perf}}$-module $X$.} We will follow the convention that the data of the map $X \to \mathbb{G}_a^{\mathrm{perf}}$ will be simply called a point. Maps between pointed $\mathbb{G}_a^{\mathrm{perf}}$-modules are maps of $\mathbb{G}_a^{\mathrm{perf}}$-modules that commute with the points. We denote the category of such objects by $\mathbb{G}_a^{\mathrm{perf}} \w{--}\mathrm{Mod}_*.$
\end{definition}{}

\begin{example}\label{alpha}Let $A$ be an $\mathbb{F}_p$-algebra. The Hopf algebra $A[x^{1/p^\infty}]/x$ along with the natural map $A[x^{1/p^\infty}] \to A[x^{1/p^\infty}]/x$ equips $\mathrm{Spec}\, A[x^{1/p^\infty}]/x$ with the structure of a pointed $\mathbb{G}_a^{\mathrm{perf}}$-module scheme. This will be denoted as $\alpha^{\natural}.$ An analogue of this does not exist when $A$ does not have characteristic $p.$ 
\end{example}{}

\begin{remark}Unlike the case of $\mathbb{G}_a$-modules from \cref{deg1}, for a $\mathbb{G}_a^{\mathrm{perf}}$-module $X = \mathrm{Spec}\, B$, it is not true that giving a map $X \to \mathbb{G}_a^{\mathrm{perf}}$ is equivalent to choosing an element of degree $1$ in $B.$
\end{remark}{}

For the remainder of this section, we will work over a base ring $A$ of characteristic $p.$

\begin{proposition}[Pullback functor]\label{pullback}Let $A$ be an $\mathbb{F}_p$-algebra. Then there is a map of ring schemes $u:\mathbb{G}_a^{\mathrm{perf}} \to \mathbb{G}_a$ over $A.$ Further, pullback along this map defines a fully faithful functor $u^*: \mathbb{G}_a\w{--}\mathrm{Mod}_* \to \mathbb{G}_a^{\mathrm{perf}}\w{--}\mathrm{Mod}_*.$
\end{proposition}{}

\begin{proof}The first part follows from considering the natural map $S ^\flat \to S$ for every $A$-algebra $S$ since $\mathbb{G}_a^{\mathrm{perf}}(S) = S^\flat$ and $\mathbb{G}_a(S)= S.$ For the second part, we start by defining the functor $u^*.$ Let $X$ be a pointed $\mathbb{G}_a$-module over $A.$ We set $u^*X := X \times_{\mathbb{G}_a} \mathbb{G}_a^{\mathrm{perf}}.$ Then it follows that $u^* X$ is naturally equipped with the structure of a pointed $\mathbb{G}_a^{\mathrm{perf}}$-module where the map $u^*X \to \mathbb{G}_a^{\mathrm{perf}}$ is given by the projection map.
\vspace{2mm}

It is clear that $u^*$ is faithful. To see that it is full, we take $\mathrm{Spec}\, M$ and $\mathrm{Spec} \,N$ to be two pointed $\mathbb{G}_a$-modules over $A$. Let $f: u^*\mathrm{Spec}\, M \to u^*\mathrm{Spec}\, N$ be a map of pointed $\mathbb{G}_a^{\mathrm{perf}}$-modules. Now the graded algebras underlying $u^* \mathrm{Spec}\, M$ and $u ^*\mathrm{Spec}\, N$ are respectively given by $M':=M \otimes_{A[x]} A[x^{1/p^\infty}]$ and $N':=N \otimes_{A[x]} A[x^{1/p^\infty}].$ The map $f$ induces a graded map on algebras $ \overline{f}:N' \to M'.$ By considering the $A$-subalgebra of elements of integral degree in $M'$ and $N'$ we recover $M$ and $N$ and also get a graded map denoted as $\overline{f}_\circ: N \to M.$ Since $\overline{f}$ was a graded Hopf algebra map, it follows that $\overline{f}_{\circ}$ is also a graded Hopf algebra map. This induces a map $f_{\circ}: \mathrm{Spec}\, M \to \mathrm{Spec}\, N$ which by construction is a map of pointed $\mathbb{G}_a$-modules. Now we note that applying $u^*$ to this map recovers $f$. To see the last statement, we note that since $f$ was a map of pointed $\mathbb{G}_a^{\w{perf}}$--modules, the graded Hopf algebra map $\overline{f}:N' \to M'$ must send $1 \otimes x^{1/p^i} \mapsto 1 \otimes x^{1/p^i}$ for $i \ge 0.$ Further, as an $A$-algebra, $N'=N \otimes_{A[x]} A[x^{1/p^\infty}]$ is generated by the image of the natural map $N \to N'$ and the elements $1 \otimes x^{1/p^i}$ for $ i \ge 0$ (similarly for $M'$). These observations imply that the map $\overline{f}$ is the map obtained from $\overline{f}_\circ$ by base changing along $A[x] \to A[x^{1/p^\infty}],$ which ultimately implies that $f = u^* f_\circ.$ 
\end{proof}{}

\begin{definition}[Fractional rank 1]\label{satu3}Let $A$ be an $\mathbb{F}_p$-algebra. A pointed $\mathbb{G}_a^{\mathrm{perf}}$-module $X$ over $A$ is said to be of \textit{fractional rank $1$} if it is isomorphic to $u^*X'$ for a pointed $\mathbb{G}_a$-module $X'.$
\end{definition}{}

\begin{remark}\label{rem3}If $X = \mathrm{Spec}\,V$ is a pointed $\mathbb{G}_a^{\mathrm{perf}}$-module of fractional rank $1$ over $\mathbb{F}_p$, then it follows from the definition and \cref{conn1} that the map of graded algebras $\mathbb{F}_p[x^{1/p^\infty}] \to V$ corresponding to the point is an isomorphism in degrees $<1$ and thus the pieces of degree $<1$ of $V$ are vector spaces of dimension $1$. More precisely, let $V_{<1}$ denote the algebra obtained by killing the ideal of elements of degree $\ge 1.$ Further, let $X$ be isomorphic to $u^* \mathrm{Spec}\, U$ for a pointed $\mathbb{G}_a$-module $\mathrm{Spec}\, U$. Then $V_{<1} \simeq \mathbb{F}_p[x^{1/p^\infty}]/x \otimes_{\mathbb{F}_p} U_0 \simeq \mathbb{F}_p[x^{1/p^\infty}]/x;$ where the last isomorphism follows from \cref{conn}. Also, we note that $\dim U_1 = \dim V_1.$
\end{remark}{}

\begin{example}\label{stilladding??945}
Let $A$ be an $\mathbb{F}_p$-algebra and let us consider $\w{Spec}\, A$ equipped with the pointed $\mathbb{G}_a$-module structure coming from the zero section of $\mathbb{G}_a.$ Then we have a natural isomorhism $u^* (\w{Spec}\, A) \simeq \alpha^\natural$ (\cref{alpha}). In particular, the pointed $\mathbb{G}_a^{\w{perf}}$-module $\alpha^\natural$ is of fractional rank $1$.
\end{example}{}
\subsection{The Hodge map}

\begin{proposition}[Hodge map]\label{Hmap}Let $A$ be an $\mathbb{F}_p$-algebra. Let $X$ be a pointed $\mathbb{G}_a^{\mathrm{perf}}$-module over $A$ of fractional rank $1.$ Then there is a unique map $\alpha^{\natural} \to X$ in $\mathbb{G}_a^{\mathrm{perf}}\w{--}\mathrm{Mod}_*.$ This map will be called the Hodge map.
\end{proposition}{}

\begin{proof}
This follows from \cref{pullback} and \cref{stilladding??945}.
\end{proof}{}

\begin{remark}We will later see (\textit{cf.} \cref{move2}) that the Hodge map $\alpha^{\natural} \to X$ corresponds to the $\mathrm{gr}_0$ map of a certain kind of ``Hodge filtration" that can be defined on a functor associated to $X.$
\end{remark}{}

Now we would like to study a ``rigidity" property of the Hodge map $\alpha^{\natural} \to X$ from \cref{Hmap} under deformations. For this purpose, we make the following definition.

\begin{definition}\label{whot}
A pointed $\mathbb{G}_a$-module $X$ is said to be \textit{full} of rank $1$ if the map $X \to \mathbb{G}_a$ induces a surjection on the piece of degree $1$ on the underlying graded algebra map obtained by taking global sections.
\end{definition}{}

\begin{definition}[Full of fractional rank $1$]\label{satu100}
A pointed $\mathbb{G}_a^{\mathrm{perf}}$-module $X$ over an $\mathbb{F}_p$-algebra $A$ is said to be \textit{full} of fractional rank $1$ if it is of fractional rank $1$ and the map $X \to \mathbb{G}_a^{\mathrm{perf}}$ induces a surjection on the piece of degree $1$ on the underlying graded algebra map obtained by taking global sections.
\end{definition}{}

\begin{remark}\label{rem4}It follows from definitions that a pointed $\mathbb{G}_a^{\mathrm{perf}}$-module $\mathrm{Spec}\, V$ is full of fractional rank $1$ if and only if it is the image under $u^*$ of a full of rank $1$ pointed $\mathbb{G}_a$-module $\mathrm{Spec}\, U$. In this situation, when $A= \mathbb{F}_p$, we have two cases.\vspace{1mm}

\noindent
\textit{Case 1.}
If the graded map $\mathbb{F}_p[x] \to U$ sends $x \to 0$, then $V = U \otimes_{\mathbb{F}_p [x]} \mathbb{F}_p [x^{1/p^\infty}] = U\otimes_{\mathbb{F}_p} \frac{\mathbb{F}_p [x^{1.p^\infty}]}{x}.$ Thus the graded map $\mathbb{F}_p[x^{1/p^\infty}] \to V$ corresponding to the pointed $\mathbb{G}_a^{\mathrm{perf}}$-module structure is an isomorphism in degrees $<1$ and $x$ is sent to zero, so $V$ has no non-zero elements in degree $i$ for $1 \le i<2.$\vspace{2mm}

\noindent
\textit{Case 2.} Otherwise, $\dim U_1 = \dim V_1 = 1$ and the graded map $\mathbb{F}_p[x] \to U$ sends $x$ to a basis element of $U_1.$ Then the map $\mathbb{F}_p [x^{1/p^\infty}] \to V$ corresponding to the pointed $\mathbb{G}_a^{\mathrm{perf}}$-module structure is an isomorphism in degrees $i$ for $0 \le i<2.$
\end{remark}{}

\begin{proposition}[Rigidity of the Hodge map]\label{hodgestability}Let $X$ be a pointed $\mathbb{G}_a^{\mathrm{perf}}$-module over $\mathbb{F}_p$ which is full of fractional rank $1$. Let $X'$ be a deformation of $X$ as a pointed $\mathbb{G}_a^{\mathrm{perf}}$-module over $\mathbb{F}_p[\epsilon]/\epsilon^2.$ Then the Hodge map $ \alpha^{\natural} \to X $ admits a unique deformation $\alpha^{\natural} [\epsilon] \to X'.$
\end{proposition}{}

\begin{proof}We write $X= \mathrm{Spec}\,B$ and $X'= \mathrm{Spec}\, B'.$ We have a map $\mathbb{F}_p [\epsilon][x^{1/p^\infty}] \to B'$ coming from the data of the point. By \cref{rem3}, this is an isomorphism in degrees $<1$ since it is so modulo $\epsilon.$ Let $B'_{<1}$ denote the graded algebra we obtain by killing the ideal of elements in degrees $\ge 1$ in $B$. This gives an isomorphism of graded algebras $B'_{<1} \simeq \frac{\mathbb{F}_p[\epsilon][x^{1/p^\infty}]}{x}$. Thus the quotient map $B' \to B'_{<1}$ can be identified with a graded algebra map $B' \to \frac{\mathbb{F}_p[\epsilon][x^{1/p^\infty}]}{x}$ or in other words, a graded map $\alpha^{\natural}[\epsilon] \to \mathrm{Spec}\, B'.$ We note that both sides have the structure of pointed $\mathbb{G}_a^{\mathrm{perf}}$-modules and the graded map of schemes we have constructed is compatible with the data of the points. Thus in order to prove that it is a map of pointed $\mathbb{G}_a^{\mathrm{perf}}$-modules, we only need to check that $B' \to \frac{\mathbb{F}_p[\epsilon][x^{1/p^\infty}]}{x}$ is a map of graded Hopf algebras, i.e., the following diagram commutes.

\begin{center}
\begin{tikzcd}
B' \arrow[r] \arrow[d] & \frac{\mathbb{F}_p[\epsilon][x^{1/p^\infty}]}{x} \arrow[d]  \\
B'\otimes B' \arrow[r] & \frac{\mathbb{F}_p[\epsilon][x^{1/p^\infty}]}{x} \otimes_{\mathbb{F}_p[\epsilon]} \frac{\mathbb{F}_p[\epsilon][x^{1/p^\infty}]}{x}
\end{tikzcd}
\end{center}

It is enough to check that the diagram commutes for homogeneous elements $b \in B'.$ When $\deg b <1$, it follows from the data of the points. By \cref{rem4}, if $X$ falls under \text{Case 1}, then the diagram commutes for $1 \le \deg b <2$, as $b$ is necessarily zero in that case. If $X$ falls under Case $2$, then the map $\mathbb{F}_p[x^{1/p^\infty}] \to B$ is an isomorphism in degrees $<2.$ Therefore, the map $\mathbb{F}_p[\epsilon][x^{1/p^\infty}] \to B'$ coming from the data of the point is an isomorphism in degrees $<2$ as well. Thus the diagram commutes for $1 \le \deg b <2$ in this case as well. Now we suppose that $\deg b \ge 2.$ The comultiplication would send this to a homogeneous element of $B'\otimes B'.$ However, any homogeneous element of degree $\ge 2$ in $B'\otimes B'$ would have to be of the form $\sum_{u}x_u \otimes y_u $ such that $x_u, y_u$ are homogeneous elements in $B'$ and $\deg x_u + \deg y_u = \deg b \ge 2$, therefore either $\deg x_u$ or $\deg y_u$ is $\ge 1$, implying that $\sum_u x_u \otimes y_u$ is sent to zero under the lower horizontal map. But since $\deg b \ge 2,$ it is sent to zero by the upper horizontal map as well, which kills every element in degree $\ge 1$, verifying the commutativity of the diagram. The uniqueness follows from the grading.
\end{proof}{}

\begin{example}\label{unstable}We note that \cref{hodgestability} is false if we do not assume the  $\mathbb{G}_a^{\mathrm{perf}}$-module to be full, i.e., assuming that the  $\mathbb{G}_a^{\mathrm{perf}}$-module is of fractional rank $1$ alone is not sufficient. We consider the group scheme $\alpha_p$ as a pointed $\mathbb{G}_a$-module over $\mathbb{F}_p$ via the map $\mathbb{F}_p[x] \to \frac{\mathbb{F}_p [x]}{x^p}$ that sends $x \to 0.$ Applying the functor $u^*$, we obtain a pointed $\mathbb{G}_a^{\mathrm{perf}}$-module via the map $\mathbb{F}_p[x^{1/p^\infty}] \to \frac{\mathbb{F}_p[x^{1/p^\infty}]}{x} \otimes_{\mathbb{F}_p} \frac{\mathbb{F}_p[t]}{t^p}.$ This can have deformations that do not admit a pointed $\mathbb{G}_a^{\mathrm{perf}}$-module map to $\alpha^{\natural}[\epsilon].$ Indeed, we consider the graded algebra $\frac{\mathbb{F}_p[\epsilon][x^{1/p^\infty}]}{x} \otimes_{\mathbb{F}_p[\epsilon]} \frac{\mathbb{F}_p[\epsilon][t]}{t^p}.$ The Hopf structure has a nontrivial deformation given by $$t \to t \otimes 1 + 1 \otimes t + \epsilon \sum_{0 < i < p} \frac{1}{p}\binom{p}{i} x^{i/p} \otimes x^{1-i/p}.$$ This deformation further has the structure of a pointed $\mathbb{G}_a^\mathrm{perf}$-module. However, this does not have a deformation of the pointed $\mathbb{G}_a^{\mathrm{perf}}$-module map $\frac{\mathbb{F}_p[x^{1/p^\infty}]}{x} \otimes_{\mathbb{F}_p} \frac{\mathbb{F}_p[t]}{t^p} \to \alpha^{\natural}$ obtained by killing all elements of degree $\ge 1.$ Indeed, any deformation of such a map would be given by killing all the generators in degrees $\ge 1$ as well, which would not be a Hopf map.
\end{example}

\subsection{Cartier duality}

In this section, we record a variant of Cartier duality in the graded situation. For more details on such constructions we refer to \cite[1.6]{GR14}. We will use this duality to study the deformation of certain pointed $\mathbb{G}_a$ and $\mathbb{G}_a^{\mathrm{perf}}$-modules. 

\begin{definition}Let $R$ be any ring. A nonnegatively graded module $\bigoplus_{i\ge 0} V_i$ over $R$ is said to be of \textit{free of finite type} if $V_i$ is a finite dimensional free module over $R$ for every $i \ge 0.$ A graded algebra over $R$ will be called free of finite type of it is free of finite type as a module over $R.$
\end{definition}

\begin{definition}
If $M=\bigoplus_{i\ge 0} V_i$ is a free of finite type graded module over $R$, then we can define the dual of $M$ as $M^*:= \bigoplus_{i\ge 0} V_i^*$, where $V_i^*$ denotes the dual of $V_i.$ It follows that $M^{**}$ is functorially isomorphic to $M.$
\end{definition}{}

\begin{definition}Let $S$ be a graded free of finite type Hopf algebra over $R.$ Then $S^*$ also has the structure of a graded free of finite type Hopf algebra over $R.$ We call $S^*$ the Cartier dual of $S.$ It follows that $S^{**}$ is naturally isomorphic to $S.$ Thus Cartier duality provides an anti-equivalence between the category of nonnegatively graded free of finite type Hopf algebras over $R$ with itself.
\end{definition}{}

\begin{definition}Let $\mathcal{P}^1_*$ denote the category of nonnegatively graded affine group schemes $X$ over $R$ whose underlying Hopf algebra is free of finite type along with the data of a map $X \to \mathbb{G}_a$ of graded group schemes such that at the level of graded algebras of global sections, this map induces an \textit{isomorphism} on degrees $\le 1.$ The map $X \to \mathbb{G}_a$ will be called a point. Morphisms between two such objects are morphisms of graded group schemes that commute with the data of the points.
\end{definition}{}

\begin{remark}\label{high}We recall that a nonnegatively graded Hopf algebra $S$ over $R$ is called connected if the degree zero piece of $S$ denoted as $S_0$ is isomorphic to $R$ as an $R$-module. Equivalently, $S$ is called connected if the structure map $R \to S$ induces isomorphism on the degree zero piece. It follows from the definition that the underlying graded Hopf algebra $\Gamma(X, \mathcal{O}_X)$ is connected for an $X \in \mathcal{P}^1_*.$ We note that if $S$ is a connected nonnegatively graded Hopf algebra, then the comultiplication on $S$ sends an element $x$ of homogeneous degree $1$ to $x \otimes 1 + 1 \otimes x.$ In particular, giving a map $R[x] \to S$ of connected nonnegatively graded Hopf algebras over a base ring $R$ amounts to choosing a homogeneous element of degree $1$ in $S.$  

\end{remark}{}

\begin{remark}The identity map $\mathbb{G}_a \to \mathbb{G}_a$ is the final object of $\mathcal{P}^1_*$.
\end{remark}

\begin{proposition}[Duality]We fix a base ring $R$ as before. The category $\mathcal{P}^1_*$ has a notion of duality which sends an object $\mathrm{Spec}\, M \to \mathbb{G}_a$ to another object $\mathrm{Spec}\, M^* \to \mathbb{G}_a$, where $M^*$ is the Cartier dual of the nonnegatively graded free of finite type Hopf algebra $M.$ Further, this duality in $\mathcal{P}^1_*$ is involutive and thus provides an anti-equivalence of $\mathcal{P}^1_*$ with itself. This duality in $\mathcal{P}^1_*$ will be called Cartier duality as well.
\end{proposition}{}

\begin{proof}Let $\mathrm{Spec}\,M \to \mathbb{G}_a \in \mathcal{P}^1_*$. We need to check that there is a map of graded Hopf algebras $R[x] \to M^*$ inducing isomorphism in degrees $\le 1$ which is further functorial and is compatible with applying Cartier duality twice. First we note the following lemma.

\begin{lemma}Let $\mathcal{PL}(R)$ denote the category whose objects are pairs $(L,p )$ where $L$ is a free module of rank $1$ over $R$ and $p$ is a basis of $L$ as an R-module. Maps are defined in the obvious way. Then there is a notion of functorial duality on $\mathcal{PL}(R)$ which is involutive and sends $L \mapsto L^*$.
\end{lemma}

\begin{proof} Our task is to construct a functor $\mathcal{PL}(R) \to \mathcal{PL}(R)$ which we define by sending $(L,p) \to (L^*, p^*)$ where $p^*: L \to R$ is the unique map that sends $p \to 1.$ It is clear that $p^*$ is a basis for $L^*.$ If $(L_1, p_1) \to (L_2, p_2)$ is an arrow in $\mathcal{PL}(R)$, then there is a natural map $L_2^* \to L_1^*$ which takes $p_2^*$ to $p_1^*$, thus it is clear that our construction defines a functor. The natural isomorphism $L \simeq L^{**}$ sends $p \to p^{**}$ so it follows that the functor we constructed is involutive. \end{proof}{}

Now we note that the Cartier dual $M^*$ of $M$ is again a connected Hopf algebra over $R$ since $M$ is connected. By \cref{high}, the map $\mathrm{Spec}\,M \to \mathbb{G}_a$ corresponds to an object $(M_1, t) \in \mathcal{PL(R)},$ where $M_1$ denotes the degree $1$ piece of $M$ and $t \in M_1$ is the image of $x$ under the graded map $R[x] \to M.$ By the above Lemma, duality provides us an object $(M_1^*, t^*).$ By \cref{high}, this corresponds to a map $\mathrm{Spec}\, M^* \to \mathbb{G}_a$. By construction, the underlying graded algebra map $R[x] \to M^*$ induces an isomorphism in degrees $\le 1$. Thus $\mathrm{Spec}\, M^* \to \mathbb{G}_a$ is naturally an object of $\mathcal{P}^1_*.$ The fact that this construction is functorial and involutive follows from the above Lemma.\end{proof}{}

\begin{example}\label{final}It follows that the category $\mathcal{P}_*^1$ has an initial object $\mathbb G_a ^* \to \mathbb G_a$, which is the Cartier dual of the final object $\mathbb{G}_a \to \mathbb G_a$ and will be simply denoted by $\mathbb{G}_a^*.$ By computing the Cartier dual, it follows that the graded algebra underlying $\mathbb{G}_a^*$ is the divided power polynomial algebra in one variable. By construction, $\mathbb{G}_a^*$ is a nonnegatively graded group scheme since it is the dual of $\mathbb{G}_a$. It also follows that in the graded algebra underlying $\mathbb{G}_a^*,$ the summands of a fixed degree are all free of rank $1.$ Let us mention the group scheme structure of $\mathbb{G}_a^*$ more explicitly at the level of functor of points. For an algebra $B,$ the $B$-valued points of $\mathbb{G}_a^*$ can be identified with the set of sequences $(b_0, b_1, \ldots, b_n ,\ldots),$ where $b_i \in B$ for $i \ge 0$ are such that $b_1=1$ and $b_n b_m = \frac{(m+n)!}{m! n!} b_{n+m}.$ The addition operation in the group scheme $\mathbb{G}_a^*$ at the level of $B$-valued points can be described as $$(b_0, b_1, \ldots, b_n, \ldots) + (c_0, c_1, \ldots, c_n, \ldots)= (d_0, d_1, \ldots, d_n, \ldots),\,\, \text{where}\,\, d_k := \sum_{i+j=k} b_{i} c_{j}.$$ For $b \in B,$ setting $b\cdot(b_0, b_1, \ldots, b_n, \ldots) := (b_0, b b_1, \ldots, b^n b_n, \ldots)$ describes the $\mathbb{G}_a$-action on $\mathbb{G}_a^*.$ By using the relation $b_n b_m = \frac{(m+n)!}{m! n!} b_{n+m}$ and the binomial theorem, one easily verifies that the $\mathbb{G}_a$-action on $\mathbb{G}_a^*$ equips $\mathbb{G}_a^*$ with the structure of a $\mathbb{G}_a$-module.
\end{example}{}

\begin{proposition}\label{dri}Let $R$ be a $\mathbb{Z}_p$-algebra. Then $\mathbb{G}_a^*$ as an object of $\mathcal{P}^1_*$ over $R$ is uniquely isomorphic to $W[F] \to \mathbb G_a$, where the latter denotes the Kernel of Frobenius on the $p$-typical Witt ring scheme (see \cref{introex1}). (In particular, $W[F]$ has the structure of a pointed $\mathbb{G}_a$-module.)
\end{proposition}{}

\begin{proof} For a proof, we refer to \cite[Lemma 3.2.6]{Dri20}. We provide sketch of a different proof. It is enough to prove the claim when $R= \mathbb{Z}_p.$ Since $F: W \to W$ is faithfully flat (see \cite[Section 3.4]{Dri20} and \cref{introex1}), $W[F]$ is a graded group scheme that is flat over $\mathbb{Z}_p.$ For every $i \ge 0$, the summand of degree $i$ in the graded algebra underlying $W[F]$ is a finitely generated flat module over $\mathbb{Z}_p$ and therefore must be free of finite rank. By going modulo $p,$ one sees that these summands of degree $i$ must be free of rank $1$ for every $i \ge 0.$ Thus one observes that $W[F] \to \mathbb G_a$ is an object of $\mathcal{P}^1_*.$ Since $\mathbb{G}_a^* \to \mathbb{G}_a$ is an initial object of $\mathcal{P}_*^1$ by \cref{final}, it follows that there is a unique map $\mathbb{G}_a^* \to W[F]$ in $\mathcal{P}^1_*.$ In order to check that this map is an isomorphism, it is enough to check it at the level of the induced map on graded algebras underlying the two group schemes. For the latter, it is further enough to check it for the induced maps on summands of degree $i$, at the level of $\mathbb{Z}_p$-modules. Since these summands all have rank $1,$ it is enough to check that these maps are nonzero modulo $p.$ However, the last statement can be seen directly by inspecting the comultiplication of the Hopf algebras underlying $\mathbb{G}_a^*$ and $W[F]$ after reducing modulo $p.$ \end{proof}{}

\begin{remark}\label{rem7}Fix $n \ge 1$. Let $\mathcal{P}^1_{*_{< n}}$ denote the full subcategory of $\mathcal{P}^1_*$ spanned by the objects $\mathrm{Spec}\, M \to \mathbb{G}_a$ whose underlying graded algebra $M$ satisfies $M_i = 0$ for $i \ge n.$ Then $\mathcal{P}^{1}_{*_{< n}}$ is preserved under Cartier duality. When $R$ is a char. $p$ base ring and $n = p^k$, then the final object of $\mathcal{P}^1_{*_{< n}}$ is given by $\alpha_{p^k} \to \mathbb{G}_a$ where $\alpha_{p^k}$ is the $\mathbb{G}_a$-module underlying $\mathrm{Spec}\, R[x]/x^{p^k}.$ Its Cartier dual is given by $W_k[F]$, where $W_k$ is the kernel of Frobenius on the $k$-truncated $p$-typical Witt ring scheme. The Hopf structure on the ring of functions on $W_k[F]$ is obtained by considering the subalgebra of elements in degree $< p^k$ in the Hopf algebra underlying $W[F].$ By duality, $W_k[F] \to \mathbb{G}_a$ is the initial object of $\mathcal{P}^1_{*_{< n}}.$ We note that $W_k[F]$ is also a $\mathbb{G}_a$-module and thus $W_k[F]$ has the structure of a pointed $\mathbb G_a$-module. Explicitly, the graded $R$-algebra underlying $W_k[F]$ (considered as a graded group scheme over $R$) is given by $\frac{R[x_0, x_1, \ldots, x_{k-1}]}{(x_0^p, x_1^p, \ldots, x_{k-1}^p)},$ where $\deg x_i= p^i.$
\end{remark}{}

\begin{example}[Duality in $\mathcal{P}^1_*$]\label{cart}We consider the graded algebra $\frac{\mathbb{F}_p [x_0, x_1, \ldots, x_n]}{(x_0^{p^{k+1}},x_1^p, \ldots, x_n^p) } $ where $\deg x_0=1$ and $\deg x_i = p^{k+i}$ for $i \ge 1.$ This has a graded subalgebra given by $\frac{\mathbb{F}_p [x_0^{p^k}, x_1, \ldots, x_n]}{(x_0^{p^{k+1}},x_1^p, \ldots, x_n^p) }$. Let us write $\frac{\mathbb{F}_p [t_0, t_1, \ldots, t_n]}{(t_0^{p},t_1^p, \ldots, t_n^p) }$ to denote the graded Hopf algebra underlying $W_{n+1}[F],$ where $\deg t_i = p^i$ for $i \ge 0.$ Thus, the algebra $\frac{\mathbb{F}_p [x_0^{p^k}, x_1, \ldots, x_n]}{(x_0^{p^{k+1}},x_1^p, \ldots, x_n^p) }$ can be naturally equipped with the Hopf structure coming via the isomorphism with $\frac{\mathbb{F}_p [t_0, t_1, \ldots, t_n]}{(t_0^{p},t_1^p, \ldots, t_n^p) }$ that sends $t_0 \mapsto x_0^{p^k}$ and $t_i \mapsto x_i$ for $i \ge 1.$ This equips the graded algebra $\frac{\mathbb{F}_p [x_0, x_1, \ldots, x_n]}{(x_0^{p^{k+1}},x_1^p, \ldots, x_n^p) } $ with a graded Hopf structure which is uniquely obtained by sending $x_0 \mapsto x_0 \otimes 1 + 1 \otimes x_0 $ and requiring that the graded subalgebra $\frac{\mathbb{F}_p [x_0^{p^k}, x_1, \ldots, x_n]}{(x_0^{p^{k+1}},x_1^p, \ldots, x_n^p) }$ equipped with the aforementioned Hopf structure is a Hopf subalgebra. There is also a map $\mathrm{Spec}\, \frac{\mathbb{F}_p [x_0, x_1, \ldots, x_n]}{(x_0^{p^{k+1}},x_1^p, \ldots, x_n^p) } \to \mathbb{G}_a$ corresponding to the element $x_0$ of degree $1$ (see \cref{high}) which makes the former an object of $\mathcal{P}^1_*.$\vspace{2mm}

We also consider the graded algebra $\frac{\mathbb{F}_p [y_0, y_1, \ldots, y_{k}]}{(y_0^p, \ldots, y_{k-1}^p, y_{k}^{p^{n+1}})},$ where $\deg y_i = p^i.$ By quotienting with $y_k^{p}$ we obtain the graded algebra $\frac{\mathbb{F}_p [y_0, y_1, \ldots, y_{k}]}{(y_0^p, \ldots, y_{k-1}^p, y_{k}^{p})}$ which has a Hopf structure coming from $W_{k+1}[F].$ This induces a graded Hopf structure on $\frac{\mathbb{F}_p [y_0, y_1, \ldots, y_{k}]}{(y_0^p, \ldots, y_{k-1}^p, y_{k}^{p^{n+1}})}$ which is uniquely obtained by sending $y_k^{p} \to y_k^p \otimes 1 + 1 \otimes y_k^p$ and requiring that the quotient map $\frac{\mathbb{F}_p [y_0, y_1, \ldots, y_{k}]}{(y_0^p, \ldots, y_{k-1}^p, y_{k}^{p^{n+1}})} \to \frac{\mathbb{F}_p [y_0, y_1, \ldots, y_{k}]}{(y_0^p, \ldots, y_{k-1}^p, y_{k}^{p})}$ is a map of graded Hopf algebras, when the quotient is equipped with the aforementioned Hopf structure. There is also a map $\mathrm{Spec}\, \frac{\mathbb{F}_p [y_0, y_1, \ldots, y_{k}]}{(y_0^p, \ldots, y_{k-1}^p, y_{k}^{p^{n+1}})} \to \mathbb{G}_a$ corresponding to the element $y_0$ of degree $1$ which makes this an object of $\mathcal{P}^1_*.$\vspace{2mm}

The above two objects of $\mathcal{P}^1_*$ are Cartier dual of each other. Indeed, by killing the ideal of elements in degrees $\ge p^{k+1}$ in $\frac{\mathbb{F}_p [x_0, x_1, \ldots, x_n]}{(x_0^{p^{k+1}},x_1^p, \ldots, x_n^p) } $, we obtain the sub Hopf algebra $\mathbb{F}_p [x_0]/x_0^{p^{k+1}}.$ This implies that killing the ideal of elements in degrees $\ge p^{k+1}$ in the Cartier dual $\frac{\mathbb{F}_p [x_0, x_1, \ldots, x_n]}{(x_0^{p^{k+1}},x_1^p, \ldots, x_n^p) }^*$ we get an isomorphism with the graded Hopf algebra $\frac{\mathbb{F}_p [y_0, y_1, \ldots, y_{k}]}{(y_0^p, \ldots, y_{k-1}^p, y_{k}^{p})}.$ Under this isomorphism, $y_k$ is identified with the basis element in degree $p^k$ of $\frac{\mathbb{F}_p [x_0, x_1, \ldots, x_n]}{(x_0^{p^{k+1}},x_1^p, \ldots, x_n^p) }^*.$ Inspecting the comultiplication in $\frac{\mathbb{F}_p [x_0, x_1, \ldots, x_n]}{(x_0^{p^{k+1}},x_1^p, \ldots, x_n^p) }$, we conclude that the powers of the basis element in degree $p^k$ of $\frac{\mathbb{F}_p [x_0, x_1, \ldots, x_n]}{(x_0^{p^{k+1}},x_1^p, \ldots, x_n^p) }^*$ gives a basis element in degrees $i$ for $p^k \le i < p^{k+n+1}.$ This shows that the two objects of $\mathcal{P}^1_*$ mentioned above are indeed dual to each other.
\end{example}{}

\subsection{Deformations of some $\mathbb{G}_a$ and $\mathbb{G}_a^{\mathrm{perf}}$-modules}
In this section, we study deformations of some pointed $\mathbb{G}_a$ (resp. $\mathbb{G}_a^{\mathrm{perf}})$-modules. All deformations are required to be flat over the base. While studying deformations, we would like to make use of certain universal properties. In order to formulate these universal properties, we will need to restrict our attention to a suitable full subcategory of $\mathbb{G}_a\w{--} \mathrm{Mod}_*$ (resp. $\mathbb{G}_a^{\mathrm{perf}}\w{--} \mathrm{Mod}_*$).

\begin{definition}A pointed $\mathbb{G}_a$-module $X$ is said to be \textit{pure} of rank $1$ if the graded Hopf algebra $\Gamma (X, \mathcal{O}_X)$ is free of finite type and the map $X \to \mathbb{G}_a$ induces an isomorphism in degree $1$ at the level of graded algebras of global sections. The full subcategory of such objects inside the category $\mathbb{G}_a\w{--} \mathrm{Mod}_*$ will be called the category of pure rank-$1$ pointed $\mathbb{G}_a$-modules.
\end{definition}{}

\begin{remark}\label{rem5}Since (by using \cref{conn}) the category of pure rank-$1$ pointed $\mathbb{G}_a$-modules is a full subcategory of $\mathcal{P}^1_*$, it follows that $\mathbb{G}_a^*$ is the initial object in the category of pure rank-$1$ pointed $\mathbb{G}_a$-modules. Therefore $\mathbb{G}_a^*$ admits no nontrivial endomorphisms other than the identity (over an arbitrary base ring). Over a $\mathbb{Z}_p$-algebra, $W[F]$ is isomorphic to $\mathbb{G}_a^*$ and thus inherits the same universal property.
\end{remark}{}

\begin{proposition}\label{12}Let $A$ be a $\mathbb{Z}_p$-algebra. Then $W[F]$ has no nontrivial endomorphism as a pointed $\mathbb{G}_a$-module over $A.$
\end{proposition}{}

\begin{proof}This follows from \cref{dri} and \cref{rem5}. \end{proof}{}

\begin{proposition}\label{mad}Let $(A, \mathfrak{m})$ be an Artinian local ring with residue field $k$ which has $\mathrm{char.} \, p>0.$ Then any deformation of the pointed $\mathbb{G}_a$-module $W[F]_k$ over $A$ is uniquely isomorphic to $W[F]_A$.
\end{proposition}{}

\begin{proof}Let $X$ be any deformation of $W[F]_k$ to $A$; in particular, $X$ is flat over $\w{Spec}\, A.$ Then $X$ is necessarily a pure rank-$1$ pointed $\mathbb{G}_a$-module. Since $W[F]_A$ is the initial object in the category of pure rank-$1$ pointed $\mathbb{G}_a$-modules, there is a unique map $W[F]_A \to X$ of pointed $\mathbb{G}_a$-modules. This map is an isomorphism after going modulo $\mathfrak{m}$, and therefore is an isomorphism. The uniqueness follows from \cref{12}. \end{proof}{}

\begin{proposition}\label{mad1}Let $A$ be an $\mathbb{F}_p$-algebra. Then the pointed $\mathbb{G}_a$-module $W_n[F]$ defined over $A$ has no nontrivial endomorphisms.
\end{proposition}

\begin{proof}Follows from \cref{rem7}.\end{proof}{}

\begin{proposition}\label{mad2}
Let $(A, \mathfrak{m})$ be an Artinian local ring over $\mathbb{F}_p$ with residue field $k$. Then any deformation of the pointed $\mathbb{G}_a$-module $W_n[F]_k$ is uniquely isomorphic to $W_n[F]_A.$
\end{proposition}{}

\begin{proof}
Follows from \cref{rem7}.
\end{proof}{}

\begin{definition}[Pure of fractional rank $1$]\label{pure}Let $A$ be a base ring of char. $p>0.$ A pointed $\mathbb{G}_a^{\mathrm{perf}}$-module $X$ over $A$ is said to be \textit{pure of fractional rank 1} if it is isomorphic to $u^* Y$ for some pure rank-$1$ pointed $\mathbb{G}_a$-module $Y.$ The full subcategory of such objects inside $\mathbb{G}_a^{\mathrm{perf}}\w{--}\mathrm{Mod}_*$ will be called the category of pure fractional rank-$1$ pointed $\mathbb{G}_a^{\mathrm{perf}}$-module. This category is also the essential image of the functor $u^*$ restricted to the category of pure rank-$1$ pointed $\mathbb{G}_a$-modules.
\end{definition}{}

\begin{proposition}\label{rem6}The category of pure fractional rank-$1$ pointed $\mathbb{G}_a^{\mathrm{perf}}$ modules has a final object given by $\mathbb{G}_a^{\mathrm{perf}}.$ This category also has an initial object which is given by $u^*W[F].$ 
\end{proposition}

\begin{proof}
This follows from the definition of pure fractional rank-$1$ modules by using \cref{rem5} and the fact that $u^*$ is fully faithful (\cref{pullback}).
\end{proof}{}

\begin{proposition}\label{new3}Let $A$ be an $\mathbb{F}_p$-algebra. Then $u^*W[F]$ has no nontrivial endomorphism as a pointed $\mathbb{G}_a^{\mathrm{perf}}$-module over $A.$
\end{proposition}{}

\begin{proof}This follows from \cref{rem6}.\end{proof}{}

Our next Proposition will deal with deformations of $u^*W[F].$ We point out that using \cref{mad} and \cref{lastprop} one can directly prove \cref{mainthm1} below. However, since the proof of \cref{lastprop} uses the language of stacks, we prefer to record an elementary argument in the case of $u^* W[F].$  Before we begin, we record a lemma.

\begin{lemma}\label{cot}
Let $S= \bigoplus_{i \in \mathbb{N}[1/p]}S_i$ be a perfect, graded $\mathbb{F}_p$-algebra. For a fixed $n \in \mathbb{N}[1/p],$ we consider the ideal $I:= \bigoplus_{i \ge n} S_i.$ Then there is a unique deformation up to unique isomorphism of $S/I$ over $\mathbb{F}_p [\epsilon]/[\epsilon^2]$ which is compatible with the grading.
\end{lemma}{}

\begin{proof}This could also be proven by a graded version of the cotangent complex but we prefer to give a direct proof. Let $B= \bigoplus_i B_i$ be a deformation of $S/I$ compatible with the grading. Since $(S/I)_i = 0$ for $i \ge n$, by going modulo $\epsilon,$ it follows that $B_i = 0$ for $i \ge n.$ Note that, since $S$ is perfect, the natural map $S\to S/I$ lifts uniquely to give a map $f:S[\epsilon]:= S \otimes \mathbb F_p[\epsilon]/\epsilon^2 \to B.$ Further, since $S$ is perfect, this map is automatically graded. Indeed, for a homogeneous element $s \in S \subset S[\epsilon]$, $s^{1/p}$ is also homogeneous and $f(s^{1/p})$ is a homogeneous element if we go modulo $\epsilon.$ By taking $p$-th powers, that implies that $f(s)$ is a homogeneous element. This shows that $f$ is a graded map. Now the map $S_i \to (S/I)_i$ is an isomorphism for $i<n$ and zero for $i \ge n.$ Since $B_i$ is flat over $\mathbb{F}_p[\epsilon]/\epsilon^2,$ by going modulo $\epsilon$ (see \cref{new2}), it follows that $f_i: S[\epsilon]_i \to B_i$ is an isomorphism for $i <n,$ and is necessarily the zero map for $i \ge n$ since $B_i = 0$ for $i \ge n.$ This shows that the kernel of $f$ is $I[\epsilon]:= \bigoplus_{k \ge n} S[\epsilon]_k.$ Now $f$ is surjective, since it is a surjection modulo $\epsilon$. This shows that $B \simeq S/I [\epsilon]$, compatible with the grading. Uniqueness follows from grading and by taking $p$-power roots.\end{proof}{}

\begin{proposition}\label{mainthm1}
The pointed $\mathbb{G}_a^{\mathrm{perf}}$-module $u^*W[F]$ over $\mathbb{F}_p$ has no nontrivial deformation over $\mathbb{F}_p[\epsilon]/\epsilon^2.$
\end{proposition}{}

\begin{proof}
Since the Hopf algebra $B$ underlying $u^*W[F]$ is not graded by nonnegative integers, the theory of Cartier duality breaks down. Indeed, dimension of the piece of degree $1$ in $B\otimes_{\mathbb{F}_p} B$ is infinite and thus does not behave well under duality. \textit{A priori}, we cannot directly apply any of our results above. Our proof will use some lemmas which will ultimately break it down to steps where we are only dealing with finite type Hopf algebras. 

\begin{lemma}\label{ten2} For $n \ge 0,$ the graded group scheme $u^*  W_{n+1}[F]$ over $\mathbb{F}_p$ has no nontrivial deformation over $\mathbb F_p [\epsilon]:= \mathbb{F}_p[\epsilon]/\epsilon^2$ as a graded group scheme. Further, $u^*  W_{n+1}[F] \otimes_{\mathbb{F}_p} \mathbb{F}_p [\epsilon]$ admits a unique endomorphism as a pointed $\mathbb{G}_a^{\mathrm{perf}}$-module.
\end{lemma}{}

\begin{proof}We will break down the proof in a few steps.\vspace{2mm}

\noindent
\textit{Step 1.} We write the graded algebra underlying $u^*  W_{n+1}[F]$ as $C=\frac{\mathbb{F}_p[x_0^{1/p^\infty}, x_1, \ldots, x_n]}{x_i^p}$, where $\deg x_i = p^i.$ This admits a map of graded Hopf algebras $\mathbb{F}_p [x_0 ^{1/p^\infty}] \to \frac{\mathbb{F}_p[x_0^{1/p^\infty}, x_1, \ldots, x_n]}{x_i^p}.$ Let $C'$ be the graded Hopf algebra underlying the deformation of $u^*  W_{n+1}[F].$ Let $C':= \bigoplus_{i \in \mathbb{N}[1/p]} C'_i$ as a graded algebra. By killing the ideal of elements of degree $\ge p$, we obtain a ring $C'_{<p}$ which is a deformation of the graded algebra $\mathbb{F}_p [x_0^{1/p^\infty}]/x_0^p$ which has to be uniquely isomorphic to the trivial deformation by \cref{cot}. Thus, by taking grading into account, we see that $C'$ has to be of the form $$\frac{\mathbb{F}_p [\epsilon][X_0 ^{1/p^\infty}, \ldots, X_n]}{(X_0^p - c_1 \epsilon X_1, X_1^p - c_2 \epsilon X_2, \ldots, X_n^p)} ,$$ where $X_i$ is a lift of $x_i$ of degree $p^i$ and $c_i \in \mathbb F_p.$ The comultiplication sends $X_0^{1/p^m} \to X_0^{1/p^m} \otimes 1 + 1 \otimes X_0^{1/p^m}$ in $C' \otimes C'.$ This shows that there is map of graded Hopf algebras $\mathbb{F}_p [\epsilon][X_0^{1/p^\infty}] \to C' $ which is a deformation of the map $\mathbb{F}_p [x_0^{1/p^\infty}] \to C.$ Thus for $k$ large enough, the graded Hopf algebra map $$\mathbb{F}_p [\epsilon][X_0^{1/p^k}] \to C'_k:= \frac{\mathbb{F}_p [\epsilon][X_0 ^{1/p^k}, \ldots, X_n]}{(X_0^p - c_1 \epsilon X_1, X_1^p - c_2 \epsilon X_2, \ldots, X_n^p)}$$ is firstly a deformation of the graded Hopf algebra map $\mathbb{F}_p [x_0^{1/p^k}] \to C_k:= \frac{\mathbb{F}_p [x_0 ^{1/p^k}, \ldots, x_n]}{ x_i^p}$, and secondly the map of graded Hopf algebras $\mathbb{F}_p [\epsilon][X_0^{1/p^\infty}] \to C' $ is obtained by pulling back the map $\mathbb{F}_p [\epsilon][X_0^{1/p^k}] \to C'_k$ along $\mathbb{F}_p[\epsilon][X_0^{1/p^k}] \to \mathbb{F}_p[\epsilon] [X_0^{1/p^\infty}].$ Thus, to prove \cref{ten2}, it is enough to prove that $\mathbb{F}_p [\epsilon][X_0^{1/p^k}] \to C'_k$ is isomorphic to the trivial deformation of $\mathbb{F}_p [x_0^{1/p^k}] \to C_k.$ For the latter claim, by a shifting of degree, it is enough to prove that the pointed $\mathbb{G}_a$-module $\mathrm{Spec}\, \frac{\mathbb{F}_p [x_0, x_1, \ldots, x_n]}{(x_0^{p^{k+1}},x_1^p, \ldots, x_n^p) } \to \mathbb{G}_a$ has no nontrivial deformations over $\mathbb F_p [\epsilon]$. Here $\deg x_0 =1$ and $\deg x_i = p^{k+i}$ for $i \ge 1.$\vspace{2mm}

\noindent
\textit{Step 2.} In order to prove that the pointed $\mathbb{G}_a$-module $\mathrm{Spec}\, \frac{\mathbb{F}_p [x_0, x_1, \ldots, x_n]}{(x_0^{p^{k+1}},x_1^p, \ldots, x_n^p) } \to \mathbb{G}_a$ has no nontrivial deformations over $\mathbb F_p [\epsilon]$, it is equivalent to prove the same for its Cartier dual, which by \cref{cart} is given by $$\mathrm{Spec}\, \frac{\mathbb{F}_p [x_0, x_1, \ldots, x_{k}]}{(x_0^p, \ldots, x_{k-1}^p, x_{k}^{p^{n+1}})} \to \mathbb{G}_a.$$ We let $D$ denote the Hopf algebra $ \frac{\mathbb{F}_p [x_0, x_1, \ldots, x_{k}]}{(x_0^p, \ldots, x_{k-1}^p, x_{k}^{p^{n+1}})}.$ Then $D/ x_k^p$ is isomorphic to $\Gamma(W_{k+1}[F], \mathcal{O})$ as a graded Hopf algebra. Let $D'$ be a deformation of $D.$ Then $D'$ as a graded algebra is of the form $$\frac{\mathbb{F}_p [\epsilon] [X_0, X_1, \ldots, X_{k}]}{(X_0^p- c_1 \epsilon X_1, \ldots, X_{k-1}^p- c_{k}\epsilon X_k, X_{k}^{p^{n+1}})}$$ where $\deg X_i= p^i$ and $X_i \in D'$ is chosen to be a lift of $x_i \in D.$ Now $D'/X_k^p$ is a deformation of the graded Hopf algebra underlying $W_{k+1}[F]$, but the latter has no nontrivial deformations by \cref{mad2}. Thus we obtain an isomorphism $(D/x_k^p) [\epsilon]=W_{k+1}[F][\epsilon] \to D'/X_k^p$ of graded Hopf algebras (which is also compatible with the pointed $\mathbb{G}_a$-module structure). This lifts uniquely to a map of graded algebras $D [\epsilon] \to D'$ as such a map is uniquely determined by the image of $x_0, \ldots, x_k$ and they have a unique homogeneous lift. One also needs to check that $x_i ^p$ is sent to zero for $0 <i \le k-1$ and $x_k^{p^{n+1}}$ is sent to zero which also follows from grading arguments. This map by construction is an isomorphism on the level of graded algebras and it would be enough to check that it is a Hopf algebra map, i.e., we need to prove that the maps $D[\epsilon] \to D' \to D'\otimes D'$ and $D[\epsilon]\to D[\epsilon] \otimes D[\epsilon] \to D' \otimes D'$ agree. For that, we only need to check that the images of $x_0, \ldots, x_k$ agree. It is known that they agree modulo the ideal $(X_k^p \otimes 1, 1 \otimes X_k^p)$ and thus by grading they are actually the same. The map we constructed is also compatible with the pointed $\mathbb{G}_a$-module structure.\vspace{2mm}

The statement about endomorphisms as pointed $\mathbb{G}_a^{\mathrm{perf}}$-module follows since $W_{n+1}[F] [\epsilon]$ has no nontrivial endomorphism as pointed $\mathbb{G}_a$-module by \cref{mad1}.\end{proof}{}

\begin{lemma}\label{ten1}The graded algebra underlying any deformation of $u^*W[F]$ as a graded group scheme is isomorphic to $$\frac{\mathbb{F}_p[\epsilon][x_0^{1/p^\infty}, x_1, \ldots]}{x_i^p}.$$
\end{lemma}{}

\begin{proof}We write the graded algebra underlying $u^*  W[F]$ as $B=\frac{\mathbb{F}_p[x_0^{1/p^\infty}, x_1, \ldots]}{x_i^p}$, where $\deg x_i = p^i.$ Similar to the proof of \cref{ten2}, it follows that the graded algebra underlying global sections of a deformation of $u^*W[F]$ is isomorphic to $$B'=\frac{\mathbb{F}_p[\epsilon][X_0^{1/p^\infty}, X_1, \ldots]}{(X_0^p - c_1 \epsilon X_1, X_1^p - c_2\epsilon X_2, \ldots)}, $$where $X_i$ is taken to be a lift of $x_i$ of degree $p^i$ and $c_i \in \mathbb F_p.$ Our goal is to prove that $c_i = 0$ for all $i$. Killing the ideal of elements of degree $\ge p^{n+1}$, we obtain the graded algebra $$B'_{<p^{n+1}}:=\frac{\mathbb{F}_p[\epsilon][X_0^{1/p^\infty}, X_1, \ldots, X_n]}{(X_0^p - c_1 \epsilon X_1,\ldots, X_n^p)}.$$ Further, this has a Hopf structure: this follows from the observation that the comultiplication in $B'$ sends $X_n^p \mapsto X_n^p \otimes 1 + 1 \otimes X_n^p.$ Now $\mathrm{Spec}\,B'_{<p^{n+1}}$ as a graded group scheme is a deformation of $u^*  W_{n+1}[F]$, and thus by \cref{ten2} must be uniquely isomorphic to the trivial deformation, which implies $c_i = 0$ for $0<i<n+1.$ Since $n$ was arbitrary, we are done.
\end{proof}{}

Now we note that the algebra $B$ underlying $u^*W[F]$ has the property that for every $n \ge 1$, the elements of degree $< p^n$ form a subalgebra denoted as $\tau_{< n}B$ which is the graded Hopf algebra underlying $u^*W_n[F].$ By \cref{ten1}, it follows that elements of degree $<p^n$ in $B'$ also form a subalgebra $\tau_{<n}B'$ which has the structure of a graded Hopf algebra. Moreover, we observe that $\tau_{<n}B'$ is a deformation of $\tau_{<n}B$ and thus by \cref{ten2}, there is a unique isomorphism of graded Hopf algebras 
$\tau_{<n}B [\epsilon] \to \tau_{<n} B'$ that sends $x_0 \mapsto X_0.$ By taking colimit over $n$, we have constructed an isomorphism $B[\epsilon] \to B'$ of graded Hopf algebras. Finally, we recall that, as noted in \cref{gradingdet}, the grading determines the $\mathbb{G}_a^{\w{perf}}$-module structure. This proves the proposition.\end{proof}{}

\begin{remark}\label{chmv981}
We point out that the statements of \cref{cot}, \cref{mainthm1} and their proofs remain valid with $\mathbb{F}_p$ replaced by any perfect field of characteristic $p>0.$
\end{remark}{}

\begin{remark}One can also approach \cref{mainthm1} by first showing that any deformation of $u^*W[F]$ must be a pullback of a deformation of $W[F].$ This is essentially carried out in a purely formal way in \cref{sec5.2} by using the connection with de Rham cohomology. In \cref{lastprop}, we prove a generalization by using similar ideas.
\end{remark}{}

\newpage

\section{Construction of functors using $\mathbb{G}_a$ and $\mathbb{G}_a^{\mathrm{perf}}$-modules }\label{section3} In this section, our ultimate goal is to create functors from $\mathrm{QRSP} \to \mathrm{Alg}_A$ via an ``unwinding" process using the data of a pointed $\mathbb{G}_a^{\mathrm{perf}}$-module (\textit{cf}. \cref{sec3.3}). This construction will be done using a closely related variant: using the data of a pointed $\mathbb{G}_a^{\mathrm{perf}}$-module, we can ``unwind" it to construct a functor $\mathscr{P}I \to \mathrm{Alg}_A$, where $\mathscr{P}I$ denotes the category with objects $(B,I)$ where $B$ is a perfect ring and $I$ is an ideal. This is carried out in \cref{sec3.2}. Further, this construction for $\mathscr{P}I$ has a closely related variant for the category $\mathfrak{C}_A$ consisting of objects $(B,I)$ where $B$ is an $A$-algebra and $I$ is an ideal of $B.$ Given a pointed $\mathbb{G}_a$-module, we can unwind it to create a functor $\mathfrak{C}_A \to \mathrm{Alg}_A.$ We will study this construction first in \cref{sec3.1}. Below we note an example which aims to explain an analogue of the unwinding construction in a simpler case.

\begin{example}\label{introex}Let $\mathcal{C}$ be a category with all colimits. Let $c \in \mathcal{C}^{\w{op}}$ be a commutative $\mathbb{F}_p$-algebra object in the category $\mathcal{C}^{\w{op}}.$ Therefore, the functor $\mathrm{Hom}_{\mathcal{C}}(c, \cdot)$ is naturally valued in commutative $\mathbb{F}_p$-algebras, when $c$ is viewed as an object of $\mathcal{C}$. Let $\overline{\mathrm{Poly}}_{\mathbb F_p}$ denote the category of not necessarily finitely generated polynomial $\mathbb{F}_p$-algebras. We will construct a functor ${\mathrm{Un}}_{c}: \overline{\mathrm{Poly}}_{\mathbb F_p} \to \mathcal{C}$ which we may call the unwinding of $c.$ For $B \in \overline{\mathrm{Poly}}_{\mathbb F_p},$ we define ${\mathrm{Un}}_c (B) \in \mathcal{C}$ such that we have a natural functorial bijection $\mathrm{Hom}_{\mathcal{C}}({\mathrm{Un}}_c (B), d) \simeq \mathrm{Hom}_{\mathrm{Alg}_{\mathbb{F}_p}}(B, \mathrm{Hom}_{\mathcal{C}}(c, d))$ for $d \in \mathcal{C}.$ This maybe computed \textit{functorially} as a colimit by using the natural diagram $\mathbb{F}_p[\mathbb{F}_p[B]]\substack{\rightarrow \\ \rightarrow} \mathbb{F}_p[B]$ in $\w{Alg}_{\mathbb{F}_p}$ whose coequalizer is $B$ (this coequalizer diagram depends functorially on the ring $B$ and is an instance of the general ``bar construction" in the context of monads). More precisely, the above coequalizer diagram induces a diagram $$\coprod_{\mathbb{F}_p[B]} c\,\, \substack{\longrightarrow \\ \longrightarrow}\coprod_{B} c$$ in $\mathcal{C}$ whose coequalizer is $\w{Un}_{c}(B);$ here the coproducts of $c$ are taken over the sets underlying the rings $B$ and $\mathbb{F}_p[B].$ Note that by construction, we have ${\mathrm{Un}}_c (\mathbb{F}_p[x]) \simeq c$. Also, by construction, ${\mathrm{Un}}_c: \overline{\mathrm{Poly}}_{\mathbb F_p} \to \mathcal{C}$ preserves coproducts. This whole discussion carries over even if $\mathbb{F}_p$ is replaced by an arbitrary commutative ring.\vspace{2mm}

This construction shows that given an object of the category $\mathcal{C}$ with appropriate extra structure, one can unwind it to create a functor from $\overline{\mathrm{Poly}}_{\mathbb F_p} \to \mathcal{C}.$ In this section, our goal is to develop a similar formalism for the categories $\mathfrak{C}_A, \mathscr{P}I$ and $\mathrm{QRSP}$ which would be useful to us in \cref{sec4} and \cref{section5}.
\end{example}{}

\subsection{Tensoring a module with a module scheme}In this section we record a construction which ``tensors" a module with a module scheme and gives an algebra as an output. In a category $\mathcal{C}$ with all coproducts, one can make sense of tensoring an object $c \in \mathcal{C}$ with a set $S$, denoted as $c \otimes S$, which has the property that we have a natural isomorphism $\mathrm{Hom}_{\mathcal{C}}(c \otimes S, d) \simeq \mathrm{Hom}_{\mathrm{Sets}}(S, \mathrm{Hom}_{\mathcal{C}}(c, d))$ for $d \in \mathcal{C}.$ In this case, $c \otimes S$ is the coproduct $\coprod_{S}c.$ Below, we carry out an analogue of this construction.

\begin{construction}\label{cons1}Let $X= \mathrm{Spec}\, B$ be an $R$-module scheme over $A$ (\cref{def1}). In this situation, we have a functor from $\mathrm{Alg}_A \to \mathrm{Mod}_R$ which sends an $A$ algebra $S$ to $X(S):=\mathrm{Hom}_A(\mathrm{Spec}\, S, X).$ This functor is limit preserving and by the adjoint functor theorem, it has a left adjoint which will be denoted by $\mathscr{T}_X(\,.\,): \mathrm{Mod}_R \to \mathrm{Alg}_A.$ In other words, we have the following natural isomorphism 
$$ \mathrm{Hom}_{\mathrm{Alg}_A}(\mathscr{T}_X(M), S) \simeq \mathrm{Hom}_{\mathrm{Mod}_R} (M, X(S)).$$
\end{construction}{}

\begin{remark}\label{bday28}
Note that for an $m \in M,$ we have a natural map $\mathrm{Hom}_{\mathrm{Mod}_R} (M, X(S)) \to X(S)$ obtained by evaluation at $m \in M.$ This induces a map denoted as $\w{ev}_m: \Gamma (X, \mathcal{O}_X) \to \mathscr{T}_{X}(M)$ that will be useful in \cref{con}.
\end{remark}{}

\begin{remark}\label{7up}
Let us also describe an explicit way to construct the algebra $\mathscr{T}_X(M)$ for an $R$-module $M.$ Considering $M$ as a set, first we take the coproduct of the algebra $B= \Gamma(X, \mathcal{O}_X)$ indexed over $M.$ We will write this as $\coprod_{M} B.$ By the universal property of the coproduct, for each $m \in M,$ we have a map which we will write as $m: B \to \coprod_{M} B.$ We also have a map $B \to B \otimes_A B$ which is the comultiplication map and a map $r: B \to B$ coming from the $R$-module action of $\mathrm{Spec}\, B$ for $r \in R$. Then $\mathscr{T}_X(M)$ is the coequalizer of the following diagram indexed by $(R\times M) \coprod (M \times M).$

\end{remark}{}
\begin{center}
\begin{tikzcd}
B \arrow[r, "rm"] \arrow[d, "r"] & \coprod_{M}B & B \arrow[l, "m+n"'] \arrow[d]         \\
B \arrow[ru, "m"']               &              & B\otimes_A B \arrow[lu, "m\otimes n"]
\end{tikzcd}
\end{center}{}

\begin{proposition}\label{sharp}In the above set up, we have \vspace{2mm}

\noindent
$\mathrm{1}$. $\mathrm{colim} \mathscr{T}_X(M_i) \simeq \mathscr{T}_X(\mathrm{colim}\, M_i).$\vspace{1mm}

\noindent
$\mathrm{2}$. $\mathscr{T}_X(M) \otimes_A \mathscr{T}_X(N) \simeq  \mathscr{T}_X (M \oplus N).$\vspace{1mm}

\noindent
$\mathrm{3}$. If $M$ is a free $R$-module of rank $1$, then $\mathscr{T}_X(M) \simeq \Gamma(X, \mathcal{O}_X).$ 
\end{proposition}{}

\begin{proof}Follows from the construction of $\mathscr{T}_X(M).$ \end{proof}{}

\begin{example}\label{7up22}
When $X$ is the zero group scheme $\w{Spec}\, A$ over $A$, thought of as an $R$-module scheme over $A$ and $M$ is any $R$-module, we have $\mathscr{T}_X (M) \simeq A$ as an $A$-algebra.
\end{example}{}

\begin{example}\label{geo}In the case where $X = \mathbb{G}_a= \mathrm{Spec}\, A[x]$ viewed as an $A$-module scheme over $A,$ given any $A$-module $M$, we have $\mathscr{T}_{\mathbb{G}_a}(M) \simeq \mathrm{Sym}_A(M)$ as an $A$-algebra. This follows from the universal property discussed in \cref{cons1} above.
\end{example}{}

\begin{example}\label{example1}We mention an example that will be particularly important to us. Let $\mathbb{G}_a^*$ be the Cartier dual of $\mathbb{G}_a$ viewed as an $A$-module scheme over $A.$ For an $A$-module $\mathscr{T}_{\mathbb{G}_a^*}(M) \simeq \Gamma_A(M).$ This essentially follows from \cite[Appendix 2]{BO78} by observing that for an $A$ algebra $R$, there is a natural isomorphism of $A$-modules $\mathbb{G}_a^* (R) \simeq \mathrm{exp}(R)$ where $\mathrm{exp}(R)$ denote the elements $f(x) \in 1 + x R[x]$ satisfying $f(x+y) = f(x)f(y)$ which forms an abelian group by multiplication of power series which further has an $R$-module structure given by $r \cdot f(x) := f(rx)$ (\textit{cf.}~\cref{final}).\end{example}{}

\begin{remark}We point out that for a fixed $R$-module $M$, the association $X \mapsto \mathscr{T}_X(M)$ is a contravariant functor. Further, in the set up of \cref{cons1}, it follows that $\w{Spec}\, \mathscr{T}_X(M)$ can be naturally equipped with the structure of a $R$-module scheme over $A$ for any fixed $X.$

\end{remark}

\begin{remark}\label{7up3}
If $X$ is a $\mathbb{G}_a$-module over $A$ and $M$ is any $A$-module, then $\w{Spec}\, \mathscr{T}_X(M)$ is naturally equipped with the structure of a $\mathbb{G}_a$-module over $A.$ Further, for any $m \in M,$ the map $X \to \w{Spec}\, \mathscr{T}_X(M)$ induced by $\w{ev}_m$ from \cref{bday28} is a map of $\mathbb{G}_a$-modules. Both of these statements follow from the functorial descriptions provided in \cref{cons1}. Therefore, the map $\w{ev}_m: \Gamma (X, \mathcal{O}_X) \to \mathscr{T}_{X}(M) $ is a map of graded Hopf algebras. From \cref{7up}, it follows that the elements in the images of the maps $\w{ev}_m$ for all $m \in M$ generate $\mathscr{T}_X(M)$ as an $A$-algebra.
\end{remark}{}

\subsection{Unwinding pointed $\mathbb{G}_a$-modules}\label{sec3.1}

\begin{notation}\label{bday}We fix an arbitrary base ring $A$ as before. Let $\mathfrak{C}_A$ denote the category of pairs $(B,I)$ where $B$ is an $A$-algebra and $I$ is an ideal of $B.$ Morphisms in $\mathfrak{C}_A$ between $(B, I)\to (B',I')$ are defined as  $A$-algebra maps $B \to B'$ that maps $I$ inside $I'.$ 
\end{notation}{}

\begin{construction}[Unwinding]\label{cons2} Let $\mathbb{G}_a\w{--}\mathrm{Mod}_*$ denote the category of pointed $\mathbb{G}_a$-modules over $A$. We will construct a (contravariant) functor $${\mathrm{Un}}: \mathbb{G}_a\w{--}\mathrm{Mod}_* \to  \mathrm{Fun}(\mathfrak{C}_A, \mathrm{Alg}_A).$$ We will say that ${\mathrm{Un}}(X)$ is the functor obtained by \textit{unwinding} the pointed $\mathbb{G}_a$-module $X.$ To describe the construction, we fix an $X\in \mathbb{G}_a\w{--}\mathrm{Mod}_*$. Given $(B,I) \in \mathfrak{C}_A$, we obtain a diagram $X_B \to \mathbb{G}_{a,B}$ of $B$-module schemes by base changing to $B.$ Now the ideal $I$ can be regarded as a $B$-module and thus by applying \cref{cons1}, we obtain a map $\mathscr{T}_{\mathbb{G}_{a,B}}(I) \to \mathscr{T}_{X_B} (I).$ By \cref{geo} we get a map of $B$-algebras $\mathrm{Sym}_B(I) \to \mathscr{T}_{X_B}(I).$ Since $I$ is an ideal of $B$, there are natural maps $\mathrm{Sym}_B(I) \to B \to \mathrm{Sym}_B(I) \to \mathscr{T}_{X_B}(I).$ Thus by composing we get another map $\mathrm{Sym}_B(I) \to \mathscr{T}_{X_B}(I).$ We denote the coequalizer of these two maps $$\mathrm{Sym}_B(I)\,\, \substack{\longrightarrow \\ \longrightarrow}\,\,\mathscr{T}_{X_B}(I)$$ by $\mathrm{Env}_X (B,I).$ This is naturally a $B$-algebra. Now we define ${\mathrm{Un}}(X) (B,I) := \mathrm{Env}_X(B,I).$
\end{construction}{}

\begin{remark}[Unwinding via universal property]\label{univ1}It will be very useful for us to have a description of the universal property of $\mathrm{Env}_X (B,I)$ as a $B$-algebra. For a $B$-algebra $S$, we note that $\mathbb{G}_{a,B}(S) = S$ is naturally a $B$-module. In fact, there is a map $B \to S$ of $B$-modules giving a natural map $I \subset B \to S.$ This gives an element $*\in \mathrm{Hom}_{B\mathrm{-Mod}}(I,S).$ Therefore, we obtain two maps $ \mathrm{Hom}_{B\mathrm{-Mod}}(I, X(S))\,\, \substack{\longrightarrow \\ \longrightarrow}\,\,\mathrm{Hom}_{B\mathrm{-Mod}}(I,S).$ Here one of the maps (of sets) sends everything to $*$ and the other one is the map induced by the data of the point $X \to \mathbb{G}_a.$ We note that by \cref{cons2}, we have $$\mathrm{Hom}_{B\mathrm{-Alg}}(\mathrm{Env}_X(B,I),S) \simeq \mathrm{Eq}(\mathrm{Hom}_{B\mathrm{-Mod}}(I, X(S))\,\, \substack{\longrightarrow \\ \longrightarrow}\,\,\mathrm{Hom}_{B\mathrm{-Mod}}(I,S)) .$$
\end{remark}{}

\begin{remark}\label{presheaf}We note that there is a natural functor $\mathfrak{G}: \mathfrak{C}_A \to \mathrm{Alg}_A$ given by $\mathfrak{G}(B,I) = B.$ From \cref{univ1}, we see that there is a natural isomorphism $\mathfrak{G} \simeq {\mathrm{Un}}(\mathbb{G}_a).$ Let $\mathrm{Fun}(\mathfrak{C}_A, \mathrm{Alg}_A)_{\mathfrak{G}/}$ denote the category of functors $F: \mathfrak{C}_A \to \mathrm{Alg}_A$ equipped with a natural transformation $\mathfrak{G} \to F.$ The morphisms are required to be compatible with this data. It follows that in \cref{cons2}, we actually produced a (contravariant) functor $${\mathrm{Un}}: \mathbb{G}_a\w{--}\mathrm{Mod}_* \to  \mathrm{Fun}(\mathfrak{C}_A, \mathrm{Alg}_A)_{\mathfrak{G}/}.$$
\end{remark}

\begin{example}\label{move4}The functor $\mathfrak{C}_A \to \mathrm{Alg}_A$ given by sending $(B,I) \to B/I$ is the unwinding of the pointed $\mathbb{G}_a$-module corresponding to the zero section $\w{Spec}\, A \to \mathbb{G}_a.$
\end{example}{}

\begin{example}[Divided power envelope via unwinding]\label{example2}We let the base ring $A$ be $\mathbb{F}_p$ for simplicity. The functor ${\mathrm{Un}}(\mathbb{G}_a^*): \mathfrak{C}_{\mathbb{F}_p} \to \mathrm{Alg}_{\mathbb{F}_p}$ takes a pair $(B,I)$ to the divided power envelope $D_B(I).$ In order to see this, we compute $\mathrm{Env}_{\mathbb{G}_a^*}(B,I)$ following \cref{cons2}. This is computed as the coequalizer of two maps $$\mathrm{Sym}_B(I)\,\, \substack{\longrightarrow \\ \longrightarrow}\,\,\mathscr{T}_{{\mathbb{G}_a ^*}_B}(I).$$ We note that by \cref{example1}, $\mathscr{T}_{{\mathbb{G}_a ^*}_B}(I) \simeq \Gamma_B(I).$ Thus the claim $\mathrm{Env}_{\mathbb{G}_a^*}(B,I) \simeq D_B(I)$ follows from \cite[Thm. 3.19]{BO78} which says that $D_B(I) = \Gamma_B(I)/ J$ where $J$ is the ideal generated by $\varphi(x) - x$ for all $x \in I$ and $\varphi: I \to \Gamma_1(I)$ is the natural map. Since $\mathbb{G}_a^* \simeq W[F]$ over $\mathbb{F}_p$, we also have $\mathrm{Env}_{W[F]} (B,I) \simeq D_B(I).$ 
\end{example}{}

\begin{remark}\label{env}
Let $X\in \mathbb{G}_a \w{--}\mathrm{Mod}_*$ and let $B$ be an $A$-algebra and $f$ be a non-zero divisor in $B.$ Then we can explicitly describe $\mathrm{Env}_X(B,f).$ We note that the ideal $I$ generated by $f$ in this case is free of rank $1$ and thus there is an isomorphism $\mathscr{T}_{X_B}(I) \simeq \Gamma(X_B, \mathcal{O}_{X_B})  \simeq \Gamma(X, \mathcal{O}_X) \otimes_A B$ by \cref{sharp}. Now $\mathrm{Env}_X(B, f)$ is the quotient $\frac{\Gamma(X, \mathcal{O}_X)\otimes_A B}{(t \otimes 1 - 1 \otimes f)}$ where $t$ is the image of $x$ under the map $A[x] \to \Gamma(X, \mathcal{O}_X)$ corresponding to the data of the point i.e., the map $X \to \mathbb{G}_a.$
\end{remark}

\begin{remark}For any $X \in \mathbb{G}_a \w{--}\mathrm{Mod}_*$, it follows that ${\mathrm{Un}}(X) (B,0) = \mathrm{Env}_X (B,0) \simeq B.$ This follows since $\mathscr{T}_{X_B}(0) \simeq B$ and therefore $\mathrm{Env}_X(B,0)$ is a coequalizer of two $B$-algebra maps $B\substack{\rightarrow \\ \rightarrow}B $ that coincide. Thus the natural map $\mathfrak{G} \to \mathrm{Un}(X)$ induces isomorphism restricted to the full subcategory spanned by objects of the form $(B,0).$
\end{remark}{}

\begin{proposition}\label{env1}Let $X \in \mathbb{G}_a \w{--}\mathrm{Mod}_*.$ Then $\mathrm{Env}_X (B[x], x) \simeq \Gamma(X, \mathcal{O}_X) \otimes_A B$ as a $B$-algebra. The map $\mathrm{Env}_X(B[x], 0) \to \mathrm{Env}(B[x],x)$ identifies with the map $B[x] \to \Gamma(X_B, \mathcal{O}_{X_B})$ coming from the data of the point $X \to \mathbb{G}_a.$
\end{proposition}{}

\begin{proof}Using \cref{env}, we have $\mathrm{Env}_X(B[x], x) \simeq \frac{\Gamma(X, \mathcal{O}_X)\otimes_A B[x]}{t \otimes 1 - 1 \otimes x}$ as $B[x]$-algebras from which the proposition follows. \end{proof}{}

Given that there is a way to unwind the data of a pointed $\mathbb{G}_a$-module $X$ and obtain a functor ${\mathrm{Un}}(X):\mathfrak{C}_A \to \mathrm{Alg}_A$, it is natural to ask if this is reversible, i.e., if there is a functor $r$ from $\mathrm{Fun}(\mathfrak{C}_A, \mathrm{Alg}_A)_{\mathfrak{G}/}$ to $\mathbb{G}_a\w{--}\mathrm{Mod}_*$ such that applying $r$ to ${\mathrm{Un}}(X)$ recovers the pointed $\mathbb{G}_a$-module $X.$ There are multiple problems in achieving this as discussed below.\vspace{2mm}

Firstly, defining the functor $r$ is not possible unless we impose some conditions on the functor $F \in \mathrm{Fun}(\mathfrak{C}_A, \mathrm{Alg}_A)_{\mathfrak{G}/}.$ Indeed, for every $A$-algebra $B$, we can look at the $B$-algebra $F (B[x], x).$ This has a $B$-action, however $F (B[x], x)$ might not be a Hopf algebra. The functor $F$ needs to preserve some pushout diagrams for that to happen; this is taken into account in \cref{contain}. Under these special assumptions on $F$ it is indeed possible to define a functor $r$ as desired. The functor $r$ is defined below in \cref{restriction}.\vspace{2mm}

However, we note that not every pointed $\mathbb{G}_a$-module can appear as image under the functor $r$ of an $F \in \mathrm{Fun}(\mathfrak{C}_A, \mathrm{Alg}_A)$. Indeed, we have the following commutative diagram in $\mathfrak{C}_A.$
\vspace{-2mm}
\begin{center}
\begin{equation}\label{qp}
\begin{tikzcd}
{(A[x], x) \otimes (A[x],x)}           &  &  & {(A[x], x)\otimes (A[x], 0)} \arrow[lll] \\
{(A[x], 0)\otimes (A[x], x)} \arrow[u] &  &  & {(A[x],x)} \arrow["x \to x \otimes x"', u ] \arrow[lll, "x \otimes x\, \leftarrow x "]            
\end{tikzcd}    
\end{equation}{}
\end{center}{}
Applying $F$ to the above diagram would impose extra conditions on the pointed $\mathbb{G}_a$-module obtained from $F$ which need not be satisfied by every pointed $\mathbb{G}_a$-module. Thus it is impossible to recover $X$ by using the functor $r$ from ${\mathrm{Un}}(X)$ unless it satisfies some special conditions to begin with. To account for this, we are naturally led to the notion of a ``quasi-ideal" in $\mathbb{G}_a$ due to Drinfeld \cite[Section 3.1]{Dri20}.

\begin{definition}[Drinfeld]\label{drin}A pointed $\mathbb{G}_a$-module $X$ with the data of the point denoted as $d: X \to \mathbb{G}_a$ will be called a \textit{quasi-ideal in} $\mathbb{G}_a$ if the following diagram commutes.

\begin{center}
\begin{tikzcd}
X \times X \arrow[r, "\mathrm{id}\times d "] \arrow[d, "d \times \mathrm{id} "] & X \times \mathbb G_a \arrow[d, "\mathrm{action}"] \\
\mathbb G_a \times X \arrow[r, "\mathrm{action}"] & X                             
\end{tikzcd}
\end{center}{}
By writing $X = \mathrm{Spec}\, B$ for a graded Hopf algebra $B$ and $t \in B$ for the fixed choice of the element in degree $1$ corresponding to the data of the point, we note that $X$ is a quasi-ideal if and only if $b \otimes t^{\deg b} = t^{\deg b} \otimes b$ in $B \otimes B$ for every homogeneous $b \in B.$ We let $\mathrm{QID}\w{--}\mathbb{G}_a$ denote the full subcategory of quasi-ideals in $\mathbb{G}_a$ inside $\mathbb{G}_a \w{--}\mathrm{Mod}_*.$
\end{definition}{}

\begin{remark}\label{qid1}Using the inclusion $\mathrm{QID}\w{--}\mathbb{G}_a \to \mathbb{G}_a \w{--}\mathrm{Mod}_*$ of categories and \cref{cons2} we obtain a (contravariant) functor still denoted us ${\mathrm{Un}}: \mathrm{QID}\w{--}\mathbb{G}_a \to \mathrm{Fun}(\mathfrak{C}_A, \mathrm{Alg}_A)_{\mathfrak{G}/}.$ We will later see that this functor is fully faithful.
\end{remark}{}

\begin{proposition}\label{koszul}Let $B$ be an $A$-algebra. Let $(f_j)_{j \in \mathscr{J}}$ be a collection of non-zero divisors in $B$ and let $I$ be the ideal generated by them. Let $F$ be the free module over $B$ spanned by $x_j$ for $j \in \mathscr{J}.$ We assume that the $B$-module map  $F \to I$ that sends $x_i \to f_i$ has kernel generated by $(f_ix_j - f_j x_i)$ for $i,j \in \mathscr{J}$. Let $X$ be a quasi-deal in $\mathbb{G}_a.$ Then the natural map $$\coprod_{j \in \mathscr{J}} \mathrm{Env}_X (B, f_j) \to \mathrm{Env}_X(B,I)$$ is an isomorphism. Here the coproduct is taken in the category of $B$-algebras.
\end{proposition}{}

\begin{proof}
By \cref{univ1}, $\mathrm{Env}_X (B,I)$ corepresents the functor $H_1$ that sends $$S \mapsto \mathrm{Eq}(\mathrm{Hom}_B(I, X(S)) \,\, \substack{\longrightarrow \\ \longrightarrow}\,\, \mathrm{Hom}_B(I, S)), $$ where one of the maps come from composing with $X(S) \to S$ and the other one maps everything to the element in $\mathrm{Hom}_B(I,S)$ corresponding to $I \subset B \to S.$ Given such an element in the equalizer, by precomposing with the surjection $F \to I$, we obtain a natural transformation from $H_1$ to the functor $H_2$ that sends $S \to \mathrm{Eq}(\mathrm{Hom}_B(F, X(S)) \,\, \substack{\longrightarrow \\ \longrightarrow}\,\, \mathrm{Hom}_B(F, S)),$ where, as before, one of the maps come from $X(S) \to S$ and the other one from sending everything to the $B$-linear map $F \to I \subset B \to S.$ Since $F \to I$ is a surjection, it follows that the map $H_1(S) \to H_2(S)$ is injective. Below we check that this map is also surjective. \vspace{2mm}

To show surjectivity, we need to show that any map $u: F \to X(S)$ which fits into the commutative diagram
\begin{center}
\begin{tikzcd}
F \arrow[r] \arrow[d] & X(S) \arrow[d] \\
I \arrow[r]           & S   \end{tikzcd}
\end{center}{}
factors through $F \to I.$ Let $u_i:=u(x_i) \in X(S).$ The map $X \to \mathbb{G}_a$ induces a map on $S$-valued points that we denote as $d: X(S) \to S.$ Since $X$ is a quasi-ideal in $\mathbb{G}_a$, it follows that $d u_j\cdot u_i = d u_i\cdot u_j.$ By the commutativity of the diagram this implies $f_j \cdot u_i = f_i \cdot u_j$ in $X(S)$ or equivalently $u (f_j x_i - f_i x_j )=0$, i.e., the map indeed factors through $I.$ Now the proposition follows by using \cref{env} and noting that $\coprod_{j \in \mathscr{J}} \mathrm{Env}_X (B, f_j)$ corepresents the functor $H_2.$ \end{proof}{}

\begin{proposition}\label{coprod}Let $B$ be an $A$-algebra. Let $S$ be any set. Let $(B[S], (S))$ denote the coproduct of $(B[x], x)$ over $S.$ Let $X$ be a quasi-ideal in $\mathbb{G}_a.$ Then the natural map $$\coprod_S \mathrm{Env}_X (B[x], x) \to \mathrm{Env}_X (B[S], (S))$$ is an isomorphism of $B$-algebras. Here the coproduct is taken in the category of $B$-algebras. In particular, we obtain an isomorphism $(\coprod_S \Gamma(X, \mathcal{O}_X)) \otimes_A B \simeq \mathrm{Env}_X (B[S], S)$ of $B$-algebras.
\end{proposition}{}

\begin{proof}This follows from \cref{koszul} and \cref{env1}. Indeed, let us consider the polynomial algebra $B[S]$ on the set $S.$ An element $s \in S$ can be thought of as an indeterminate $s \in B[S]$; the ideal $(S)$ is the ideal generated by these indeterminates. Due to the fact that in the polynomial ring $B[t_1, \ldots, t_r],$ the elements $t_1,\ldots,t_r$ are non-zero divisors and form a (Koszul) regular sequence \cite[Tag 062D]{SP}, one can apply \cref{koszul} in this situation to calculate $\mathrm{Env}_{X}(B[S], (S))$; this gives us an isomorphism $$\coprod_{s \in S} \mathrm{Env}_{X}(B[S], s) \simeq \mathrm{Env}_{X}(B[S], (S)).$$Here the coproduct is being taken in the category of $B[S]$-algebras. For an element $s \in S,$ let $(S \setminus s) \subset S$ denote the complement of $s.$ Then $B[S] \simeq B[(S \setminus s)] \otimes_{B} B[s],$ where $B[s]$ is the polynomial ring in the single indeterminate $s.$ Therefore, $\mathrm{Env}_{X}(B[S], s) \simeq \mathrm{Env}_{X}(B[(S \setminus s)]\otimes_B B[s], (1 \otimes s)) \simeq \Gamma(X_B, \mathcal{O}) \otimes_{B} B[(S \setminus s)],$ as $B[(S \setminus s)]$-algebra. Here, the last isomorphism follows from \cref{env1}. Note that since $X$ is a quasi-ideal in $\mathbb{G}_a,$ we have a map $X \to \mathbb{G}_a$ which gives an element $t \in \Gamma(X_B, \mathcal{O}_X)$ by passing to the map induced on the ring of global sections. Further, by \cref{env1}, the $B[S] \simeq B[(S \setminus s)] \otimes_B B[s]$-algebra structure on $\Gamma(X_B, \mathcal{O}) \otimes_{B} B[(S \setminus s)]$ is induced by the map of $B[(S \setminus s)]$-algebras $ B[(S \setminus s)] \otimes_{B} B[s] \to \Gamma(X_B, \mathcal{O}) \otimes_{B} B[(S \setminus s)]$ that sends $s \mapsto t\otimes 1.$ Therefore, 
$$\coprod_{s \in S} \mathrm{Env}_{X}(B[S], s) \simeq \coprod_{s \in S} \Gamma(X_B, \mathcal{O}) \otimes_{B} B[(S \setminus s)]$$ in the category of $B[S]$-algebras. However, one observes that the right hand side is isomorphic to $\coprod_{s \in S} \Gamma(X_B, \mathcal{O}),$ where the coproduct is taken in the category of $B$-algebras; the $B[S]$-algebra structure is given by the map $B[S] \to \coprod_{s \in S} \Gamma(X_B, \mathcal{O})$ which is obtained by taking coproduct of the map $B[x] \to \Gamma(X_B, \mathcal{O})$ that sends $ x \mapsto t$ over the set $S$ in the category of $B$-algebras. By \cref{env1}, we have $\coprod_{s \in S} \mathrm{Env}_X (B[x], x) \simeq \coprod_{s \in S} \Gamma(X_B, \mathcal{O})$, which finishes the proof.
\end{proof}{}

\begin{remark}
The natural maps appearing in \cref{koszul} or \cref{coprod} exist for any pointed $\mathbb{G}_a$-module $X.$ The fact that these maps are isomorphisms (in either of the propositions) implies that $X$ satisfies the property of being a quasi-ideal in $\mathbb{G}_a$. This follows from \cref{drin} and functoriality of $\mathrm{Un}(X)$ applied to the diagram in \cref{qp}. We thank the referee for pointing this out.
\end{remark}{}

\begin{definition}\label{contain}Let $F \in \mathrm{Fun}(\mathfrak{C}_A, \mathrm{Alg}_A)_{\mathfrak{G}/}$ (\cref{presheaf}) be a functor which satisfies the following conditions.\vspace{2mm}

\noindent
$\mathrm{1}.$ The natural map $\mathfrak{G}(B,0) \to F(B,0)$ is an isomorphism for every $A$-algebra $B.$\vspace{1mm}

\noindent
$\mathrm{2}.$ The natural map $F((B[x],x)) \otimes_B F((B[x], x)) \to F(B[x]\otimes_B B[x], (x\otimes1, 1\otimes x))$ is an isomorphism.\vspace{1mm}

\noindent
$\mathrm{3}.$ The natural map $F(A[x], x) \otimes_A B \to F(B[x], x)$ is an isomorphism.\vspace{2mm}

We denote the full subcategory of such functors inside $\mathrm{Fun}(\mathfrak{C}_A, \mathrm{Alg}_A)_{\mathfrak{G}/}$ as $\mathrm{Fun}(\mathfrak{C}_A, \mathrm{Alg}_A)^{\otimes}_{\mathfrak{G}/}.$
\end{definition}{}

\begin{remark}\label{acci}
We note that the unwinding functor ${\mathrm{QID}\w{--}\mathbb{G}_a} ^{\mathrm{op}} \to \mathrm{Fun}(\mathfrak{C}_A, \mathrm{Alg}_A)_{\mathfrak{G}/}$ maps a quasi-ideal inside the full subcategory 
$\mathrm{Fun}(\mathfrak{C}_A, \mathrm{Alg}_A)^{\otimes}_{\mathfrak{G}/}.$ Indeed, this follows from \cref{env1} and \cref{coprod}. Therefore, the unwinding functor factors to give a functor still denoted as $\mathrm{Un}: {\mathrm{QID}\w{--}\mathbb{G}_a} ^{\mathrm{op}} \to \mathrm{Fun}(\mathfrak{C}_A, \mathrm{Alg}_A)^{\otimes}_{\mathfrak{G}/}.$
\end{remark}{}

\begin{proposition}\label{restriction}Let $F \in \mathrm{Fun}(\mathfrak{C}_A, \mathrm{Alg}_A)_{\mathfrak{G}/}^\otimes.$ For every $A$-algebra $B$, $\mathrm{Spec}\,F(B[x],x)$ naturally has the structure of a $B$-module scheme. Thus we obtain a (contravariant) functor $r: \mathrm{Fun}(\mathfrak{C}_A, \mathrm{Alg}_A)_{\mathfrak{G}/}^\otimes \to \mathrm{QID}\w{--} \mathbb{G}_a.$
\end{proposition}{}

\begin{proof}
We note that $(B[x],x)$ is a cogroup object of $\mathfrak{C}_{A _{(B,0)/}}$. Therefore, it follows from definitions that $\w{Spec}\, F(B[x],x)$ is a group scheme. The $B$-action on $\w{Spec}\, F(B[x],x)$ is given by functoriality along the arrows $(B[x],x) \to (B[x],x)$ given by $x \mapsto b x$ for $b \in B.$ Therefore, $\w{Spec}\, F(B[x],x)$ is indeed naturally a $B$-module scheme over $B.$ \cref{detai} implies that varying this data over all $A$-algebras $B$ provides us with a $\mathbb{G}_a$-module. Further, functoriality along the arrows $(B[x], 0) \to (B[x], x)$ equips this $\mathbb{G}_a$-module with the structure of a pointed $\mathbb{G}_a$-module. To see that it is a quasi-ideal in $\mathbb{G}_a,$ we use functoriality along the commutative diagram in \cref{qp}.
\end{proof}{}

\begin{proposition}\label{satu1}The functor  $r:\mathrm{Fun}(\mathfrak{C}_A, \mathrm{Alg}_A)^{\otimes}_{\mathfrak{G}/}\to {\mathrm{QID}\w{--}\mathbb{G}_a} ^{\mathrm{op}}$ has a left adjoint given by ${\mathrm{Un}}$ from \cref{acci}.
\end{proposition}{}

\begin{proof}Let $F \in \mathrm{Fun}(\mathfrak{C}_A, \mathrm{Alg}_A)^{\otimes}_{\mathfrak{G}/}$ and $X \in {\mathrm{QID}\w{--}\mathbb{G}_a}^{\mathrm{op}}.$ We prove that there is a natural bijection $$\mathrm{Hom}({\mathrm{Un}}(X), F) \simeq \mathrm{Hom}(X, rF).$$ Applying $r$ and noting that $r {\mathrm{Un}}(X) \simeq X$ by \cref{env1} and \cref{coprod} provides a map from the left hand side to the right hand side which will be called $s$. We will construct a map the other way. We will first construct a map ${\mathrm{Un}} (rF) \to F.$ To do so, we note that there is an isomorphism $$\varphi: \mathrm{Hom}_{(B,0)}((B,I), \cdot) \simeq \mathrm{Eq} \left(\mathrm{Hom}_{B\w{-Mod}}(I, \mathrm{Hom}_{(B,0)}((B[x], x), \cdot)) \,\, \substack{\longrightarrow \\ \longrightarrow}\,\ \mathrm{Hom}_{B\w{-Mod}}(I, \mathrm{Hom}_{(B,0)}((B[x], 0), \cdot)) \right)  $$ in $\mathrm{Psh}(\mathfrak{C}_{A _{(B,0)/}}^{\mathrm{op}}).$ Here on the right hand side, one of the maps is induced by the map $(B[x], 0) \to (B[x], x)$ in $\mathfrak{C}_A$ and the other map is obtained by sending everything to the element in $\mathrm{Hom}_{B\w{-Mod}}(I, \mathrm{Hom}_{(B,0)}((B[x],0), \cdot))$ corresponding to the map induced by the inclusion $I \subset B$ and the fact that $\mathrm{Hom}_{(B,0)} ((B[x],0), \cdot )$ is naturally valued in $B$-algebras.\vspace{2mm}

We note that $F$ induces a map $F^{\mathrm{op}}: \mathrm{Psh}(\mathfrak{C}_{A _{(B,0)/}}^{\mathrm{op}}) \to \mathrm{Psh}(\mathrm{Alg}_B^{\mathrm{op}}).$ Applying $F^{\mathrm{op}}$ to the above isomorphism $\varphi$ and noting that $F^{\mathrm{op}}\mathrm{Hom}_{(B,0)}((B,I), \cdot)\simeq \mathrm{Hom}_{B}(F(B,I), \cdot) $ we obtain a diagram 
$$\mathrm{Hom}_{B}(F(B,I), \cdot) \to  \left(F^{\mathrm{op}}\mathrm{Hom}_{B\w{-Mod}}(I, \mathrm{Hom}_{(B,0)}((B[x], x), \cdot)) \,\, \substack{\longrightarrow \\ \longrightarrow}\,\ F^{\mathrm{op}} \mathrm{Hom}_{B\w{-Mod}}(I, \mathrm{Hom}_{(B,0)}((B[x], 0), \cdot)) \right)$$ in $\mathrm{Psh}(\mathrm{Alg}_B^ {\mathrm{op}}).$

\begin{lemma}The following diagram in $\mathrm{Psh}(\mathrm{Alg}_B^ {\mathrm{op}})$ commutes.  

\begin{center}
\begin{tikzcd}
F^{\mathrm{op}}\mathrm{Hom}_{B\mathrm{-Mod}}(I, \mathrm{Hom}_{(B,0)}((B[x], x), \cdot)) \arrow[d] \arrow[rr,shift left=.45ex,] \ar[rr,shift right=.45ex,swap,] &  & F^{\mathrm{op}} \mathrm{Hom}_{B\mathrm{-Mod}}(I, \mathrm{Hom}_{(B,0)}((B[x], 0), \cdot)) \arrow[d] \\
\mathrm{Hom}_{B\mathrm{-Mod}}(I, \mathrm{Hom}_{B}(F(B[x], x), \cdot)) \arrow[rr,shift left=.45ex,] \ar[rr,shift right=.45ex,swap,]       &  &  \mathrm{Hom}_{B\mathrm{-Mod}}(I, \mathrm{Hom}_{B}(F(B[x], 0), \cdot))         
\end{tikzcd}
\end{center}{}
\end{lemma}{}

\begin{proof}This follows from universal property and assumption 2 in \cref{contain} on $F$. We will show how to construct a map $F^{\mathrm{op}}\mathrm{Hom}_{B\mathrm{-Mod}}(I, \mathrm{Hom}_{(B,0)}((B[x], x), \cdot)) \to \mathrm{Hom}_{B\mathrm{-Mod}}(I, \mathrm{Hom}_{B}(F(B[x], x), \cdot)).$ We note that $I$ is the coequalizer of a diagram $\left(F^{B \times I \coprod I \times I} \,\, \substack{\longrightarrow \\ \longrightarrow}\,\  F^{I}\right)$ of free $B$-modules where the first map sends the basis elements $x_{(b,i)} \to x_{bi}$ and $x_{(i,i')} \to x_{i+i'}$ and the second map sends $x_{(b,i)} \to b x_{i}$ and $x_{(i,i')} \to x_{i}+ x_{i'}.$ Therefore, $\mathrm{Hom}_{B\mathrm{-Mod}}(I, \mathrm{Hom}_{(B,0)}((B[x],x), \cdot))$ is the equalizer of the two maps $$ \prod_{I}\mathrm{Hom}_{(B,0)}((B[x],x), \cdot)\,\, \substack{\longrightarrow \\ \longrightarrow}\,\ \prod_{B\times I \coprod I \times I} \mathrm{Hom}_{(B,0)} ((B[x], x), \cdot).$$ One of the maps corresponds to the map determined by $x_{(b,i)} \to x_{bi}$ and $x_{(i,i')}\to x_{i+ i'}.$ The other map is induced by combining the maps $$\prod_{I}\mathrm{Hom}_{(B,0)}((B[x],x), \cdot) \to \mathrm{Hom}_{(B,0)}((B[x], x), \cdot)\to  \mathrm{Hom}_{(B,0)}((B[x], x), \cdot), $$ where the first map is projection from $i$-th factor and the second map is obtained by using $(B[x],x) \to (B[x], x)$ that sends $x \to bx$ for $b \in B$; and $$ \prod_I \mathrm{Hom}_{(B,0)}((B[x], x), \cdot) \to  \mathrm{Hom}_{(B,0)}((B[x], x), \cdot)\times  \mathrm{Hom}_{(B,0)}((B[x], x), \cdot) \to  \mathrm{Hom}_{(B,0)}((B[x], x), \cdot)$$ where the first map is projection from $(i,i')$-th factor and the last map uses the map $(B[x], x) \to  (B[x], x)\otimes (B[x], x)$ given by $x \to x\otimes 1 + 1 \otimes x.$ Now universal property of limits and assumption 2 in \cref{contain} constructs the desired map $$F^{\mathrm{op}}\mathrm{Hom}_{B\mathrm{-Mod}}(I, \mathrm{Hom}_{(B,0)}((B[x], x), \cdot)) \to \mathrm{Hom}_{B\mathrm{-Mod}}(I, \mathrm{Hom}_{B}(F(B[x], x), \cdot))$$ and the naturality guarantees the commutativity of the diagram in the Lemma. \end{proof}{}

Thus we obtain a map $$\mathrm{Hom}_{B}(F(B,I), \cdot) \to  \mathrm{Eq}\left(\mathrm{Hom}_{B\mathrm{-Mod}}(I, \mathrm{Hom}_{B}(F(B[x], x), \cdot)) \,\, \substack{\longrightarrow \\ \longrightarrow}\,\  \mathrm{Hom}_{B\mathrm{-Mod}}(I, \mathrm{Hom}_{B}(F(B[x], 0), \cdot)) \right)$$ in $\mathrm{Psh}(\mathrm{Alg}_B^ {\mathrm{op}})$. Using \cref{univ1}, we note that the right hand side is corepresented by $\mathrm{Env}_{rF}(B,I)$. Thus we obtain a natural map $\mathrm{Env}_{rF}(B,I) \to F(B,I).$ This provides the map ${\mathrm{Un}}(rF) \to F$ that we wanted. Now given a map $X \to rF$ in ${\mathrm{QID}\w{--}\mathbb{G}_a} ^{\mathrm{op}}$, we obtain a map ${\mathrm{Un}}(X) \to {\mathrm{Un}}(rF) \to F.$ This gives a map from $\mathrm{Hom}(X, rF)$ to $\mathrm{Hom}({\mathrm{Un}}(X), F)$ which will be called $t$. By \cref{env1}, it follows that $st$ is identity. In order to show that $ts$ is identity, it will be sufficient to show that if there are two natural transformations $U,V: {\mathrm{Un}}(X) \to F$ that are mapped to the same element by $s$ then $U$ and $V$ are the same natural transformation. Note that we always have a commutative diagram 

\begin{center}
\begin{tikzcd}
{{\mathrm{Un}}(r {\mathrm{Un}}(X))} \arrow[d] \arrow[rr, "\simeq"] &  & { {\mathrm{Un}}(X)} \arrow[d, "U"'] \arrow[d, "V"] \\
{{\mathrm{Un}}(rF)} \arrow[rr]                                   &  & { F }                                          
\end{tikzcd}
\end{center}{}
Since the upper horizontal arrow is an isomorphism, the above diagram shows that $U$ and $V$ are the same natural transformation as desired.\end{proof}{}

\begin{proposition}\label{high11}The functor $${\mathrm{Un}}: {\mathrm{QID}\w{--}\mathbb{G}_a}^{\mathrm{op}} \to \mathrm{Fun}(\mathfrak{C}_A, \mathrm{Alg}_A)_{\mathfrak{G}/}$$ is fully faithful.
\end{proposition}{}

\begin{proof}Follows from \cref{satu1} since $r {\mathrm{Un}} (X) \simeq X$.\end{proof}{}

\subsection{Unwinding pointed $\mathbb{G}_a^{\mathrm{perf}}$-modules I}\label{sec3.2}In this section, we will record an analogue of the construction from previous section for pointed $\mathbb{G}_a^{\mathrm{perf}}$-modules. In order to do that, some modifications are needed. As is the case with $\mathbb{G}_a^{\mathrm{perf}}$-modules, we work with a fixed prime $p.$

\begin{notation}
Let $\mathscr{P}I$ denote the category of pairs $(B,I)$ where $B$ is a perfect ring and $I$ is an ideal. Morphisms are defined to be maps $(B,I) \to (B', I')$ where $B \to B'$ is a ring homomorphism such that $I$ is mapped inside $I'.$ Let $A$ be a fixed Artinian local ring with residue field $\mathbb{F}_p.$ Let $\mathbb{G}_a^{\mathrm{perf}} \w{--}\mathrm{Mod}_*$ denote the category of pointed $\mathbb{G}_a^{\mathrm{perf}}$-modules over $A.$
\end{notation}{}

\begin{construction}[Unwinding]\label{cons4}We will construct a (contravariant) functor $${\mathrm{Un}}: \mathbb{G}_a^{\mathrm{perf}} \w{--}\mathrm{Mod}_* \to \mathrm{Fun}(\mathscr{P}I, \mathrm{Alg}_A).$$ We will say that ${\mathrm{Un}}(X)$ is the functor obtained by \textit{unwinding} the pointed $\mathbb{G}_a^{\mathrm{perf}}$-module $X.$ To describe the construction, we fix an $X \in \mathbb{G}_a^{\mathrm{perf}} \w{--}\mathrm{Mod}_*.$ Given $(B,I) \in \mathscr{P}I,$ we obtain a diagram $X_B \to \mathbb{G}_{a,B} ^\mathrm{perf}$ of $B$-module schemes over $W_A(B)$ by \cref{perfect}. Now the ideal $I$ can be regarded as a $B$-module and thus by applying \cref{cons1}, we get a map $\mathscr{T}_{\mathbb{G}_{a,B}^\mathrm{perf}}(I) \to \mathscr{T}_{X_B}(I).$ Since $I \subset B,$ by the universal property of the construction in \cref{cons1}, we obtain a map $\mathscr{T}_{\mathbb{G}_{a,B}^\mathrm{perf}}(I) \to W_A(B).$ By composition, we get a map $\mathscr{T}_{\mathbb{G}_{a,B}^\mathrm{perf}}(I) \to W_A(B) \to \mathscr{T}_{\mathbb{G}_{a,B}^\mathrm{perf}}(I) \to \mathscr{T}_{X_B}(I).$ Therefore, we now have two maps $$\mathscr{T}_{\mathbb{G}_{a,B}^\mathrm{perf}}(I) \,\, \substack{\longrightarrow \\ \longrightarrow}\,\,\mathscr{T}_{X_B}(I).$$ We denote the coequalizer of the above diagram by $\mathrm{Env}_X(B,I)$ which is naturally a $W_A(B)$-algebra. Now we define ${\mathrm{Un}}(X)(B,I) :=\mathrm{Env}_X(B,I).$
\end{construction}{}

\begin{remark}[Unwinding via universal property]\label{univ2}We describe the universal property of $\mathrm{Env}_X (B,I)$ as a $W_A(B)$-algebra. For a $W_A(B)$-algebra $S,$ we note that $\mathbb{G}_{a,B}^\mathrm{perf} (S)= S^\flat$ is naturally a $B$-module. In fact, there is a map $B \to S^\flat$ of $B$-modules giving a natural map $I \to S^\flat.$ This gives an element $* \in \mathrm{Hom}_B (I, S^\flat).$ By \cref{perfect}, $X$ can be regarded as a $B$-module scheme over $W_A(B)$. Therefore, we obtain two maps $\mathrm{Hom}_{B}(I, X(S)) \,\, \substack{\longrightarrow \\ \longrightarrow}\,\, \mathrm{Hom}_B (I, S^\flat).$ Here one of the maps (of sets) sends everything to $*$ and the other one is the map induced by the data of the point $X \to \mathbb{G}_a ^{\mathrm{perf}}.$ We note that by \cref{cons4}, we have $$\mathrm{Hom}_{W_A(B)}(\mathrm{Env}_X(B,I), S) \simeq \mathrm{Eq}(\mathrm{Hom}_{B}(I, X(S)) \,\, \substack{\longrightarrow \\ \longrightarrow}\,\, \mathrm{Hom}_B (I, S^\flat)). $$ 
\end{remark}{}

\begin{remark}\label{dinner}We note that there is a natural functor $\mathfrak{G}: \mathscr{P}I \to \mathrm{Alg}_A$ given by $(B,I) \mapsto W_A(B).$ From \cref{univ2}, it follows that $\mathfrak{G} \simeq  {\mathrm{Un}}(\mathbb{G}_a^{\mathrm{perf}}).$ Let $\mathrm{Fun}(\mathscr{P}I, \mathrm{Alg}_A)_{\mathfrak{G}/}$ denote the category of functors $F: \mathscr{P}I \to \mathrm{Alg}_A$ equipped with a natural transformation $\mathfrak{G} \to F.$ The morphisms are required to be compatible with this data. It follows that in \cref{cons4}, we actually produced a (contravariant) functor $${\mathrm{Un}}: \mathbb{G}_a^{\mathrm{perf}} \w{--}\mathrm{Mod}_* \to \mathrm{Fun}(\mathscr{P}I, \mathrm{Alg}_A)_{\mathfrak{G}/}.$$
\end{remark}{}

\begin{remark}\label{kle}For any pointed $\mathbb{G}_a^{\mathrm{perf}}$-module $X$, we have the isomorphism ${\mathrm{Un}}(X) (B,0) \simeq W_A(B).$ This follows from the universal property of the unwinding construction. Thus the natural map $\mathfrak{G} \to \mathrm{Un}(X)$ induces isomorphism restricted to the full subcategory of $\mathscr{P}I $ spanned by objects of the form $(B,0)$ for a perfect ring $B.$
\end{remark}{}

\begin{remark}\label{satu4}We point out that unless $A = \mathbb{F}_p,$ sending $(B,I) \mapsto B/I$ is in general not an object of $\mathrm{Fun}(\mathscr{P}I, \mathrm{Alg}_A)_{\mathfrak{G}/}$ that can be obtained via applying the unwinding functor. When we are working over $A = \mathbb{F}_p,$ the functor $(B,I)\to B/I$ is naturally isomorphic to ${\mathrm{Un}} (\alpha^{\natural})$, where $\alpha^{\natural}$ is the pointed $\mathbb{G}_a^{\mathrm{perf}}$-module as described in \cref{alpha}. Further, in this case, the functor $(B,I) \mapsto (B/I)_{\mathrm{perf}}:= \mathrm{colim}_{x \mapsto x^p} (B/I)$ is the unwinding of the pointed $\mathbb{G}_a^{\mathrm{perf}}$-module corresponding to zero.
\end{remark}{}

\begin{remark}\label{gax1}
Let $A \to A'$ be a map of Artinian local rings. Let $(B,I) \in \mathscr{P}I$. We note that $W_{A'}(B) \simeq W_{A}(B) \otimes_{A} A'.$ Further, if $X$ is a pointed $\mathbb{G}_a^{\w{perf}}$-module over $A,$ then $X':=X \times_{\w{Spec}\, A} \w{Spec}\, A'$ is naturally a pointed $\mathbb{G}_a^{\w{perf}}$-module over $A'.$ From the universal property described in \cref{univ2}, it follows that $$\mathrm{Env}_{X'} (B,I) \simeq \mathrm{Env}_{X}(B,I)\otimes_{W_A(B)} W_{A'} (B),$$ as a $W_A(B)$-algebra. This implies that $\mathrm{Un}(X')(B,I) \simeq \w{Un} (X) (B,I) \otimes_{A} A'.$  
\end{remark}{}

Having discussed the unwinding functor for pointed $\mathbb{G}_a$ and $\mathbb{G}_a^{\mathrm{perf}}$-modules, let us record a statement regarding their compatibility. For simplicity, we work over the fixed base ring $\mathbb{F}_p.$ In this case, one has the functor $u^*: \mathbb{G}_a \w{--}\mathrm{Mod}_* \to \mathbb{G}_a^{\mathrm{perf}} \w{--}\mathrm{Mod}_*$ from \cref{pullback}. We prove the following

\begin{proposition}\label{move5}Let $X$ be a pointed $\mathbb{G}_a$-module over $\mathbb{F}_p.$ Let $(B,I) \in \mathscr{P}I.$ Then there is a natural isomorphism
$$\mathrm{Env}_X (B,I) \simeq \mathrm{Env}_{u^*X}(B,I).$$
\end{proposition}{}

\begin{proof}This follows by using the universal properties of $\mathrm{Env}_X(B,I)$ and $\mathrm{Env}_{u^*X}(B,I)$ as a $B$-algebra from \cref{univ1} and \cref{univ2} and the following pullback diagram of $B$-modules for a given $B$-algebra $S$ from \cref{pullback}.

\begin{center}
\begin{tikzcd}
u^*X(S) \arrow[r] \arrow[d] & S^\flat \arrow[d] \\
X(S) \arrow[r]            & S \end{tikzcd}
\end{center}{}
\end{proof}{}

\begin{example}In \cref{satu4} we stated that ${\mathrm{Un}}(\alpha^{\natural})$ is the functor that sends $(B,I) \in \mathscr{P}I$ to $B/I.$ This also follows from \cref{move4} and \cref{move5}.
\end{example}{}

\begin{example}\label{example3}We note that ${\mathrm{Un}}(u^*W[F])$ is the functor that sends $(B,I) \in \mathscr{P}I$ to $D_B(I).$ This follows from \cref{example2}.
\end{example}{}

\begin{proposition}\label{lunch}
Let $X \in \mathbb{G}_a^{\mathrm{perf}}\w{--} \mathrm{Mod}_*.$ Let $B$ be a perfect ring. Then $\mathrm{Env}_X (B[x^{1/p^\infty}],x) \simeq \Gamma(X, \mathcal{O}_X) \otimes_A W_A(B)$ as a $W_A(B)$-algebra. The map $\mathrm{Env}_X(B[x^{1/p^\infty}],0) \to \mathrm{Env}_X (B[x^{1/p^\infty}],x)  $ identifies with the map $W_A(B)[x^{1/p^\infty}] \to \Gamma (X, \mathcal{O}_X) \otimes_A W_A(B)$ corresponding to the data of the point $X \to \mathbb{G}_a^{\mathrm{perf}}.$
\end{proposition}{}

\begin{proof}We compute using \cref{cons4}. Since $x$ is a non-zero divisor in $B[x^{1/p^\infty}],$ the ideal it generates is free of rank $1.$ Therefore, by using \cref{sharp}, we see that $\mathrm{Env}_X (B[x^{1/p^\infty}],x)$ is computed as a coequalizer of the following diagram $$W_A(B[x^{1/p^\infty}])[y^{1/p^\infty}]\,\, \substack{\longrightarrow \\ \longrightarrow}\,\, \Gamma(X, \mathcal{O}_X) \otimes_A W_A (B[x^{1/p^\infty}]).$$ Here one of the map corresponds to the map $W_A(B)[x^{1/p^\infty}][y^{1/p^\infty}] \to W_A(B) [x^{1/p^\infty}]$ that sends $y ^{1/p^n} \to x^{1/p^n}$ for all $n$ and is a $W_A(B)[x^{1/p^\infty}]$-algebra map. The other map corresponds to the data of the point, i.e., obtained by base changing a map $A[y^{1/p^\infty}] \to \Gamma(X, \mathcal{O}_X).$ Taking the coequalizer we get the desired conclusion. \end{proof}{}

\begin{definition}\label{quasi-ideal}A pointed $\mathbb{G}_a^{\mathrm{perf}}$-module $X$ with the data of the point denoted as $d: X \to \mathbb{G}_a^{\mathrm{perf}}$ will be called a \textit{quasi-ideal in} $\mathbb{G}_a^{\mathrm{perf}}$ if the following diagram commutes.

\begin{center}
\begin{tikzcd}
X \times X \arrow[r, "\mathrm{id}\times d "] \arrow[d, "d \times \mathrm{id} "] & X \times \mathbb G_a^{\mathrm{perf}} \arrow[d, "\mathrm{action}"] \\
\mathbb G_a^{\mathrm{perf}} \times X \arrow[r, "\mathrm{action}"] & X                             
\end{tikzcd}
\end{center}{}
We will denote the category of quasi-ideals in $\mathbb{G}_a^{\mathrm{perf}}$ by $\mathrm{QID}\w{--}\mathbb{G}_a^{\mathrm{perf}}$ which is the full subcategory spanned by quasi-ideals in $\mathbb{G}_a^{\mathrm{perf}}$ inside $\mathbb{G}_a^{\mathrm{perf}}\w{--}\mathrm{Mod}_*.$
\end{definition}{}

\begin{remark}\label{qu}Using the inclusion $\mathrm{QID}\w{--}\mathbb{G}_a^{\mathrm{perf}} \to \mathbb{G}_a^{\mathrm{perf}}\w{--}\mathrm{Mod}_*$ of categories, we can define a (contravariant) functor $\mathrm{QID}\w{--}\mathbb{G}_a^{\mathrm{perf}}  \to \mathrm{Fun}(\mathscr{P}I, \mathrm{Alg}_A)_{\mathfrak{G}/} $ which will again be called unwinding and will be denoted by ${\mathrm{Un}}$.
\end{remark}{}

\begin{proposition}\label{lunch1}Let $B$ be a perfect ring. Let $(f_j)_{j \in \mathscr{J}}$ be a collection of non-zero divisors in $B$ and let $I$ be the ideal generated by them. Let $F$ be the free module over $B$ spanned by $x_j$ for $j \in \mathscr{J}.$ We assume that the $B$-module map $F \to I$ that sends $x_i \to f_i$ has kernel generated by $(f_i x_j - f_j x_i)$ for $i,j \in \mathscr{J}.$ Let $X$ be a quasi-ideal in $\mathbb{G}_a^{\mathrm{perf}}.$ Then the natural map $$\coprod_{j \in \mathscr{J}}\mathrm{Env}_X (B, f_j) \to \mathrm{Env}_X (B,I)$$ is an isomorphism. Here the coproduct is taken in the category of $W_A(B)$-algebras.
\end{proposition}{}

\begin{proof}Using \cref{univ2}, this follows in a way similar to the proof of \cref{koszul}.\end{proof}{}

\begin{proposition}\label{xx}Let $B$ be a perfect ring. Let $S$ be any set. Let $(B[S^{1/p^\infty}], (S))$ denote the coproduct of $(B[x^{1/p^\infty}], x)$ over $S.$ Let $X$ be a quasi-ideal in $\mathbb{G}_a^{\mathrm{perf}}.$ Then the natural map $$\coprod_S \mathrm{Env}_X (B[x^{1/p^\infty}], x) \to \mathrm{Env}_X (B[S^{1/p^\infty}], S)$$ is an isomorphism of $W_A(B)$-algebras. Here the coproduct is taken in the category of $W_A(B)$-algebras. In particular, we obtain an isomorphism $\coprod_S \Gamma(X, \mathcal{O}_X) \otimes_A W_A(B) \simeq \mathrm{Env}_S (B[S^{1/p^\infty}], S)$ of $W_A(B)$-algebras.

\end{proposition}{}
\begin{proof}This follows from \cref{lunch} and \cref{lunch1} in a way similar to the proof of \cref{coprod}. \end{proof}{}

\begin{definition}Let $F \in \mathrm{Fun}(\mathscr{P}I, \mathrm{Alg}_A)_{\mathfrak{G}/}$ be a functor which satisfies the following conditions.\vspace{2mm}

\noindent
\text{1.} The natural map $\mathfrak{G}(B,0) \to F(B,0)$ is an isomorphism for every perfect ring $B.$\vspace{1mm}

\noindent
\text{2.} The natural map $F(B[x^{1/p^\infty}], x) \otimes_{W_A(B)} F(B[x^{1/p^\infty}], x) \to F (B[x^{1/p^\infty}] \otimes_B B[x^{1/p^\infty}], (x \otimes 1 , 1 \otimes x ))$ is an isomorphism.\vspace{1mm}

\noindent
\text{3.} The natural map $F (\mathbb{F}_p [x^{1/p^\infty}], x) \otimes_A W_A(B) \to F (B[x^{1/p^\infty}],x)$ is an isomorphism.\vspace{2mm}

We denote the full subcategory of such functors inside $\mathrm{Fun}(\mathscr{P}I, \mathrm{Alg}_A)_{\mathfrak{G}/}$ as $\mathrm{Fun}(\mathscr{P}I, \mathrm{Alg}_A)_{\mathfrak{G}/}^{\otimes}.$
\end{definition}{}

\begin{remark}\label{acci1}
We note that the unwinding functor ${\mathrm{QID}\w{--} \mathbb{G}_a^{\mathrm{perf}}}^{\mathrm{op}} \to \mathrm{Fun}(\mathscr{P}I, \mathrm{Alg}_A)_{\mathfrak{G}/}$ maps a quasi-ideal inside the full subcategory 
$\mathrm{Fun}(\mathscr{P}I, \mathrm{Alg}_A)_{\mathfrak{G}/}^{\otimes}.$ Indeed, this follows from \cref{lunch} and \cref{xx}. Therefore, the unwinding functor factors to give a functor still denoted as $\mathrm{Un}: {\mathrm{QID}\w{--} \mathbb{G}_a^{\mathrm{perf}}}^{\mathrm{op}} \to \mathrm{Fun}(\mathscr{P}I, \mathrm{Alg}_A)_{\mathfrak{G}/}^{\otimes}.$
\end{remark}{}

\begin{proposition}\label{wp}Let $F \in \mathrm{Fun}(\mathscr{P}I, \mathrm{Alg}_A)_{\mathfrak{G}/}^{\otimes}.$ For every perfect ring $B,$ $\mathrm{Spec}\, F (B[x^{1/p^\infty}], x)$ is naturally a $B$-module scheme over $W_A(B).$ Consequently, we have a (contravariant) functor $r: \mathrm{Fun}(\mathscr{P}I, \mathrm{Alg}_A)_{\mathfrak{G}/}^{\otimes} \to \mathrm{QID}\w{--} \mathbb{G}_a^{\mathrm{perf}}.$
\end{proposition}{}

\begin{proof}We note that $(B[x^{1/p^\infty}], x)$ is a cogroup object of $\mathscr{P}I_{(B,0)/}$. Therefore, it follows from definitions that $\mathrm{Spec}\, F (B[x^{1/p^\infty}], x)$ is a group scheme. The $B$-action on $\mathrm{Spec}\, F (B[x^{1/p^\infty}], x)$ is given by functoriality along the arrows $(B[x^{1/p^\infty}], x) \to (B[x^{1/p^\infty}], x)$ given by $x^{1/p^n} \to b^{1/p^n} x^{1/p^n}$ for all $n \ge 1.$ Therefore, $\mathrm{Spec}\, F (B[x^{1/p^\infty}], x)$ is indeed naturally a $B$-module scheme over $W_A(B).$ \cref{perfect} implies that varying this data over all perfect rings $B$ provides us a $\mathbb{G}_a^{\mathrm{perf}}$-module. Further, functoriality along the maps $(B[x^{1/p^\infty}],0) \to (B[x^{1/p^\infty}],x)$ equips this $\mathbb{G}_a^{\mathrm{perf}}$-module with the structure of a pointed $\mathbb{G}_a^{\mathrm{perf}}$-module. To see that it is a quasi-ideal in $\mathbb{G}_a ^{\mathrm{perf}}$, we use functoriality along the following commutative diagram in $\mathscr{P}I.$

\begin{center}
\begin{tikzcd}
{(\mathbb{F}_p[x^{1/p^\infty}], x) \otimes (\mathbb{F}_p[x^{1/p^\infty}],x)}           &  &  & {(\mathbb{F}_p[x^{1/p^\infty}], x)\otimes (\mathbb{F}_p[x^{1/p^\infty}], 0)} \arrow[lll] \\
{(\mathbb{F}_p[x^{1/p^\infty}], 0)\otimes (\mathbb{F}_p[x^{1/p^\infty}], x)} \arrow[u] &  &  & {(\mathbb{F}_p[x^{1/p^\infty}],x)} \arrow[u, "x^{\frac{1}{p^n}} \to x^{\frac{1}{p^n}} \otimes x^\frac{1}{p^n}   "'] \arrow[lll, "x^{\frac{1}{p^n}}\otimes x^\frac{1}{p^n}  \leftarrow x^{\frac{1}{p^n}}   "]            
\end{tikzcd}    
\end{center}
\end{proof}

\begin{proposition}\label{satu2}The functor $r: \mathrm{Fun}(\mathscr{P}I, \mathrm{Alg}_A)_{\mathfrak{G}/}^{\otimes} \to {\mathrm{QID}\w{--} \mathbb{G}_a^{\mathrm{perf}}}^{\mathrm{op}}$ has a left adjoint given by ${\mathrm{Un}}$ from \cref{acci1}.
\end{proposition}{}

\begin{proof}The proof follows in a similar way to the proof of \cref{satu1} once we note the following statement about the category $\mathscr{P}I$. Let $B$ be a perfect ring so that $(B,0)$ is an object of $\mathscr{P}I.$ Then there is an isomorphism 
$$\varphi: \mathrm{Hom}_{(B,0)}((B,I), \cdot) \simeq \mathrm{Eq} \left(\mathrm{Hom}_{B\w{-Mod}}(I, \mathrm{Hom}_{(B,0)}((B[x^{1/p^\infty}], x), \cdot)) \,\, \substack{\longrightarrow \\ \longrightarrow}\,\ \mathrm{Hom}_{B\w{-Mod}}(I, \mathrm{Hom}_{(B,0)}((B[x^{1/p^\infty}], 0), \cdot)) \right) $$
in $\mathrm{Psh}(\mathscr{P}I_{(B,0)/}^{\mathrm{op}}).$ Here one of the arrows is induced by the map $(B[x^{1/p^\infty}],0) \to (B[x^{1/p^\infty}], x)$ and the other map is obtained by sending everything to the element of $\mathrm{Hom}_{B\mathrm{-Mod}}(I, \mathrm{Hom}_{(B,0)}((B[x^{1/p^\infty}], 0), \cdot))$ corresponding to the map induced by the inclusion $I \subset B$ and the fact that $\mathrm{Hom}_{(B,0)}((B[x^{1/p^\infty}], 0), \cdot)$ is naturally valued in $B$-algebras. \end{proof}{}

\begin{proposition}\label{tr}The functor $${\mathrm{Un}}: {\mathrm{QID}\w{--}\mathbb{G}_a^{\mathrm{perf}}}^{\mathrm{op}} \to \mathrm{Fun}(\mathscr{P}I, \mathrm{Alg}_A)_{\mathfrak{G}/} $$ is fully faithful.
\end{proposition}{}

\begin{proof}Follows from \cref{satu2} after noting that $r {\mathrm{Un}}(X) \simeq X$ by \cref{lunch} and \cref{xx}. \end{proof}

\subsection{Unwinding pointed $\mathbb{G}_a^{\mathrm{perf}}$-modules II}\label{sec3.3}In this section, we record a variant of the construction appearing in the previous section. We will use a pointed $\mathbb{G}_a^{\mathrm{perf}}$-module to produce a functor from QRSP algebras over $\mathbb{F}_p$ to $\mathrm{Alg}_A.$ Our goal is to formulate and prove an analogue of \cref{tr} in this context. We will begin by recalling the definition of QRSP algebras from \cite[Def. 8.8]{BMS19}.

\begin{definition}\label{covid}An $\mathbb{F}_p$-algebra $S$ is said to be semiperfect if the natural map $S^\flat \to S$ is surjecrive. $S$ is called quasiregular semiperfect (QRSP) if $S$ is semiperfect and the cotangent complex $\mathbb{L}_{S/\mathbb{F}_p}$ is a flat $S$-module supported in (homological) degree 1.
\end{definition}{}

\begin{example}The algebra $\mathbb{F}_p[x^{1/p^\infty}]/x$ is an example of a QRSP algebra.
\end{example}{}

\begin{remark}
The condition on the cotangent complex appearing in \cref{covid} is not relevant for the constructions appearing in this section. However, this condition is important when we compare our constructions with de Rham and crystalline cohomology in \cref{sec4}.
\end{remark}{}

\begin{construction}\label{athosp}For a QRSP algebra $S,$ by sending $S \mapsto (S^\flat, \mathrm{Ker}(S^\flat \to S))$ we can define a functor $\mathrm{QRSP} \to \mathscr{P}I.$ Using \cref{cons4}, this produces a (contravariant) functor ${\mathrm{Un}}:\mathbb{G}_a^{\mathrm{perf}}\w{--} \mathrm{Mod}_* \to \mathrm{Fun}(\mathrm{QRSP}, \mathrm{Alg}_A)$ which will again be called the unwinding of a pointed $\mathbb{G}_a^{\mathrm{perf}}$-module when no confusion is likely to occur. 
\end{construction}

\begin{remark}We note that the functor $\mathfrak{G}: \mathscr{P}I \to \mathrm{Alg}_A$ that sends $(R,I) \mapsto W_A(R)$ produces a functor $\mathrm{QRSP} \to \mathrm{Alg}_A $ that sends $S \mapsto W_A(S^\flat)$ which will again be denoted by $\mathfrak{G}.$ It follows from \cref{dinner} that we have actually produced a (contravariant) functor ${\mathrm{Un}}: \mathbb{G}_a^{\mathrm{perf}} \w{--} \mathrm{Mod}_* \to \mathrm{Fun} (\mathrm{QRSP}, \mathrm{Alg}_A)_{\mathfrak{G}/}.$
\end{remark}{}

\begin{example}\label{satu5}The identity functor $\mathrm{QRSP} \to \mathrm{QRSP}$ induces a functor $\mathrm{QRSP} \to \mathrm{Alg}_{\mathbb{F}_p}$ which is the unwinding of the pointed $\mathbb{G}_a^{\mathrm{perf}}$-module corresponding to $\alpha^{\natural}$ over $\mathbb{F}_p.$ This follows from the construction and \cref{satu4}. 
\end{example}{}

\begin{definition}\label{smoke}Let $F \in \mathrm{Fun} (\mathrm{QRSP}, \mathrm{Alg}_A)_{\mathfrak{G}/}$ be a functor that satisfies the following conditions.\vspace{2mm}

\noindent
\text{1.} The natural map $\mathfrak{G}(B) \to F(B)$ is an isomorphism for every perfect ring $B.$\vspace{1mm}

\noindent
\text{2.} The natural map $F(\frac{B[x^{1/p^\infty}]}{x}) \otimes_{W_A(B)}F(\frac{B[x^{1/p^\infty}]}{x}) \to F (\frac{B[x^{1/p^\infty}]}{x} \otimes_B \frac{B[x^{1/p^\infty}]}{x})$ is an isomorphism for every perfect ring $B.$\vspace{1mm}

\noindent
\text{3.} The natural map $F (\frac{\mathbb{F}_p[x^{1/p^\infty}]}{x})\otimes_A W_A(B) \to F (\frac{B[x^{1/p^\infty}]}{x})$ is an isomorphism for every perfect ring $B.$\vspace{2mm}

The full subcategory spanned by such functors inside $\mathrm{Fun} (\mathrm{QRSP}, \mathrm{Alg}_A)_{\mathfrak{G}/}$ will be denoted as $\mathrm{Fun} (\mathrm{QRSP}, \mathrm{Alg}_A)_{\mathfrak{G}/}^\otimes.$
\end{definition}{}

\begin{proposition}\label{smoke1}Let $F \in \mathrm{Fun} (\mathrm{QRSP}, \mathrm{Alg}_A)_{\mathfrak{G}/}^\otimes$. For every perfect ring $B$, $\mathrm{Spec}\, F(\frac{B[x^{1/p^\infty}]}{x}) $ is naturally a $B$-module scheme over $W_A(B).$ Consequently, we have a (contravariant) functor $r:\mathrm{Fun} (\mathrm{QRSP}, \mathrm{Alg}_A)_{\mathfrak{G}/}^\otimes \to \mathrm{QID}\w{--} \mathbb{G}_a^{\mathrm{perf}}.$
\end{proposition}{}

\begin{proof}This follows in a way similar to the proof of \cref{wp}. We note that $\frac{B[x^{1/p^\infty}]}{x}$ is a cogroup object of $\mathrm{QRSP}_{B/}.$ Therefore, it follows from the definitions that $\mathrm{Spec}\, F (\frac{B[x^{1/p^\infty}]}{x})$ has the structure of a group scheme over $W_A(B)$. The $B$-action on $\mathrm{Spec}\, F (\frac{B[x^{1/p^\infty}]}{x})$ is given by functoriality along the maps $\frac{B[x^{1/p^\infty}]}{x} \to \frac{B[x^{1/p^\infty}]}{x}$ that sends $x^{1/p^n} \to b^{1/p^n} x^{1/p^n}$ for all $n \ge 1.$ Therefore, $\mathrm{Spec}\, F (\frac{B[x^{1/p^\infty}]}{x})$ is indeed naturally a $B$-module scheme over $W_A(B).$ \cref{perfect} implies that varying this data over all perfect rings $B$ provides us a $\mathbb{G}_a^{\mathrm{perf}}$-module. Further, functoriality along the maps $B[x^{1/p^\infty}] \to \frac{B[x^{1/p^\infty}]}{x}$ equips this $\mathbb{G}_a^{\mathrm{perf}}$-module with the structure of a pointed $\mathbb{G}_a^{\mathrm{perf}}$-module. To see that it is a quasi-ideal in $\mathbb{G}_a ^{\mathrm{perf}}$, we use functoriality along the following commutative diagram in $\mathrm{QRSP}.$

\begin{center}
\begin{tikzcd}
\frac{\mathbb F_p[x^{1/p^\infty}]}{x} \otimes \frac{\mathbb F_p[x^{1/p^\infty}]}{x}         &  &  & \frac{\mathbb F_p[x^{1/p^\infty}]}{x}\otimes \mathbb{F}_p[x^{1/p^\infty}] \arrow[lll] \\
\mathbb{F}_p[x^{1/p^\infty}]\otimes \frac{\mathbb F_p[x^{1/p^\infty}]}{x} \arrow[u] &  &  & \frac{\mathbb F_p[x^{1/p^\infty}]}{x} \arrow[u, "x^{\frac{1}{p^n}} \to x^{\frac{1}{p^n}} \otimes x^\frac{1}{p^n}   "'] \arrow[lll, "x^{\frac{1}{p^n}}\otimes x^\frac{1}{p^n}  \leftarrow x^{\frac{1}{p^n}}   "]            
\end{tikzcd}    
\end{center}\end{proof}{}

\begin{remark}Note that we \textit{do not} have a (contravariant) functor $\mathrm{QID}\w{--} \mathbb{G}_a^{\mathrm{perf}} \to \mathrm{Fun} (\mathrm{QRSP}, \mathrm{Alg}_A)_{\mathfrak{G}/}^\otimes$ induced by the unwinding. Indeed, the unwinding of the quasi-ideal $\mathbb{G}_a^{\mathrm{perf}}$ produces the functor that sends a QRSP algebra $S \mapsto S^\flat$ and does not satisfy the last two conditions of \cref{smoke}. One may use \cref{smoke3} to see the latter claim. Roughly speaking, since $S^\flat$ is defined via an inverse limit, one cannot expect the unwinding of an arbitrary quasi-ideal (as in \cref{athosp}) to preserve the pushout diagrams demanded by \cref{smoke}. However, we will work towards rectifying this situation by restricting our attention to a special class of quasi-ideals. In any case, we have the following proposition.
\end{remark}{}

\begin{proposition}\label{stu1}Let $r:\mathrm{Fun} (\mathrm{QRSP}, \mathrm{Alg}_A)_{\mathfrak{G}/}^\otimes \to {\mathrm{QID}\w{--}\mathbb{G}_a^{\mathrm{perf}}}^{\mathrm{op}}$ be the functor from \cref{smoke1}. Let $F \in \mathrm{Fun} (\mathrm{QRSP}, \mathrm{Alg}_A)_{\mathfrak{G}/}^\otimes.$ Then there is a natural transformation ${\mathrm{Un}} (rF) \to F$ in $\mathrm{Fun} (\mathrm{QRSP}, \mathrm{Alg}_A)_{\mathfrak{G}/}.$
\end{proposition}{}

\begin{proof}This follows in a way similar to the proof of \cref{satu1} once we note the following statement about the category $\mathrm{QRSP}.$ Let $S$ be a QRSP algebra. Let $I:= \mathrm{Ker}(S^\flat \to S)$. Then there is an isomorphism 
$$ \varphi: \mathrm{Hom}_{S^\flat} (S, \cdot) \simeq \mathrm{Eq} \left (\mathrm{Hom}_{S^\flat\mathrm{Mod}}(I, \mathrm{Hom}_{S^\flat}(\frac{S^\flat[x^{1/p^\infty}]}{x},\cdot)) \,\, \substack{\longrightarrow \\ \longrightarrow}\,\ \mathrm{Hom}_{S^\flat\mathrm{Mod}}(I, \mathrm{Hom}_{S^\flat}(S^\flat[x^{1/p^\infty}],\cdot))      \right)$$ in $\mathrm{PSh}(\mathrm{QRSP}_{S^\flat/}^{\mathrm{op}}).$ Here one of the arrows is induced by the map $S^\flat[x^{1/p^\infty}] \to \frac{S^\flat[x^{1/p^\infty}]}{x}$ and the other map is obtained by sending everything to the element of $\mathrm{Hom}_{S^\flat\mathrm{Mod}}(I, \mathrm{Hom}_{S^\flat}(S^\flat[x^{1/p^\infty}],\cdot))$ corresponding to the map induced by the inclusion $I \subset S^\flat$ and the fact that $\mathrm{Hom}_{S^\flat}(S^\flat[x^{1/p^\infty}],\cdot)$ is naturally valued in $S^\flat$-algebras. \end{proof}{}

\begin{definition}[Nilpotent quasi-ideals]\label{nilpotent}Let $X$ be a quasi-ideal in $\mathbb{G}_a^{\mathrm{perf}}$ over $\mathbb{F}_p$. We will call $X$ a \textit{nilpotent} quasi-ideal in $\mathbb{G}_a^{\mathrm{perf}}$ if the graded map $\mathbb{F}_p[x^{1/p^\infty}] \to \Gamma(X, \mathcal{O}_X) $ corresponding to $X \to \mathbb{G}_a^{\mathrm{perf}}$ is zero in large enough degrees. In other words writing $t \in \Gamma(X, \mathcal{O}_X)$ as the image of $x$, we need $t$ to be a nilpotent element. We define $\mathcal{N} \mathrm{QID}\w{--} \mathbb{G}_a^{\mathrm{perf}}$ to be the full subcategory of $\mathrm{QID}\w{--}\mathbb{G}_a^{\mathrm{perf}}$ spanned by nilpotent quasi-ideals.
\end{definition}{}

\begin{example}\label{socs}
The zero section $\w{Spec}\, \mathbb{F}_p \to \mathbb{G}_a^{\w{perf}}$ viewed as a quasi-ideal in $\mathbb{G}_a^{\w{perf}}$ is an example of a nilpotent quasi-ideal; this is also the initial object in the category $\mathcal{N} \mathrm{QID}\w{--} \mathbb{G}_a^{\mathrm{perf}}$. Further, we note that $\alpha^{\natural}$ and $u^*W[F]$ are both examples of nilpotent quasi-ideals in $\mathbb{G}_a^{\mathrm{perf}}$ over $\mathbb{F}_p.$ However, $\mathbb{G}_a^{\mathrm{perf}}$ is \textit{not} an example of a nilpotent quasi-ideal.
\end{example}{}

\begin{remark}We note that for every $\mathbb{F}_p$-algebra $R,$ using the map $X \to \mathbb{G}_a^{\mathrm{perf}}$, one gets a map $X(R) \to R^\flat$ at the level of $R$-valued points. Composing along the map $R ^\flat \to R,$ we get a map $w: X(R) \to R.$ It follows that if $X$ is a nilpotent quasi-ideal in $\mathbb{G}_a^\mathrm{perf},$ then $w(z)$ is a nilpotent element of $R$ for every $z \in X(R).$ One can analogously define a notion of ``nilpotent quasi-ideals" in $\mathbb{G}_a$ as well, which can be though of as an analogue of locally nilpotent ideals at the level of $R$-valued points for every $\mathbb{F}_p$-algebra $R.$ Since we do not use the notion of nilpotent quasi-ideals in $\mathbb{G}_a$, we do not discuss them here.
\end{remark}{}

\begin{remark}\label{nil1}In fact if $X$ is a quasi-ideal in $\mathbb{G}_a^{\mathrm{perf}}$ over $\mathbb{F}_p$ that is not isomorphic to $\mathbb{G}_a^{\mathrm{perf}}$ then $X$ is a nilpotent quasi-ideal. To see this, we note that by writing $X = \mathrm{Spec}\, B$ for a graded Hopf algebra $B$ and $t^i$ for the image of $x^i$ under the map $\mathbb{F}_p[x^{1/p^\infty}] \to B$ (here $i  \in \mathbb{N}[1/p]$), we note that $X$ is a quasi-ideal if and only if $b \otimes t^{\deg b} = t^{\deg b} \otimes b$ in $B \otimes B$ for every homogeneous $b \in B.$ Now $b \otimes t^{\deg b} = t^{\deg b} \otimes b$ implies that $t^{\deg b}$ and $b$ are linearly dependent in the $\mathbb{F}_p$ vector space $B.$ Thus if $X$ is not nilpotent, i.e., if $t^i \ne 0$ for all $i$, then any non-zero homogeneous $b \in B$ is in the linear span of $t^i.$ Thus as a graded algebra $B \simeq \mathbb{F}_p[x^{1/p^\infty}]$ and since the map $\mathbb{F}_p[x^{1/p^\infty}] \to B$ is a map of graded Hopf algebras, it follows that the quasi-ideal $X$ is isomorphic to $\mathbb{G}_a^{\mathrm{perf}}.$
\end{remark}{}

We will now record a lemma.
\begin{lemma}\label{smoke3}For a perfect ring $B$, let $S := B[x_1^{1/p^\infty}, \ldots, x_n^{1/p^\infty}]$ and $I:= (x_1, \ldots, x_n).$ Then $(S/I)^\flat = \widehat{S}$, the $I$-adic completion of $S.$ Further, kernel of the map $(S/I)^\flat \to S/I$ is identified with the ideal $(x_1, \ldots, x_n)$ in $\widehat{S}.$
\end{lemma}{}

\begin{proof}This will follow from the more general fact that if $S$ is a perfect ring and $I$ is a finitely generated ideal then $(S/I)^\flat$ is isomorphic to the $I$-adic completion of $S$ denotes as $\widehat{S}.$ We let $I^{[p^n]}:= \left \{x^{p^n} \mid x\in I  \right \}.$ Since $S$ is perfect, it follows that $I^{[p^n]}$ is an ideal of $S.$ By sending an element to its $p^n$-th power we get an isomorphism $\phi^n: S/I \to S/I^{[p^n]}.$ These maps provide an isomorphism of inverse systems as below. 

\begin{center}
\begin{tikzcd}
\ldots \arrow[r] & S/I \arrow[r, "\phi"] \arrow[d, "\phi^2"] & S/I \arrow[r, "\phi"] \arrow[d, "\phi"] & S/I \arrow[d] \\
\ldots \arrow[r] & {S/I^{[p^2]}} \arrow[r]                   & {S/I^{[p]}} \arrow[r]                   & S/I          
\end{tikzcd}
\end{center}{}
Now since $I$ is finitely generated, $\left \{ I^{[p^n]} \right \}$ and $\left \{ I^n \right \}$ generates the same topology on $S$. This gives the isomorphism $(S/I)^\flat \simeq \widehat{S}.$ \end{proof}{}

\begin{proposition}\label{smoke4}Let $X$ be a nilpotent quasi-ideal. Then ${\mathrm{Un}}(X) (\frac{B[x^{1/p^\infty}]}{x})\simeq \Gamma(X, \mathcal{O}_X)\otimes_{\mathbb{F}_p} B.$
\end{proposition}

\begin{proof}${\mathrm{Un}}(X) (\frac{B[x^{1/p^\infty}]}{x})$ by definition and \cref{smoke3} is $\mathrm{Un}(X) (\widehat{B[x^{1/p^\infty}]}, x).$ Since $x$ is a non-zero divisor, this is computed as coequalizer of the two maps $$ \widehat{B[x^{1/p^\infty}]}[y^{1/p^\infty}] \,\, \substack{\longrightarrow \\ \longrightarrow}\,\, \Gamma(X, \mathcal{O}_X) \otimes_{\mathbb{F}_p} \widehat{B[x^{1/p^\infty}]}$$ where one of the maps is induced by $\mathbb{F}_p[y^{1/p^\infty}] \to \Gamma(X, \mathcal{O}_X)$ corresponding to the data of the point. The other map is the $\widehat{B[x^{1/p^\infty}]}$-algebra map that sends $y^{1/p^n} \to x^{1/p^n}. $ Since the quasi-ideal is nilpotent, a power of $y$ is sent to zero by the first map. Hence we obtain the required isomorphism. \end{proof}{}

\begin{proposition}\label{smoke5}Let $X$ be a nilpotent quasi-ideal. Then $${\mathrm{Un}}(X) \left(\frac{B[x^{1/p^\infty}]}{x} \otimes_B \frac{B[x^{1/p^\infty}]}{x}\right) \simeq {\mathrm{Un}}(X)\left(\frac{B[x^{1/p^\infty}]}{x}\right)\otimes_B {\mathrm{Un}}(X) \left(\frac{B[x^{1/p^\infty}]}{x}\right) \simeq \Gamma(X, \mathcal{O}_X)\otimes_{\mathbb{F}_p} \Gamma(X, \mathcal{O}_X)\otimes_{\mathbb{F}_p}B.$$
\end{proposition}{}

\begin{proof}By definition and \cref{smoke3}, the left hand side is isomorphic to ${\mathrm{Un}}(X) (\widehat{B[x_1 ^{1/p^\infty}, x_2 ^{1/p^\infty}]}, (x_1, x_2)).$ By regularity of $(x_1, x_2)$ as an ideal of $\widehat{B[x_1 ^{1/p^\infty}, x_2 ^{1/p^\infty}]}$ and \cref{lunch1} that is computed as $$ \mathrm{Env}_X(\widehat{B[x_1 ^{1/p^\infty}, x_2 ^{1/p^\infty}]}, x_1) \otimes_{\widehat{B[x_1^{1/p^\infty}, x_2 ^{1/p^\infty}]}}    \mathrm{Env}_X(\widehat{B[x_1 ^{1/p^\infty}, x_2 ^{1/p^\infty}]}, x_2) . $$ By letting $t^{1/p^n}$ denote the image of $y^{1/p^n}$ under the map $\mathbb{F}_p[y^{1/p^\infty}] \to \Gamma(X, \mathcal{O}_X)$ corresponding to the data of the point, we obtain that the above expression is isomorphic to $$\frac{   \widehat{B[x_1^{1/p^\infty}, x_2^{1/p^\infty}]} \otimes \Gamma(X, \mathcal{O}_X)} {(x_1 ^{1/p^n}\otimes 1 - 1 \otimes t^{1/p^n})} \otimes_{\widehat{B[x_1^{1/p^\infty}, x_2^{1/p^\infty}]}}  \frac{   \widehat{B[x_1^{1/p^\infty}, x_2^{1/p^\infty}]} \otimes \Gamma(X, \mathcal{O}_X)} {(x_2 ^{1/p^n}\otimes 1 - 1 \otimes t^{1/p^n})}.$$ Since $X$ is nilpotent, a power of $t$ is zero which along with \cref{smoke4} gives the required conclusion.
\end{proof}{}

\begin{remark}More generally, the proof of \cref{smoke5} shows that for a nilpotent quasi-ideal $X$, ${\mathrm{Un}}(X)$ commutes with finite coproducts of $\frac{B[x^{1/p^\infty}]}{x}.$ 
\end{remark}{}

\begin{proposition}\label{smoke6}The unwinding of a nilpotent quasi-ideal (over $\mathbb{F}_p$) satisfies the properties in \cref{smoke}, i.e., we have a functor $${\mathrm{Un}}: {\mathcal{N}\mathrm{QID}\w{--} \mathbb{G}_a^{\mathrm{perf}}}^{\mathrm{op}} \to \mathrm{Fun} (\mathrm{QRSP}, \mathrm{Alg}_{\mathbb{F}_p})_{\mathfrak{G}/}^\otimes.$$
\end{proposition}{}

\begin{proof}Let $X$ be a nilpotent quasi-ideal. By definition, we need to check three properties for the functor $F:={\mathrm{Un}}(X).$ The first one is that the natural map $B \to F(B)$ is an isomorphism for every perfect ring $B$ which follows from \cref{kle}. The other two properties follow from \cref{smoke4} and \cref{smoke5}. \end{proof}{}

\begin{proposition}\label{last}The functor $r: \mathrm{Fun} (\mathrm{QRSP}, \mathrm{Alg}_{\mathbb{F}_p})_{\mathfrak{G}/}^\otimes \to {\mathrm{QID}\w{--} \mathbb{G}_a ^{\mathrm{perf}}}^{\mathrm{op}}$ factors to give a functor $r: \mathrm{Fun} (\mathrm{QRSP}, \mathrm{Alg}_{\mathbb{F}_p})_{\mathfrak{G}/}^\otimes \to {\mathcal{N} \mathrm{QID}\w{--} \mathbb{G}_a^{\mathrm{perf}}}^{\mathrm{op}} $ which admits a left adjoint given by ${\mathrm{Un}}$ from \cref{smoke6}.
 \end{proposition}{}

\begin{proof}First we prove that we indeed have a factorization which gives the functor $r: \mathrm{Fun} (\mathrm{QRSP}, \mathrm{Alg}_{\mathbb{F}_p})_{\mathfrak{G}/}^\otimes \to \mathcal{N} \mathrm{QID}\w{--} \mathbb{G}_a^{\mathrm{perf}}. $ By \cref{nil1}, it would be enough to prove that the essential image of the functor $r: \mathrm{Fun} (\mathrm{QRSP}, \mathrm{Alg}_{\mathbb{F}_p})_{\mathfrak{G}/}^\otimes \to  \mathrm{QID}\w{--} \mathbb{G}_a^{\mathrm{perf}}$ does not contain $\mathbb{G}_a^{\mathrm{perf}}.$ We assume on the contrary that there is an $F \in \mathrm{Fun} (\mathrm{QRSP}, \mathrm{Alg}_{\mathbb{F}_p})_{\mathfrak{G}/}^\otimes$ such that $rF \simeq \mathbb{G}_a^{\mathrm{perf}}$ as quasi-ideals in $\mathbb{G}_a^{\mathrm{perf}}.$ This implies that the arrow $f:\mathbb{F}_p[x^{1/p^\infty}] \to \mathbb{F}_p[x^{1/p^\infty}]/x$ is sent to an isomorphism by $F.$ The arrow $f$ factors as $\mathbb{F}_p[x^{1/p^\infty}] \to \widehat{\mathbb{F}_p[x^{1/p^\infty}]} \to \mathbb{F}_p[x^{1/p^\infty}]/x.$ Applying $F$ to it and using the first property from \cref{smoke} gives the following maps $$\mathbb{F}_p[x^{1/p^\infty}] \to \widehat{\mathbb{F}_p[x^{1/p^\infty}]} \to F(\mathbb{F}_p[x^{1/p^\infty}]/x)$$ whose composition is an isomorphism. This shows that there are maps $\mathbb{F}_p[x^{1/p^\infty}] \to \widehat{\mathbb{F}_p[x^{1/p^\infty}]} \to \mathbb{F}_p[x^{1/p^\infty}]$ whose composition is the identity. That implies that there is a map $\widehat{\mathbb{F}_p[x^{1/p^\infty}]} \to \mathbb{F}_p[x^{1/p^\infty}]$ that sends $x \to x$. But no such map can exist since $1+x$ is a unit on the source but not on the target of the map. Now the required adjunction follows from \cref{stu1}, \cref{smoke4} and \cref{smoke5} (similar to the proof of \cref{satu1}) by noting the commutative diagram 

\begin{center}
\begin{tikzcd}
{{\mathrm{Un}}(r {\mathrm{Un}}(X))} \arrow[d] \arrow[rr, "\simeq"] &  & { {\mathrm{Un}}(X)}  \arrow[d] \\
{{\mathrm{Un}}(rF)} \arrow[rr]                                   &  & { F }                                          
\end{tikzcd}
\end{center}{}for any natural transformation ${\mathrm{Un}}(X) \to F$ where $X \in \mathcal{N} \mathrm{QID}\w{--} \mathbb{G}_a^{\mathrm{perf}}$ and $F \in \mathrm{Fun} (\mathrm{QRSP}, \mathrm{Alg}_{\mathbb{F}_p})_{\mathfrak{G}/}^\otimes.$\end{proof}{}

\begin{proposition}\label{nil}The functor $${\mathrm{Un}}: {\mathcal{N}\mathrm{QID}\w{--} \mathbb{G}_a^{\mathrm{perf}}}^{\mathrm{op}} \to \mathrm{Fun} (\mathrm{QRSP}, \mathrm{Alg}_{\mathbb{F}_p})_{\mathfrak{G}/}^\otimes$$ defined in \cref{smoke6} is fully faithful.
\end{proposition}{}

\begin{proof}We fix two nilpotent quasi-ideals $X$ and $Y.$ By using \cref{smoke4} and \cref{smoke5}, there are natural isomorphisms $r {\mathrm{Un}}(X) \simeq X$ and $r {\mathrm{Un}}(Y) \simeq Y$, where $r$ is the functor from \cref{stu1}. Therefore, \cref{smoke4} and \cref{smoke5} implies that ${\mathrm{Un}}$ is faithful. To show that it is full, it would be enough to prove that if $F$ and $G$ are two natural transformations between ${\mathrm{Un}}(X)$ and ${\mathrm{Un}}(Y)$ such that they are the same transformation $X \to Y$ in ${\mathcal{N}\mathrm{QID}\w{--} \mathbb{G}_a^{\mathrm{perf}}}^{\mathrm{op}}$ after applying $r$, then $F= G.$ For this, we note the following commutative diagram.
\begin{center}
\begin{tikzcd}
{\mathrm{Un}}(r \mathrm{Un}(X)) \arrow[rr, "\simeq"] \arrow[d] &  & {\mathrm{Un}}(X)  \arrow[d, "F"'] \arrow[d, "G"] \\
{\mathrm{Un}}(r \mathrm{Un}(Y)) \arrow[rr, "\simeq"]                  &  & {\mathrm{Un}}(Y)                              
\end{tikzcd}    
\end{center}{}
The diagram above shows that $F= G,$ as desired.\vspace{2mm}

Alternatively, this follows from \cref{last} since $r {\mathrm{Un}}(X) \simeq X$ for a nilpotent quasi-ideal $X.$
\end{proof}{}

\begin{remark}\label{la}More generally, let $X$ and $Y$ be two quasi-ideals in $\mathbb{G}_a^{\mathrm{perf}}$ over an Artinian local ring $A$ with residue field $\mathbb{F}_p$ such that the functors ${\mathrm{Un}}(X)$ and ${\mathrm{Un}}(Y)$ satisfies the three conditions in \cref{smoke} and such that there are natural isomorphisms $r {\mathrm{Un}}(X) \simeq X$ and $r {\mathrm{Un}}(Y) \simeq Y$. Then the above proof shows that there is a natural bijection $\mathrm{Hom}_{\mathrm{QID}\w{--} \mathbb{G}_a^{\mathrm{perf}}}(Y,X) \simeq \mathrm{Hom}({\mathrm{Un}}(X), {\mathrm{Un}}(Y))$ where the latter Hom is computed in $\mathrm{Fun} (\mathrm{QRSP}, \mathrm{Alg}_A)_{\mathfrak{G}/}.$
\end{remark}{}

Now we are ready to make the following definitions.

\begin{definition}We let $\mathrm{Fun}(\mathrm{QRSP}, \mathrm{Alg}_{\mathbb{F}_p})^{\mathcal{N}{\mathrm{Un}}}_{\mathfrak{G}/}$ denote the full subcategory of 
$\mathrm{Fun}(\mathrm{QRSP}, \mathrm{Alg}_{\mathbb{F}_p})^{\otimes}_{\mathfrak{G}/}$ spanned by image of nilpotent quasi-ideals under the functor ${\mathrm{Un}}$ from \cref{smoke6}.
\end{definition}{}
\begin{definition}We let $\mathrm{Fun}(\mathrm{QRSP}, \mathrm{Alg}_{\mathbb{F}_p})^{\mathrm{ rk}=1, \mathcal{N}{\mathrm{Un}}}_{\mathfrak{G}/}$ denote the full subcategory of $\mathrm{Fun}(\mathrm{QRSP}, \mathrm{Alg}_{\mathbb{F}_p})^{\otimes}_{\mathfrak{G}/}$ spanned by the unwinding of the nilpotent quasi-ideals whose underlying pointed $\mathbb{G}_a^{\mathrm{perf}}$-module is of fractional rank $1$ (\cref{satu3}).
\end{definition}{}

\begin{definition}We let $\mathrm{Fun}(\mathrm{QRSP}, \mathrm{Alg}_{\mathbb{F}_p})^{\mathrm{pure~rk}=1, \mathcal{N}{\mathrm{Un}}}_{\mathfrak{G}/}$ denote the full subcategory of $\mathrm{Fun}(\mathrm{QRSP}, \mathrm{Alg}_{\mathbb{F}_p})^{\otimes}_{\mathfrak{G}/}$ spanned by the unwinding of the nilpotent quasi-ideals whose underlying pointed $\mathbb{G}_a^{\mathrm{perf}}$-module is pure of fractional rank $1$ (\cref{pure}).
\end{definition}{}

Thus we obtain the following chain of inclusion of categories $$\mathrm{Fun}(\mathrm{QRSP}, \mathrm{Alg}_{\mathbb{F}_p})^{\mathrm{pure~rk}=1, \mathcal{N}{\mathrm{Un}}}_{\mathfrak{G}/} \subset \mathrm{Fun}(\mathrm{QRSP}, \mathrm{Alg}_{\mathbb{F}_p})^{\mathrm{ rk}=1, \mathcal{N}{\mathrm{Un}}}_{\mathfrak{G}/} \subset \mathrm{Fun}(\mathrm{QRSP}, \mathrm{Alg}_{\mathbb{F}_p})^{\mathcal{N}{\mathrm{Un}}}_{\mathfrak{G}/} \subset \mathrm{Fun}(\mathrm{QRSP}, \mathrm{Alg}_{\mathbb{F}_p})^{\otimes}_{\mathfrak{G}/}$$ which are all full subcategories of $\mathrm{Fun}(\mathrm{QRSP}, \mathrm{Alg}_{\mathbb{F}_p})_{\mathfrak{G}/}$

\begin{proposition}
 The category $\mathrm{Fun}(\mathrm{QRSP}, \mathrm{Alg}_{\mathbb{F}_p})^{\mathcal{N}{\mathrm{Un}}}_{\mathfrak{G}/}$ has a final object given by the functor $(\,\cdot\,)_{\mathrm{perf}}: \mathrm{QRSP} \to \mathrm{Alg}_{\mathbb{F}_p}$ that sends $S \mapsto S_{\mathrm{perf}}:= \mathrm{colim}_{x \mapsto x^p}S.$
\end{proposition}{}

\begin{proof}
This follows from the fact that the zero quasi-ideal $\mathrm{Spec}\,\mathbb{F}_p \to \mathbb{G}_a^{\w{perf}}$ is a nilpotent quasi-ideal (\cref{socs}), \cref{satu4} and \cref{nil}.
\end{proof}{}

\begin{proposition}\label{satu6}The category $\mathrm{Fun}(\mathrm{QRSP}, \mathrm{Alg}_{\mathbb{F}_p})^{\mathrm{ rk}=1, \mathcal{N}{\mathrm{Un}}}_{\mathfrak{G}/}$ has a final object given by the functor $\mathrm{id}: \mathrm{QRSP} \to \mathrm{Alg}_{\mathbb{F}_p}$ that sends $S \mapsto S.$
\end{proposition}{}

\begin{proof}This follows from \cref{Hmap}, \cref{satu5} and \cref{nil}. \end{proof}{}

\begin{proposition}\label{univ}The category $\mathrm{Fun}(\mathrm{QRSP}, \mathrm{Alg}_{\mathbb{F}_p})^{\mathrm{pure~rk}=1, \mathcal{N}{\mathrm{Un}}}_{\mathfrak{G}/}$ has a final object given by the functor ${\mathrm{Un}} (u^*W[F]).$
\end{proposition}

\begin{proof}This follows from \cref{rem6} and \cref{nil}. \end{proof}

\subsection{The generalized Hodge filtration}\label{sec3.4}Let $X$ be a fixed pointed $\mathbb{G}_a^{\mathrm{perf}}$-module over $\mathbb{F}_p$. The goal of this section is to construct a decreasing filtration on the functor ${\mathrm{Un}}(X)$ defined on QRSP algebras which will be called the ``Hodge filtration". This will be done by explicitly constructing a functorial filtration on ${\mathrm{Un}}(X) (S)=\mathrm{Env}_X (S^\flat, \mathrm{Ker}(S^\flat \to S)).$ We will show that under the assumption that $X$ is a fractional rank-$1$ pointed $\mathbb{G}_a^{\mathrm{perf}}$-module, $\mathrm{gr}^0$ of the filtration on ${\mathrm{Un}}(X)(S)$ identifies with $S.$ This induces a  natural transformation $\mathrm{gr}^0:{\mathrm{Un}}(X) \to \mathrm{id}$ of functors. Under the additional assumption that $X$ is a nilpotent quasi-ideal, this natural transformation is the same as the one coming from \cref{satu6}.

\begin{construction}[Hodge filtration]\label{con}Let $X$ be a pointed $\mathbb{G}_a^{\mathrm{perf}}$-module over $\mathbb{F}_p.$ We will construct a natural decreasing filtration on $\mathrm{Env}_X (B,I)$ for $(B,I) \in \mathscr{P}I.$ By \cref{cons1}, for a fixed $i \in I$, we have a natural map $\Gamma (X_B, \mathcal{O}_{X_B}) \to \coprod_I  \Gamma(X_B, \mathcal{O}_{X_B}) \to \mathscr{T}_{X_B}(I),$ where the first map maps $\Gamma(X_B, \mathcal{O}_{X_B})$ to the $i$-th factor in the coproduct. The composite map is the map $\w{ev}_i : \Gamma (X_B, \mathcal{O}_{X_B}) \to \mathscr{T}_{X_B}(I)$ defined in \cref{bday28}. \vspace{2mm}

By \cref{cons2}, there is a natural surjection $\mathscr{T}_{X_B}(I) \to \mathrm{Env}_X(B,I).$ Composing this with $\w{ev}_i$, we obtain the map $$ [i]: \Gamma(X_B, \mathcal{O}_{X_B}) \to \mathrm{Env}_X(B,I). $$ If $m \in \Gamma(X_B, \mathcal{O}_{X_B})$ is a homogeneous element, we will write $[i]^m:= [i](m)$.\vspace{2mm}

For a nonnegative integer $n$, we let $\mathrm{Fil}^n \mathrm{Env}_X (B,I)$ denote the ideal of $\mathrm{Env}_X (B,I)$ generated by the ``monomials" of the form $[i_1]^{m_1} \cdots [i_k]^{m_k}$ such that $\sum_{u=1}^{k}\deg m_u \ge n,$ for homogeneous elements $m_u \in \Gamma(X_B, \mathcal{O}_{X_B})$ of \textit{integral degree}, $k \ge 1$ and for $i_1, \ldots, i_k \in I.$ This defines a decreasing filtration which will be called the Hodge filtration on $\mathrm{Env}_X(B,I)$. A map $\varphi : (B,I) \to (B',I')$ sends $[i]^m \to [\varphi(i)]^m,$ so the construction is functorial. 
\end{construction}{}

\begin{definition}Let $X$ be a pointed $\mathbb{G}_a^{\mathrm{perf}}$-module over $\mathbb{F}_p$ and let $S$ be a QRSP algebra. The decreasing filtration defined by $$\mathrm{Fil}^n {\mathrm{Un}}(X) (S):= \mathrm{Fil}^n \mathrm{Env}_X (S^\flat, \mathrm{Ker}(S^\flat \to S))$$ will be called the \textit{Hodge filtration} on ${\mathrm{Un}} (X) (S).$
\end{definition}{}

\begin{definition}\label{hdc}Let $\widehat{{\mathrm{Un}}(X) (S)}$ be the completion of ${\mathrm{Un}} (X) (S)$ with respect to the Hodge filtration. Then the functor $\widehat{{\mathrm{Un}}(X)}$ will be called the \textit{Hodge completion} of ${\mathrm{Un}}(X).$ 
\end{definition}{}

\begin{example}[$I$-adic filtration]Let $X = \mathbb{G}_a^{\mathrm{perf}}$ (equipped with the natural structure of a pointed $\mathbb{G}_a^{\mathrm{perf}}$-module). Then $\mathrm{Env}_X (B,I) \simeq B$ by \cref{dinner}. In this case, the Hodge filtration identifies with the $I$-adic filtration on $B.$
\end{example}{}

\begin{example}[Divided power filtration]\label{example4}Let $B$ be an $\mathbb{F}_p$-algebra and $I$ be an ideal of $B.$ Then there is a filtration on $D_B(I)$ given by setting $\mathrm{Fil}^n (D_B(I))$ to be the ideal generated by divided power monomials $[i_1]^{m_1} \cdots [i_k]^{m_k}$ such that $\sum_{u=1}^{k} m_i \ge n$ and $i_1, \ldots, i_k \in I$, which is called the divided power filtration. We note that $\mathrm{Fil}^0 (D_B(I))= D_B(I)$ and $\mathrm{gr}^0 (D_B(I))= B/I.$ If we further assume that $B$ is a perfect ring, then we have an isomorphism $\mathrm{Env}_{u^*W[F]} (B,I) \simeq D_B(I)$ by \cref{example3} which is further a filtered isomorphism when we equip $\mathrm{Env}_{u^*W[F]}(B,I)$ with the Hodge filtration 
from \cref{con}.
\end{example}{}

\begin{remark}\label{bday2}
Let $X$ be a pointed $\mathbb{G}_a$-module over an arbitrary base ring $A.$ Let $(B, I) \in \mathfrak{C}_A.$ Similar to \cref{con}, one has a map 
$$[i]: \Gamma(X_B, \mathcal{O}_{X_B}) \to \mathrm{Env}_X(B,I)$$ obtained as a composition of $\w{ev}_i: \Gamma(X_B, \mathcal{O}_{X_B}) \to \mathscr{T}_{X_B}(I)$ (\cref{bday28}) and $\mathscr{T}_{X_B}(I) \to \mathrm{Env}_X (B,I).$ If $m \in \Gamma(X_B, \mathcal{O}_{X_B})$ is a homogeneous element, we can  again write $[i]^m:= [i](m)$. For a nonnegative integer $n$, we let $\mathrm{Fil}^n \mathrm{Env}_X (B,I)$ denote the ideal of $\mathrm{Env}_X (B,I)$ generated by the ``monomials" of the form $[i_1]^{m_1} \cdots [i_k]^{m_k}$ such that $\sum_{u=1}^{k}\deg m_u \ge n,$ for homogeneous elements $m_u \in \Gamma(X_B, \mathcal{O}_{X_B})$, $k \ge 1$ and for $i_1, \ldots, i_k \in I.$ This defines a decreasing filtration which may again be called the Hodge filtration on $\mathrm{Env}_X(B,I).$ 
\end{remark}{}

\begin{example}
When $X = \mathbb{G}_a$ (equipped with the natural structure of a pointed $\mathbb{G}_a$-module), we have $\mathrm{Env}_X(B,I) \simeq B$ (\cref{presheaf}) and the Hodge filtration on the left hand side constructed in \cref{bday2} identifies with the $I$-adic filtration on $B.$ 
\end{example}{}

\begin{lemma}\label{7up7}
Let $X$ be a pointed $\mathbb{G}_a$-module over $A$ and $(B,I) \in \mathfrak{C}_A.$ There is a natural isomorphism $\mathrm{gr}^0 (\mathrm{Env}_X (B,I)) \simeq B/I,$ where the $\mathrm{gr}^0$ on the left hand side is taken with respect to the filtration constructed in \cref{bday2}.
\end{lemma}

\begin{proof}
Note that the zero section $\w{Spec}\, A \to \mathbb{G}_a$ can be thought of as a pointed $\mathbb{G}_a$-module which admits a unique map to $X$ (as a pointed $\mathbb{G}_a$-module). Applying the unwinding construction and using \cref{move4}, we get us a natural map $\mathrm{Env}_{X} (B, I) \to B/I.$ By functoriality of the constructions appearing in \cref{cons1}, \cref{cons2} and using \cref{7up22}, the map $\mathrm{Env}_{X} (B, I) \to B/I$ fits into the following commutative diagram

\begin{center}
    \begin{tikzcd}
\mathscr{T}_{X_B} (I) \arrow[d] \arrow[rr] &  & B \arrow[d] \\
{\mathrm{Env}_X (B,I)} \arrow[rr]          &  & B/I .       
\end{tikzcd}
\end{center}{}
We proceed towards computing the kernel of the map $\mathrm{Env}_{X} (B, I) \to B/I$, which we will denote by $K.$ As noted in \cref{7up3}, $\w{Spec}\, \mathscr{T}_{X_B} (I)$ naturally has the structure of a $\mathbb{G}_a$-module over $B$ and the map $\mathscr{T}_{X_B} (I) \to B$ is the map induced on global sections by the zero section $\w{Spec}\, B \to \w{Spec}\, \mathscr{T}_{X_B} (I).$ We write $\mathscr{T}_{X_B} (I) = \bigoplus_{n \in \mathbb{N}} (\mathscr{T}_{X_B} (I))_n.$ Since $\w{Spec}\, \mathscr{T}_{X_B} (I)$ is a $\mathbb{G}_a$-module, by \cref{conn}, it follows that $(\mathscr{T}_{X_B}(I))_0$ is naturally isomorphic to $B.$ Therefore, it follows that the kernel of the composite map $\mathscr{T}_{X_B} (I) \to B \to B/I$ naturally identifies with $I \oplus (\mathscr{T}_{X_B} (I))_{>0},$ where $(\mathscr{T}_{X_B} (I))_{>0} := \bigoplus_{n>0} (\mathscr{T}_{X_B} (I))_n.$ We observe that all the arrows in the above diagram are surjective. Therefore, by commutativity of the above diagram, the kernel $K$ of the map $\mathrm{Env}_{X} (B, I) \to B/I$ is generated by the image of $I \oplus (\mathscr{T}_{X_B} (I))_{>0}$ under the map $\mathscr{T}_{X_B}(I) \to \mathrm{Env}_X (B,I).$ 
\vspace{2mm}

By \cref{cons2}, the map $\mathscr{T}_{X_B}(I) \to \mathrm{Env}_X (B,I)$ is induced by taking coequalizer of two arrows $\mathrm{Sym}_B(I)\,\, \substack{\longrightarrow \\ \longrightarrow}\,\,\mathscr{T}_{X_B}(I).$ Using their descriptions from \cref{cons2} and writing $\w{Sym}_{B}(I) = \bigoplus_{i \in \mathbb{N}} \w{Sym}^i_{B} (I)$, we see that one of the arrows map $I=\w{Sym}^1_{B} I$ inside $(\mathscr{T}_{X_B}(I))_1$ and the other one maps $I = \w{Sym}^1_{B} I$ inside $(\mathscr{T}_{X_B}(I))_0 = B$ by the natural inclusion $I \subset B.$ Therefore, under the map $\mathscr{T}_{X_B}(I) \to \mathrm{Env}_X (B,I),$ the image of $I \subset (\mathscr{T}_{X_B}(I))_0$ is contained inside the image of $(\mathscr{T}_{X_B}(I))_1.$ This implies that the kernel $K$ of the map $\mathrm{Env}_{X} (B, I) \to B/I$ can be generated by the image of $(\mathscr{T}_{X_B} (I))_{>0}$ under the map $\mathscr{T}_{X_B}(I) \to \mathrm{Env}_X (B,I).$ \vspace{2mm}

Finally, we note that image of $(\mathscr{T}_{X_B} (I))_{>0}$ under the map $\mathscr{T}_{X_B}(I) \to \mathrm{Env}_X (B,I)$ generates the ideal $\mathrm{Fil}^1 \mathrm{Env}_X (B,I)$ as defined in \cref{bday2}. Indeed, the last claim can be seen by considering the map $\w{ev}_i : \Gamma(X_B, \mathcal{O}_{X_B}) \to \mathscr{T}_{X_B}(I)$ and recalling it's properties from \cref{7up3} which imply that the elements $\w{ev}_i (u)$ for all $i \in I$ and all $u \in \Gamma(X_B, \mathcal{O}_{X_B})$ such that $\w{deg}(u) \ge 1$ generate the ideal $(\mathscr{T}_{X_B} (I))_{>0}$ in $\mathscr{T}_{X_B}(I).$ This finishes the proof.
\end{proof}{}

\begin{remark}\label{gax}
The proof of \cref{7up7} shows that the natural transformation $\mathrm{Env}_X(B,I) \mapsto \w{gr}^{0} (\mathrm{Env}_X(B,I))$ $ \simeq B/I$ described in \cref{7up7} is the same as the one obtained by applying the unwinding construction to the unique map from the zero section $\w{Spec}\, A \to \mathbb{G}_a$ (viewed as a pointed $\mathbb{G}_a$-module) to $X.$
\end{remark}{}

\begin{remark}\label{7up8}
Let $X$ be a pointed $\mathbb{G}_a$-module over $\mathbb{F}_p.$ Let $(B,I) \in \mathscr{P}I.$ By \cref{move5}, there is a natural isomorphism
$$\mathrm{Env}_X (B,I) \simeq \mathrm{Env}_{u^*X}(B,I).$$ By construction, it follows that this isomorphism is further a filtered isomorphism, where the filtration on the right hand side is coming from \cref{con} and the one on the left hand side is coming from \cref{bday2}. 
\end{remark}{}

\begin{remark}\label{qui}
In the case when $X$ is a pointed $\mathbb{G}_a$-module over $A,$ analogous to \cref{hdc}, one can define a functor $\widehat{{\mathrm{Un}}(X)}$ by completing with respect to the Hodge filtration on $\mathrm{Env}_X(B,I)$ as constructed in \cref{bday2}. As an example, when $X = \mathbb{G}_a$ (equipped with the structure of a pointed $\mathbb{G}_a$-module), $\widehat{{\mathrm{Un}}(\mathbb{G}_a)}(B,I)$ is simply the $I$-adic completion of $B.$
\end{remark}{}

\begin{remark}
Let $\hat{\mathbb{G}}_a$ be the ``formal affine line" over an arbitrary base ring $A.$ More precisely, for an $A$-algebra $B$, the $B$-valued points of $\hat{\mathbb{G}}_a$ denoted as $\hat{\mathbb{G}}_a(B)$ are defined to be the set of nilpotent elements of $B,$ i.e., the nilradical of $B.$ This equips $\hat{\mathbb{G}}_a$ with the structure of a ``$\mathbb{G}_a$-module". There is also a natural transformation of functors $\hat{\mathbb{G}}_a \to \mathbb{G}_a,$ which, in some sense, equips $\hat{\mathbb{G}}_a$ with the structure of a ``pointed $\mathbb{G}_a$-module". The constructions of our paper do not deal with examples such as $\hat{\mathbb{G}}_a$ that are not representable by an affine scheme and being representable is part of the definition of a $\mathbb{G}_a$-module for us. However, roughly speaking, the functor $\widehat{{\mathrm{Un}}(\mathbb{G}_a)}$ as discussed in \cref{qui} could be thought of as ``unwinding" of $\hat{\mathbb{G}}_a.$ We thank the referee for suggesting to include this remark.
\end{remark}{}

\begin{proposition}\label{move}Let $X$ be a pointed $\mathbb{G}_a^{\mathrm{perf}}$-module over $\mathbb{F}_p$ of fractional rank $1$. Let $S$ be a QRSP algebra. Then the $\mathrm{gr}^0$ of the Hodge filtration on ${\mathrm{Un}}(X)(S)$ is naturally isomorphic to $S.$
\end{proposition}{}

\begin{proof}We write $I = \mathrm{Ker}(S^\flat \to S).$ Since $X$ is of fractional rank $1$ (\cref{satu3}), we have $\mathrm{gr}^0\mathrm{Env}_X (S^\flat, I) \simeq \mathrm{Env}_{\alpha^{\natural}}(S^\flat, I) \simeq S.$ Indeed, the first isomorphism follows from \cref{7up7}, \cref{7up8} and recalling that $u^* (\w{Spec}\, \mathbb{F}_p) \simeq \alpha^{\natural}$ (when $\w{Spec}\, \mathbb{F}_p$ is equipped with the structure of a pointed $\mathbb{G}_a$-module corresponding to the zero section $\w{Spec}\, \mathbb{F}_p \to \mathbb{G}_a)$. Finally, the last isomorphism follows from \cref{satu4}. \end{proof}{}

\begin{example}
We point out that the assumption that $X$ is fractional of rank $1$ is crucial in \cref{move}. Indeed, let us take $X$ to be $\alpha^{\natural} \times \alpha^{\natural}$ considered to be a pointed $\mathbb{G}_a^{\w{perf}}$-module via projection onto the first component $\alpha^{\natural} \times \alpha^{\natural} \to \alpha^{\natural}$ composed with the natural map $\alpha^{\natural} \to \mathbb{G}_a^{\w{perf}}.$ Then $X$ is not fractional of rank $1.$ Further, in this case, one computes directly that $\mathrm{Un}(X) (\frac{\mathbb{F}_p[x^{1/p^\infty}]}{x}) \simeq \frac{\mathbb{F}_p[u^{1/p^\infty}, v^{1/p^\infty}]}{(u,v)}$ and under this identification, $\w{Fil}^1\mathrm{Un}(X) (\frac{\mathbb{F}_p[x^{1/p^\infty}]}{x})$ is the ideal generated by the elements $u^i v^j$ such that $i+j = 1.$ Therefore, $\w{gr}^0\mathrm{Un}(X) (\frac{\mathbb{F}_p[x^{1/p^\infty}]}{x})$ is not isomorphic to $\frac{\mathbb{F}_p[x^{1/p^\infty}]}{x}.$
\end{example}{}

\begin{remark}\label{move1}Let $X$ be a pointed $\mathbb{G}_a^{\mathrm{perf}}$-module over $\mathbb{F}_p$ of fractional rank $1$. The natural transformation $\mathrm{gr}^0: {\mathrm{Un}}(X) \to \mathrm{id}$ induced by \cref{move} is the same as the one obtained via the unwinding functor from the Hodge map defined in \cref{Hmap}. This follows from \cref{gax} and \cref{7up8}.

\end{remark}{}

\begin{proposition}\label{move2}Let $X$ be a nilpotent quasi-ideal over $\mathbb{F}_p$ which is of fractional rank $1$ as a pointed $\mathbb{G}_a^{\mathrm{perf}}$-module. Then the natural transformation $\mathrm{gr}^0: {\mathrm{Un}}(X) \to \mathrm{id}$ is the unique natural transformation between $\mathrm{Un}(X)$ and $\mathrm{id}$ viewed as objects of the category $\mathrm{Fun}(\mathrm{QRSP}, \mathrm{Alg}_{\mathbb{F}_p})_{\mathfrak{G}/}.$
\end{proposition}{}

\begin{proof}This follows from \cref{satu6} and \cref{move1}. \end{proof}{}

\begin{remark} As discussed in \cref{bday2}, for a pointed $\mathbb{G}_a$-module $X$ over $\mathbb{F}_p$, one can define a Hodge filtration $\mathrm{Fil}^n$ on $\mathrm{Env}_X (B,I)$ for any $\mathbb{F}_p$-algebra $B$ and an ideal $I \subset B.$ As a possible application, one can define a candidate theory of $\mathcal{D}$-modules for any given pointed $\mathbb{G}_a$-module $X.$ In order to do so, we let $J$ denote the kernel of the diagonal map $B \otimes B \to B.$ One can attempt to define the ring of ``differential operators" of $B$ associated to $X$ as  $$\mathcal{D}_B ^{X}:= \varinjlim_{n} \mathrm{Hom}_{B}( \mathrm{Env}_X (B \otimes B, J)/ \mathrm{Fil}^n, B).$$ In the case when $X$ is the pointed $\mathbb{G}_a$-module given by $W[F],$ we recover the ring of crystalline differential operators \cite{BO78}. In the case when $X$ is the pointed $\mathbb{G}_a$-module given by $\mathbb{G}_a$ itself, we recover Grothendieck's ring of differential operators \cite{Gro67}.
\end{remark}{}

\newpage

\section{Revisiting de Rham and crystalline cohomology via unwinding}\label{sec4}In this section, we will briefly recall the notion of derived de Rham cohomology. Our goal is to provide a different definition of derived de Rham cohomology as the unwinding of a particular quasi-ideal. The point of our definition is that a single object (a quasi-ideal in $\mathbb{G}_a^{\mathrm{perf}}$) can recover the entire theory of derived de rham cohomology via the unwinding functor. Thus many questions about (derived) de Rham cohomology theory can be translated into a question about quasi-ideals and can be approached in a direct manner. We will also provide a definition of crystalline cohomology as the unwinding of some quasi-ideal in $\mathbb{G}_a^{\mathrm{perf}}$.\vspace{2mm}

In order to achieve these goals, there is at least one immediate technical obstruction. The theory of derived de Rham cohomology works with commutative algebra objects in derived categories, whereas the notion of quasi-ideals only work with discrete rings. We are able to overcome this difficulty by crucially relying on the notion of QRSP algebras (\cref{covid}) due to the work of Bhatt, Morrow and Scholze \cite[Def. 8.8]{BMS19}. For our purpose, QRSP algebras are abundant enough such that derived de Rham cohomology can be completely understood by its values on them, and further, $\mathrm{dR}(S)$ is a discrete ring for a QRSP algebra $S.$ These properties make it possible to understand derived de Rham cohomology via the unwinding functor constructed in \cref{sec3.3}. \vspace{2mm}

The following definition is from \cite[Rmk. 2.2]{Bha12}.

\begin{definition}[Derived de Rham cohomology]\label{bye2}Derived de Rham cohomology is a functor denoted as $\mathrm{dR}$ from the $\infty$-category of simplicial commutative $\mathbb{F}_p$-algebras to the $\infty$-category of commutative algebra objects in the derived $\infty$-category $D(\mathbb{F}_p)$ obtained by left Kan extending the classical algebraic de Rham cohomology functor which sends a finitely generated polynomial $\mathbb{F}_p$-algebra $P$ to the algebraic de Rham complex of $P$, i.e., $$0 \to P \to \Omega^1 _{P/\mathbb{F}_p} \to\Omega^2_{P/\mathbb{F}_p}\to \ldots \to $$ 
\end{definition}{}

For a smooth $\mathbb{F}_p$-algebra, derived de Rham cohomology agrees with classical algebraic de Rham cohomology \cite[Cor. 3.10]{Bha12}. By definition, the derived de Rham cohomology functor $\mathrm{dR}$ is determined by its restriction to polynomial algebras over $\mathbb{F}_p.$  Further, since $\mathrm{dR}$ has quasisyntomic descent \cite[Ex. 5.12]{BMS19}, the functor $\mathrm{dR}$ restricted to polynomial algebras can be completely understood by restriction of $\mathrm{dR}$ to QRSP algebras by using descent along the map $P \to P_{\mathrm{perf}}$ for a polynomial algebra $P$ and the fact that each term in the Čech conerve of $P \to P_{\mathrm{perf}}$ is a QRSP algebra.
 
\begin{proposition}\label{citebms}Let $S$ be a QRSP algebra. Then $\mathrm{dR}(S) \simeq D_{S^\flat}(I)$ where $I := \mathrm{Ker}(S^\flat \to S).$
\end{proposition}{}

\begin{proof}This is \cite[Prop. 8.12]{BMS19}.\end{proof}{}

\begin{proposition}[Derived de Rham cohomology via unwinding]\label{NW2}As functors from $\mathrm{QRSP} \to \mathrm{Alg}_{\mathbb{F}_p}$, the functor $\mathrm{dR}$ and the functor ${\mathrm{Un}}(u^*W[F])$ are naturally isomorphic. Further, the Hodge filtration on $\mathrm{dR}(S)$ coincides with the Hodge filtration on ${\mathrm{Un}}(u^*W[F]) (S)$ constructed in \cref{sec3.4}.
\end{proposition}{}

\begin{proof}This follows from \cite[Prop. 8.12]{BMS19}, \cref{example3} and \cref{example4}.\end{proof}{}

\begin{remark}\label{bye666}
We point out that the functor $\w{dR}:\mathrm{QRSP} \to \mathrm{Alg}_{\mathbb{F}_p}$ can be seen as an object of $\mathrm{Fun}(\mathrm{QRSP}, \mathrm{Alg}_{\mathbb{F}_p})^{\otimes}_{\mathfrak{G}/}$. One may see this directly from \cite[Prop. 8.12]{BMS19} or use \cref{NW2}, the fact that $u^* W[F]$ is a nilpotent quasi-ideal in $\mathbb{G}_a^{\w{perf}}$ (\cref{socs}) and \cref{smoke6}.
\end{remark}

\begin{proposition}[Universal property of $\mathrm{dR}$]\label{univprop}As a functor from $\mathrm{QRSP} \to \mathrm{Alg}_{\mathbb{F}_p}$, $\mathrm{dR}$ is the final object of $\mathrm{Fun}(\mathrm{QRSP}, \mathrm{Alg}_{\mathbb{F}_p})_{\mathfrak{G}/}^{\mathrm{pure~rk}=1, \mathcal{N}{\mathrm{Un}}}.$
\end{proposition}{}

\begin{proof}This follows from \cref{univ} and \cref{NW2}. \end{proof}{}

\begin{proposition}\label{inhosp2}The natural transformation $\mathrm{gr}^0 : \mathrm{dR} \to \mathrm{id}$ coming from the Hodge filtration in derived de Rham cohomology is the unique natural transformation between $\mathrm{dR}$ and $\mathrm{id}$ viewed as objects of the category $\mathrm{Fun}(\mathrm{QRSP}, \mathrm{Alg}_{\mathbb{F}_p})_{\mathfrak{G}/}.$
\end{proposition}{}

\begin{proof}This follows from \cref{move2}. \end{proof}{}

\begin{proposition}\label{NW}The functor $\mathrm{dR}$ has no nontrivial endomorphisms as an object of $\mathrm{Fun} (\mathrm{QRSP}, \mathrm{Alg}_{\mathbb{F}_p})_{\mathfrak{G}/}.$
\end{proposition}{}

\begin{proof}By \cref{univprop}, $\mathrm{dR}$ is the final object in a full subcategory of $\mathrm{Fun} (\mathrm{QRSP}, \mathrm{Alg}_{\mathbb{F}_p})_{\mathfrak{G}/}$ which gives the claim. \end{proof}{}

\begin{proposition}\cite[Prop. 10.3.1]{BLM20}\label{bml} If we consider $\mathrm{dR}$ as a functor defined on smooth $\mathbb{F}_p$ algebras, then any endomorphism of $\mathrm{dR}$ that commutes with the $\mathrm{gr}^0$ map of the Hodge filtration $\mathrm{gr}^0:\mathrm{dR} \to \mathrm{id}$ is identity.
 
\end{proposition}{}

\begin{proof}By left Kan extension, we obtain an endomorphism of $\mathrm{dR}: \mathrm{QRSP} \to \mathrm{Alg}_{\mathbb{F}_p}.$
We check that this endomorphism is an endomorphism in the category $\mathrm{Fun}(\mathrm{QRSP}, \mathrm{Alg}_{\mathbb{F}_p})_{\mathfrak{G}/}.$ This will follow by functoriality of $\mathrm{dR}$ along the arrows $S ^\flat \to S$ for a QRSP algebra $S$ after noting that the endomorphism is identity when restricted to perfect rings. But the latter follows by the hypothesis that the endomorphism commutes with $\mathrm{gr}^0: \mathrm{dR} \to \mathrm{id}.$ Now by \cref{NW} and quasisyntomic descent, we obtain the desired statement. 
\end{proof}{}

Next we study the crystalline situation. Our goal is to prove that the theory of derived crystalline cohomology \cite[8.2]{BMS19} can be entirely recovered from a single quasi-ideal in $\mathbb{G}_a^{\mathrm{perf}}.$ By left Kan extension and quasisyntomic descent, one can again restrict attention to only QRSP algebras. We make the following definitions. 

\begin{definition}Let $S$ be a QRSP algebra. We define $ \mathbb{A}_{\mathrm{crys}}(S)$ to be the $p$-adic completion of the divided power envelope of $W(S^\flat) \to S.$ Here our divided powers are required to be compatible with those on $(p) \subset W(S^\flat).$
\end{definition}{}

\begin{remark}From the first part of \cite[Thm. 8.14]{BMS19}, it follows that $\mathbb{A}_{\mathrm{crys}}(S)$ is flat over $\mathbb{Z}_p$ for a QRSP algebra $S.$
\end{remark}{}

\begin{definition}\label{new}Let $(A, \mathfrak{m})$ be an Artinian local ring with residue field $\mathbb{F}_p$ and $S$ be a QRSP algebra. We will let $R \Gamma_{\mathrm{crys}} (S)_A:= \mathbb{A}_{\mathrm{crys}} (S) \otimes_{\mathbb{Z}_p} A.$
\end{definition}{}

We note that $R \Gamma_{\mathrm{crys}}(S)_A$ is a flat $A$-algebra. Further, $R \Gamma_{\mathrm{crys}}(S)_A \otimes_A \mathbb{F}_p \simeq \mathrm{dR} (S)$ by \cite[Prop. 8.12]{BMS19}. Our goal is to prove the following.

\begin{proposition}[Derived crystalline cohomology via unwinding]\label{NW3}The functor $R \Gamma_{\mathrm{crys}}(\cdot)_A: \mathrm{QRSP} \to \mathrm{Alg}_A$ is the unwinding of a quasi-ideal in $\mathbb{G}_a^{\mathrm{perf}}$ over $A.$
 \end{proposition}{}

\begin{proof}There is a functor $\mathfrak{G}:\mathrm{QRSP} \to \mathrm{Alg}_{A}$ sending $S \mapsto W_A(S^\flat).$ By the natural identifications $W_A(S^\flat) \simeq W(S^\flat) \otimes_{\mathbb{Z}_p} A \simeq \mathbb{A}_{\mathrm{crys}}(S^\flat) \otimes_{\mathbb{Z}_p} A = R \Gamma_{\mathrm{crys}} (S^\flat)_A$ and functoriality along $S^\flat \to S,$ we obtain a natural transformation $\mathfrak{G} \to R \Gamma_{\mathrm{crys}}(\cdot)_A.$ Thus we can view $R \Gamma_{\mathrm{crys}}$ as an object of $\mathrm{Fun}(\mathrm{QRSP}, \mathrm{Alg}_A)_{\mathfrak{G}/}.$ From \cref{NW2}, by going modulo $\mathfrak{m}$ and using flatness we can conclude that $R \Gamma_{\mathrm{crys}}$ satisfies the three conditions of \cref{smoke} and thus is an object of $\mathrm{Fun}(\mathrm{QRSP}, \mathrm{Alg}_A)_{\mathfrak{G}/}^{\otimes}.$ Therefore by \cref{smoke1}, $r (R \Gamma_{\mathrm{crys}})$ is a quasi-ideal in $\mathbb{G}_a^{\mathrm{perf}}$ over $A.$ By \cref{stu1}, there is a natural transformation ${\mathrm{Un}}(r (R\Gamma_{\mathrm{crys}})) \to R \Gamma_{\mathrm{crys}}.$ Since $R \Gamma_{\mathrm{crys}} \otimes_A \mathbb{F}_p \simeq \mathrm{dR}$, it follows that $r (R \Gamma_{crys})$ is a deformation of the quasi-ideal in $\mathbb{G}_a^{\mathrm{perf}}$ given by $u^*W[F].$ By \cref{gax1}, ${\mathrm{Un}}(r (R\Gamma_{\mathrm{crys}})) \otimes_A \mathbb{F}_p \simeq {\mathrm{Un}}(u^*W[F]) \simeq \mathrm{dR}.$ Thus the map ${\mathrm{Un}}(r (R\Gamma_{\mathrm{crys}})) \to R \Gamma_{\mathrm{crys}}$ is a natural isomorphism by the following lemma.

\begin{lemma}\label{new2}Let $(A, \mathfrak{m})$ be an Artinian local ring. Let $M \to N$ be a map of $A$-modules where $N$ is flat. Suppose that $M \otimes_A A/\mathfrak{m} \to N \otimes_A A/\mathfrak{m}$ is an isomorphism. Then the map $M \to N$ is an isomorphism. 
\end{lemma}{}

\begin{proof}Since $\mathfrak{m}$ is nilpotent, it follows that $M \to N$ is surjective. Let $K := \mathrm{Ker}(M \to N).$ Then we have an exact sequence $0 \to K \to M \to N \to 0.$ By flatness of $N$, we get an exact sequence $$0 \to K/\mathfrak{m}K \to M/\mathfrak{m}M \to N/\mathfrak{m}N \to 0.$$ By hypothesis, we must have $K /\mathfrak{m}K = 0.$ Again, since $\mathfrak{m}$ is nilpotent, this implies $K=0$, which proves the lemma. \end{proof}{}

Indeed, for every QRSP algebra $S$, we have a map ${\mathrm{Un}}(r (R\Gamma_{\mathrm{crys}}))(S) \to R \Gamma_{\mathrm{crys}}(S)_A$ which is an isomorphism modulo $\mathfrak{m}$ and $R \Gamma_{\mathrm{crys}}(S)_A$ is flat over $A.$
Thus the map must be an isomorphism by the lemma.
\end{proof}{}

\begin{remark}\label{whatmovie?7}
The functor $\w{QRSP} \to D(\mathbb{Z}_p)$ that sends $S \mapsto A_{\w{crys}}(S)$ defines a sheaf for the quasisyntomic topology \cite[Lemma 4.27,~Rmk.~8.15]{BMS19}. We note that if $(A, \mathfrak{m})$ is any Artinian local ring with residue field $\mathbb{F}_p$, then the functor $(\,\cdot\,)\otimes^L_{\mathbb{Z}_p} A : D(\mathbb{Z}_p)\to D(A)$ preserves all limits. Indeed, by using the $\mathfrak{m}$-adic filtration on $A$ (which is finite and the graded pieces are finite dimensional $\mathbb{F}_p$-vector spaces), this boils down to showing that $(\,\cdot\,)\otimes^L_{\mathbb{Z}_p} \mathbb{F}_p$ preserves all limits; the latter claim follows because $\mathbb{F}_p$ is quasi-isomorphic to the two term complex $(\mathbb{Z}_p \xrightarrow{p} \mathbb{Z}_p).$ This implies that the functor $R\Gamma_{\w{crys}} (\,\cdot\,)_A$ from \cref{new} maybe viewed as a quasisyntomic sheaf (with values in $D(A)$). By \cite[Lemma~4.31]{BMS19}, we obtain an extended functor still denoted as $R\Gamma_{\w{crys}}(\,\cdot\,)_A: \w{QSyn}_{\mathbb{F}_p} \to D(A),$ where $\mathrm{QSyn}_{\mathbb{F}_p}$ denotes the category of quasisyntomic $\mathbb{F}_p$-algebras \cite[Def.~1.7]{BMS19}. By \cite[Ex.~5.12,~Prop.~8.12]{BMS19}, it follows that $R\Gamma_{\w{crys}}(\,\cdot\,)_A \otimes^L _A A/\mathfrak{m} \simeq \w{dR}$ as functors from $\w{QSyn}_{\mathbb{F}_p} \to D(A)$. Note that $R\Gamma_{\w{crys}}(\,\cdot\,)_A$ can be naturally enhanced as a functor from $\w{QSyn}_{\mathbb{F}_p} \to \w{CAlg}(D(A))$ and the isomorphism $R\Gamma_{\w{crys}}(\,\cdot\,)_A \otimes^L _A A/\mathfrak{m} \simeq \w{dR}$ remains true at the level of functors from $\w{QSyn}_{\mathbb{F}_p} \to \w{CAlg}(D(A)).$
\end{remark}{}

\begin{remark}By the proof of \cref{NW3}, we also saw that the quasi-ideal in $\mathbb{G}_a^{\mathrm{perf}}$ given by $r (R \Gamma_{\mathrm{crys}})$ is a deformation of $u^*W[F]$ over the ring $A.$ Once we know this description, it is not difficult to describe $r (R \Gamma_{\mathrm{crys}})$ explicitly. On the other hand, according to the proposition, this quasi-ideal recovers derived crystalline cohomology for QRSP algebras via the unwinding functor. Therefore, combining with \cref{whatmovie?7}, this gives a way of defining crystalline cohomology without mentioning divided power structures and using the pointed $\mathbb{G}_a^{\mathrm{perf}}$-module structure instead.
\end{remark}{}

\newpage

\section{Formal \'etaleness of de Rham cohomology}\label{section5}In this final section, our goal is to prove our main theorem that the functor $\mathrm{dR}$ is formally \'etale. More precisely, we prove the following

\begin{theorem}\label{mainthm}Let $$\mathrm{dR} : \mathrm{Alg}^{\mathrm{sm}}_{\mathbb F _p} \to \mathrm{CAlg}(D(\mathbb{F}_p))$$
be the algebraic de Rham cohomology functor defined on the category of smooth $\mathbb{F}_p$-algebras $\mathrm{Alg}^{\mathrm{sm}}_{\mathbb F _p}$. Given an Artinian local ring $(A, \mathfrak{m})$ with residue field $\mathbb{F}_p$, the functor $\mathrm{dR}$ admits a unique deformation $$\mathrm{dR}':  \mathrm{Alg}^{\mathrm{sm}}_{\mathbb F _p} \to \mathrm{CAlg}(D(A)) .$$ Further, the deformation $\mathrm{dR}'$ is unique up to unique isomorphism. Here a deformation is supposed to mean the data of isomorphism of functors $\mathrm{dR}' \otimes_A ^{L} \mathbb{F}_p \simeq \mathrm{dR}.$ More precisely, the space of deformations of $\mathrm{dR}$ (as defined in \cref{bbleave}) is contractible. \textit{cf.}~\cref{hiltra34}.
\end{theorem}{}

\begin{remark}\label{bbleave}
In this remark, we clarify the $\infty$-categorical technicalities underlying the notion of the ``space of deformations" that appears in \cref{mainthm}. Note that we have a functor between $\infty$-categories $$ \w{Fun}(\w{Alg}_{\mathbb{F}_p}^{\w{sm}}, \w{CAlg}(D(A))) \to \w{Fun}(\w{Alg}_{\mathbb{F}_p}^{\w{sm}}, \w{CAlg}(D(\mathbb{F}_p)))$$ obtained by using the functor $(\, \cdot \,) \otimes_A^{L} \mathbb{F}_p.$ Now $\w{dR}$ can be seen as an object of $\w{Fun}(\w{Alg}_{\mathbb{F}_p}^{\w{sm}}, \w{CAlg}(D(\mathbb{F}_p))).$ We define an $\infty$-category $$\mathcal{D}\w{ef}(\mathrm{dR}):= \w{Fun}(\w{Alg}_{\mathbb{F}_p}^{\w{sm}}, \w{CAlg}(D(A))) \times_{\w{Fun}(\w{Alg}_{\mathbb{F}_p}^{\w{sm}}, \w{CAlg}(D(\mathbb{F}_p)))} \left \{\w{dR} \right \},$$ where the fiber product is taken in the $\infty$-category $\widehat{\mathcal{C}\w{at}}_{\infty}$ of (not necessarily small) $\infty$-categories \cite[Def. 5.5.3.1]{Luuu}. An object of $\mathcal{D}\w{ef}(\mathrm{dR})$ is a functor $\mathrm{dR}':  \mathrm{Alg}^{\mathrm{sm}}_{\mathbb F _p} \to \mathrm{CAlg}(D(A))$ equipped with the data of an isomorphism of functors $\mathrm{dR}' \otimes_A ^{L} \mathbb{F}_p \simeq \mathrm{dR}.$ One notes that every morphism in $\mathcal{D}\w{ef}(\mathrm{dR})$ is an equivalence. Thus $\mathcal{D}\w{ef}(\mathrm{dR})$ is an $\infty$-groupoid, which we informally call ``the space of deformations" of $\w{dR}$ (see \cite[Def. 1.2.5.1, Rmk. 1.2.5.2  ]{Luuu}). \cref{mainthm} asserts that $\mathcal{D}\w{ef}(\mathrm{dR}) \simeq \left \{ * \right \}.$ Note that the existence of crystalline cohomology (\textit{cf.}~\cref{whatmovie?7}) assures that $\mathcal{D}\w{ef}(\mathrm{dR})$ is not the empty $\infty$-groupoid.
\end{remark}{}

\begin{remark}\label{bye1}
We note that the functor $\mathrm{dR} : \mathrm{Alg}^{\mathrm{sm}}_{\mathbb F _p} \to \mathrm{CAlg}(D(\mathbb{F}_p))$ is left Kan extended from its restriction to $\w{Poly}_{\mathbb{F}_p},$ where the latter denotes the category of finitely generated polynomial $\mathbb{F}_p$-algebras; this was observed in \cite[Cor. 3.10]{Bha12} and is a consequence of the derived Cartier isomorphism \cite[Prop. 3.5]{Bha12}. In fact, according to Def. 2.1 and Rmk. 2.2 loc. cit. when we consider derived de Rham cohmology, as a functor from the $\infty$-category of simplicial commutative $\mathbb{F}_p$-algebras to the $\infty$-category $\mathrm{CAlg}(D(\mathbb{F}_p)),$ it preserves all colimits. Concretely, the last fact implies that derived de Rham cohomology functor sends a Tor independent pushout diagram of ordinary rings to pushout diagrams of $\mathbb{E}_{\infty}$-algebras.
\end{remark}{}

\subsection{First proof using deformation theory of $u^*W[F]$}\label{section5.1}
First we make some adjustments so that the values taken by $\mathrm{dR}$ are discrete rings as opposed to commutative algebra objects in a derived category. As noted in \cref{bye1}, the functor $\mathrm{dR} : \mathrm{Alg}^{\mathrm{sm}}_{\mathbb F _p} \to \mathrm{CAlg}(D(\mathbb{F}_p))$ is left Kan extended from its restriction to $\w{Poly}_{\mathbb{F}_p}.$ If $\mathrm{dR}':  \mathrm{Alg}^{\mathrm{sm}}_{\mathbb F _p} \to \mathrm{CAlg}(D(A))$ is a deformation of $\mathrm{dR}$ as in \cref{mainthm}, by going derived modulo $\mathfrak{m},$ one sees that the same property holds for $\mathrm{dR}'.$ Thus, by considering left Kan extensions, in order to prove the statement regarding unique deformation in \cref{mainthm}, it would be enough to prove the same statement for the functor $$\mathrm{dR}: \mathrm{Poly}_{\mathbb{F}_p} \to \mathrm{CAlg}(D(\mathbb{F}_p)) .$$ Let $\mathrm{dR}': \mathrm{Poly}_{\mathbb{F}_p} \to \mathrm{CAlg}(D(A))$ be a deformation of $\mathrm{dR}.$ To prove that $\mathrm{dR}'$ is unique up to unique isomorphism, we can again do a left Kan extension to obtain a functor $\mathrm{dR}': \mathrm{QSyn}_{\mathbb{F}_p} \to \mathrm{CAlg}(D(A))$
which extends $\mathrm{dR}'$ from polynomial algebras, where $\mathrm{QSyn}_{\mathbb{F}_p}$ denotes the category of quasisyntomic $\mathbb{F}_p$-algebras \cite[Def.~1.7]{BMS19}. By construction, $\mathrm{dR}'$ is a deformation of the derived de Rham cohomology functor $\mathrm{dR}: \mathrm{QSyn}_{\mathbb{F}_p} \to \mathrm{CAlg}(D(\mathbb{F}_p)).$
Now the category of quasisyntomic $\mathbb{F}_p$-algebras contains all the QRSP algebras. Therefore one can try to recover the functor $\mathrm{dR}'$ restricted to finitely generated polynomial algebras $P$ via descent along the (faithfully flat) map $P \to P_{\mathrm{perf}}.$ Here $P_{\mathrm{perf}}:= \mathrm{colim}_{x \to x^p} P.$ The following lemma guarantees that this is possible by using the notion of descendability \cite[Def. 3.18]{Mat16}.

\begin{lemma}\label{aroma}Let $P \in \mathrm{Poly}_{\mathbb{F}_p}$. Then the map $\mathrm{dR}'(P) \to \mathrm{dR}'(P_{\mathrm{perf}})$ is descendable.
\end{lemma}{}

\begin{proof}This follows from the fact that $\mathrm{dR}(P) \to \mathrm{dR}(P_{\mathrm{perf}})$ is descendable \cite[Lemma 8.6]{BS19} and the fact that descendability can be checked (derived) modulo $\mathfrak{m}$ since $\mathfrak{m}$ is nilpotent. The latter claim follows from \cite[Prop 3.24]{Mat16} and \cite[Prop 3.35]{Mat16}. The fact that descendability implies descent along $P \to P_{\mathrm{perf}}$ for $\mathrm{dR}'$ follows from the fact that $\mathrm{dR}'$ sends Tor independent pushout diagrams in $\mathrm{QSyn}_{\mathbb{F}_p}$ to pushouts in $\w{CAlg}(D(A))$ and \cite[Prop. 3.20]{Mat16}. The former fact can be seen by using the same property for $\w{dR}$ as noted in \cref{bye1} and going derived modulo $\mathfrak{m}.$ \end{proof}{}

By using the above lemma and the fact that each term in the Čech conerve of $P \to P_{\mathrm{perf}}$ is a QRSP algebra, in order to prove \cref{mainthm}, it is enough to prove the following statement.

\begin{itemize}
 \item Let $\mathrm{dR} :\mathrm{QRSP} \to \mathrm{CAlg}(D(\mathbb{F}_p))$ be the derived de Rham cohomology functor. Given an Artinian local ring $(A, \mathfrak{m})$ with residue field $\mathbb{F}_p,$ the functor $\mathrm{dR}$ admits a deformation $\mathrm{dR}':\mathrm{QRSP} \to \mathrm{CAlg}(D(A))$ which is unique up to unique isomorphism.
\end{itemize}

Now we note again that $\mathrm{dR}(S)$ is a discrete ring for a QRSP algebra $S.$ Thus $\mathrm{dR}'$ is also a discrete ring which is \textit{flat} over $A$ by \cite[Tag 051H]{SP}. Thus it is enough to prove the following statement (\cref{why?}). 

\begin{itemize}
\item Let $\mathrm{dR} :\mathrm{QRSP} \to \mathrm{Alg}_{\mathbb{F}_p}$ be the derived de Rham cohomology functor. Given an Artinian local ring $(A, \mathfrak{m})$ with residue field $\mathbb{F}_p,$ the functor $\mathrm{dR}$ admits a deformation $\mathrm{dR}':\mathrm{QRSP} \to \mathrm{Alg}_A$ which is unique up to unique isomorphism.
\end{itemize}

Now we appeal to the general theory of deformations of (commutative) ring objects in a topos due to Illusie \cite{Ill71} to reduce the problem further. Let $\mathcal{X}$ denote the topos $\mathrm{PShv}(\mathrm{QRSP}^{\mathrm{op}})$ of presheaves of sets on $\mathrm{QRSP}^{\mathrm{op}}.$ Then $\mathrm{dR}$ can be viewed as an $\mathbb{F}_p$-algebra object in $\mathcal{X}$ and we are trying to understand deformations $\mathrm{dR}'$ of $\mathrm{dR}$ which are flat $A$-algebra objects. Since $(A, \mathfrak{m})$ is an Artinian local ring, $\mathfrak{m}^n=0$ for some $n>0.$ By considering the $\mathfrak{m}$-adic filtration, we have a tower $A= A/\mathfrak{m}^n \to A/\mathfrak{m}^{n-1} \to \ldots \to A/\mathfrak{m}^2 \to A/\mathfrak{m} = \mathbb{F}_p,$ where kernel of each map is a square-zero ideal. Note that by \cref{new}, we already know that there exists a deformation of $\mathrm{dR}$ over $A$ given by $\mathrm{dR}_A:= R \Gamma_{\mathrm{crys}}(\,\cdot\,)_A$ (which is a flat $A$-algebra object). Therefore, in order to prove that $\w{dR}_A$ is unique up to unique isomorphism as a deformation of $\w{dR},$ it is enough to prove that $\w{dR}_{A/\mathfrak{m}^{r+1}}:= \w{dR}_A \otimes_A A/\mathfrak{m}^{r+1}$ is unique up to unique isomorphism as a deformation of $\w{dR}_{A/\mathfrak{m}^{r}}:= \w{dR}_A \otimes_A A/\mathfrak{m}^{r}$ for all $1 \le r \le n-1.$ \vspace{2mm}

To this end, we will study the problem of deforming $\w{dR}_{A/\mathfrak{m}^{r}}$ (considered as an $A/\mathfrak{m}^{r}$-algebra object) along the morphism $A/\mathfrak{m}^{r+1} \to A/\mathfrak{m}^{r}$. By construction, $\w{dR}_{A/\mathfrak{m}^{r+1}}$ is already one such deformation; therefore, the obstruction class of deforming $\w{dR}_{A/\mathfrak{m}^{r}}$ as in \cite[Cor. 2.1.3.3~(i)]{Ill71} vanishes. In order to prove that $\w{dR}_{A/\mathfrak{m}^{r+1}}$ is unique up to unique isomorphism as a deformation, by \cite[Prop. 2.1.2.3,~Cor. 2.1.3.3~(ii),(iii)]{Ill71}, it is equivalent to show that \begin{equation}\label{shapiro1219}
    \w{Ext}^i_{\w{dR}_{A/\mathfrak{m}^{r}}} (\mathbb{L}_{\w{dR}_{A/\mathfrak{m}^{r}}}, \mathfrak{m}^r/ \mathfrak{m}^{r+1} \otimes_{A/\mathfrak{m}^r} \w{dR}_{A/\mathfrak{m}^{r}})=0
\end{equation}for $i \in \left \{ 0,1  \right \};$ here $\mathbb{L}_{\w{dR}_{A/\mathfrak{m}^{r}}}$ denotes the contangent complex of the $A/\mathfrak{m}^{r}$-algebra object $\w{dR}_{A/\mathfrak{m}^{r}}$ relative to $A/\mathfrak{m}^{r}$ as considered in  \cite[Cor. 2.1.3.3]{Ill71}, which can be viewed as an object in the derived category of $\w{dR}_{A/\mathfrak{m}^{r}}$-module objects. Now we note that $\mathfrak{m}^r/ \mathfrak{m}^{r+1}$ is naturally an $A/\mathfrak{m}$-module, and $\mathfrak{m}^r/ \mathfrak{m}^{r+1} \otimes_{A/\mathfrak{m}^r} \w{dR}_{A/\mathfrak{m}^{r}} \simeq \mathfrak{m}^r/ \mathfrak{m}^{r+1} \otimes^{L} _{A/\mathfrak{m}^r} \w{dR}_{A/\mathfrak{m}^{r}} \simeq \mathfrak{m}^r/ \mathfrak{m}^{r+1} \otimes^{L} _{A/\mathfrak{m}} A/\mathfrak{m} \otimes^{L} _{A/\mathfrak{m}^r} \w{dR}_{A/\mathfrak{m}^{r}} \simeq \mathfrak{m}^r/ \mathfrak{m}^{r+1} \otimes^{L} _{A/\mathfrak{m}}\w{dR}$ (where we can switch between the derived and nonderived tensor product by using flatness of $\w{dR}_A$ as an $A$-algebra object). Therefore, we have
$$\w{Ext}^i_{\w{dR}_{A/\mathfrak{m}^{r}}} (\mathbb{L}_{\w{dR}_{A/\mathfrak{m}^{r}}}, \mathfrak{m}^r/ \mathfrak{m}^{r+1} \otimes^L _{A/\mathfrak{m}^r} \w{dR}_{A/\mathfrak{m}^{r}}) \simeq \w{Ext}^i_{\w{dR}} (\mathbb{L}_{\w{dR}_{A/\mathfrak{m}^{r}}}\otimes^L _{\w{dR}_{A/\mathfrak{m}^{r}}} \w{dR}, \mathfrak{m}^r/ \mathfrak{m}^{r+1} \otimes^L _{A/\mathfrak{m}} \w{dR}).$$ Using $\mathbb{L}_{\w{dR}}$ to denote $\mathbb{L}_{\w{dR}_{A/\mathfrak{m}}}$ defined above, by the base change formula \cite[Prop. 2.2.1]{Ill71}, we have $\mathbb{L}_{\w{dR}_{A/\mathfrak{m}^{r}}}\otimes^L _{\w{dR}_{A/\mathfrak{m}^{r}}}\w{dR} \simeq \mathbb{L}_{\w{dR}}.$ In order to prove the vanishing in equation \cref{shapiro1219}, equivalently, we need to prove that \begin{equation}\label{shapir29}
    \w{Ext}^i_{\w{dR}} (\mathbb{L}_{\w{dR}}, \mathfrak{m}^r/ \mathfrak{m}^{r+1} \otimes_{\mathbb{F}_p} \w{dR})=0
\end{equation}{}for $i \in \left \{ 0,1  \right \}.$ Since $A$ is an artinian local ring with residue field $\mathbb{F}_p,$ it follows that $\mathfrak{m}^r/\mathfrak{m}^{r+1}$ is a finite dimensional $\mathbb{F}_p$-vector space. Therefore, to prove the vanishing in \cref{shapir29}, it is enough to show that $\w{Ext}^i_{\w{dR}} (\mathbb{L}_{\w{dR}}, \w{dR})=0$ for $i \in \left \{ 0,1  \right \}.$ Equivalently, it is enough to prove that equation \cref{shapir29} holds for $i \in \left \{ 0,1  \right \}$ in the special case when $A = \mathbb{F}_p[\epsilon]/\epsilon^2,$ $\mathfrak{m} = (\epsilon)$ and $r=1.$ Finally, arguing backwards and using the reductions we have made so far, we see that in order to prove \cref{mainthm}, it is enough to prove the following proposition.

\begin{proposition}\label{new4} Let $\mathrm{dR} :\mathrm{QRSP} \to \mathrm{Alg}_{\mathbb{F}_p}$ be the derived de Rham cohomology functor. Then $\mathrm{dR}$ has no nontrivial deformation to $\mathbb{F}_p[\epsilon]:=\mathbb{F}_p [\epsilon]/\epsilon^2.$ Further, the deformation is unique up to unique isomorphism. (Here the trivial deformation is given by tensoring up to $\mathbb{F}_p [\epsilon].$)
 \end{proposition}{}

\begin{proof}We note that there is a natural transformation $\mathfrak{G} \to \mathrm{dR}$ where $\mathfrak{G}(S) = S^\flat.$ Thus $\mathrm{dR}$ can be viewed as an object of $\mathrm{Fun}(\mathrm{QRSP}, \mathrm{Alg}_{\mathbb{F}_p})_{/\mathfrak{G}}.$ Since cotangent complex of a perfect ring vanishes, the maps $S^\flat \to \mathrm{dR}(S)$ lifts uniquely to maps $S^\flat[\epsilon] \to \mathrm{dR}'(S)$ for any deformation $\mathrm{dR}'$ of $\mathrm{dR}.$ It follows that any deformation of $\mathrm{dR}$ and any endomorphism of $\mathrm{dR}$ as a deformation can be studied as deformations and endomorphisms in the category $\mathrm{Fun}(\mathrm{QRSP}, \mathrm{Alg}_{\mathbb{F}_p[\epsilon]})_{/\mathfrak{G}}.$ Further, as noted in \cref{bye666}, $\w{dR}$ can be viewed as an object of $\mathrm{Fun}(\mathrm{QRSP}, \mathrm{Alg}_{\mathbb{F}_p})^{\otimes}_{\mathfrak{G}/}.$ Therefore, by going modulo $\epsilon$ and using flatness, we see that any deformation of $\w{dR}$ satisfies the three conditions of \cref{smoke} and can be viewed as an object of $\mathrm{Fun}(\mathrm{QRSP}, \mathrm{Alg}_{\mathbb{F}_p[\epsilon]})_{\mathfrak{G}/}^{\otimes}.$ Thus we can also work inside the category $\mathrm{Fun}(\mathrm{QRSP}, \mathrm{Alg}_{\mathbb{F}_p[\epsilon]})_{/\mathfrak{G}}^\otimes.$ Writing $\mathrm{dR}'$ for a deformation of $\mathrm{dR},$ by using \cref{smoke1} and \cref{NW2} we see that $r (\mathrm{dR}')$ is a deformation of $u^*W[F]$ as a quasi-ideal in $\mathbb{G}_a^{\mathrm{perf}}$. By \cref{stu1}, there is a natural transformation ${\mathrm{Un}}(r (\mathrm{dR}')) \to \mathrm{dR}'.$ Going modulo $\epsilon$ and using \cref{new2}, this natural transformation is actually a natural isomorphism. By \cref{la} it would be enough to show that any deformation of $u^*W[F]$ to $\mathbb{F}_p [\epsilon]$ is uniquely isomorphic to the trivial deformation. This follows from \cref{new3} and \cref{mainthm1}. \end{proof}{}

\begin{remark}\label{hiltra34}
We note that the statement of \cref{mainthm} remains valid when $\mathbb{F}_p$ is replaced by any perfect field of characteristic $p.$ Indeed, the key calculation (involving deformations) that goes into the proof above relies on \cref{mainthm1}, which, as noted in \cref{chmv981}, remains valid over any perfect field of characteristic $p.$ Furthermore, the theory of $\mathbb{G}_a^{\w{perf}}$-modules developed in this paper and the inputs we use from \cite{BMS19} (such as \cite[Prop. 8.12]{BMS19}) remain valid when the base is an arbitrary perfect field as well. We thank the referee for pointing this out.
\end{remark}{}

\subsection{Second proof using deformation theory of $W[F]$}\label{sec5.2}
The goal of this subsection is to provide a slightly different proof of \cref{new4} that avoids the use of deformation theory of $u^*W[F]$, i.e., \cref{mainthm1}. Instead, we would use the deformation theory of $W[F]$ which is much easier to understand by universal properties (\cref{mad}) and the rigidity of the Hodge map (\cref{hodgestability}).

\begin{lemma}\label{d1}Let $\mathrm{dR}: \mathrm{QRSP} \to \mathrm{Alg}_{\mathbb{F}_p} $ be the derived de Rham cohomology functor. Let $\mathrm{dR}': \mathrm{QRSP} \to \mathrm{Alg}_{\mathbb{F}_p[\epsilon]}$ be a deformation of $\mathrm{dR}.$ Then there exists a functor $$\mathrm{dR}'' : \mathrm{Poly}_{\mathbb F_p} \to \mathrm{CAlg}(D(\mathbb{F}_p[\epsilon]))$$ which is a deformation of $\mathrm{dR}$ such that $\mathrm{dR}'$ is the restriction of left Kan extension of $\mathrm{dR}''$ to QRSP algebras.
\end{lemma}{}

\begin{proof}Let $P$ be a finitely generated polynomial algebra over $\mathbb{F}_p.$ Let $C^\bullet (P)$ denote the Čech conerve of the map $P \to P_{\mathrm{perf}}.$ We note that $C^\bullet (P)$ is a cosimplicial object in the category $\mathrm{QRSP}$. Applying $\mathrm{dR}'$ to $C^\bullet(P)$ we obtain a cosimplicial $\mathbb{F}_p[\epsilon]$-algebra and we define $\mathrm{dR}'' (P)$ to be the totalization of this cosimplicial algebra as an object of $\mathrm{CAlg}(D(\mathbb{F}_p[\epsilon])).$ This defines the functor $$\mathrm{dR}'' : \mathrm{Poly}_{\mathbb F_p} \to \mathrm{CAlg}(D(\mathbb{F}_p[\epsilon])).$$
First, we check that $\mathrm{dR}''$ defined as above is indeed a deformation of $\mathrm{dR}.$ By definition $\mathrm{dR}'' (P) = \mathrm{Tot} (\mathrm{dR}'(C^\bullet(P)))$ which is an inverse limit of the pro-object $ \left \{\mathrm{Tot}_{n}(\mathrm{dR}'(C^\bullet(P)))\right \}.$ Since $\mathrm{dR}(P) \to \mathrm{dR}(P_{\mathrm{perf}})$ is descendable \cite[Lemma 8.6]{BS19}, it follows that $\left \{\mathrm{Tot}_{n}(\mathrm{dR}(C^\bullet(P)))\right \}$ is a pro-constant pro-object. Therefore, by \cref{below}, $\left \{\mathrm{Tot}_{n}(\mathrm{dR}'(C^\bullet(P)))\right \}$ is a pro-constant pro-object. This proves that $\mathrm{Tot}(\mathrm{dR}'(C^\bullet(P))) \otimes_{\mathbb{F}_p[\epsilon]} \mathbb{F}_p \simeq \mathrm{Tot} (\mathrm{dR} (C^\bullet (P)))$; thus $\mathrm{dR}''$ is indeed a deformation of $\mathrm{dR}.$\vspace{2mm}

Next, we need to check that the left Kan extension of $\mathrm{dR}''$ is naturally isomorphic to $\mathrm{dR}'$ on QRSP algebras. To do so, we will first construct a natural map $\mathrm{dR}''(P) \to \mathrm{dR}'(S)$ for a map $P \to S$ where $P$ is a finitely generated polynomial $\mathbb{F}_p$-algebra and $S$ is a QRSP algebra. Let $C^\bullet(P)$ denote the Čech conerve of $P \to P_{\mathrm{perf}}$ and $C^\bullet (S/P)$ denote the Čech conerve of $S \to S\otimes_P P_{\mathrm{perf}}.$ There is a natural map of cosimplicial rings $C^\bullet (P) \to C^\bullet (S/P).$ Applying $\mathrm{dR}'$ we obtain a map $\mathrm{dR}' (C^\bullet (P)) \to \mathrm{dR}' (C^\bullet (S/P)).$ Thus in order to construct the map $\mathrm{dR}''(P) \to \mathrm{dR}'(S)$, it would be enough to show that $\mathrm{Tot} (\mathrm{dR}' (C^\bullet (S/P))) \simeq \mathrm{dR}'(S)$, since by definition $\mathrm{Tot}(\mathrm{dR}'(C^\bullet(P)) )\simeq \mathrm{dR}''(P).$ For that, it would be enough to show that $\mathrm{dR}' (S) \to \mathrm{dR}'(S \otimes_P P_{\mathrm{perf}})$ is descendable. But that follows (by going modulo $\epsilon$) since $\mathrm{dR} (S) \to \mathrm{dR}(S \otimes_P P_{\mathrm{perf}})$ is descendable; this is true because the latter map is a base change of the descendable map $\mathrm{dR}(P) \to \mathrm{dR}(P_{\mathrm{perf}}).$ Going back, this constructs the required map $\mathrm{dR}'' (P) \to \mathrm{dR}'(S).$ Now, writing $\mathrm{dR}''(S)$ for the left Kan extension of $\mathrm{dR}''$ evaluated at $S$, we obtain a natural map $\mathrm{dR}''(S) \to \mathrm{dR}'(S).$ By construction, this map is an isomorphism after going (derived) modulo $\epsilon$ and thus must be an isomorphism.
\end{proof}{}

The following lemma was used in the above proof.

\begin{lemma}\label{below}Let $R$ be a ring and $I$ be a nilpotent ideal. Let $\left \{B_n\right \}$  be a pro-object in $D(R)$. If $\left \{B_n \otimes_R R/I\right \}$ is a pro-constant pro-object then $\left \{B_n \right \}$ is pro-constant itself.
\end{lemma}{}

\begin{proof}
Let $C$ be the collection of objects $u$  in $D(R)$ such that $\left \{B_n \otimes_R u\right \}$ is a pro-constant pro-system. Then $C$ is a thick tensor-ideal which contains $R/I$ by assumption. Since $R$ has a finite filtration whose graded pieces are $R/I$-modules (i.e., the $I$-adic filtration) it follows that $R \in C$. Since $R$ is the unit under tensor this gives that $\left \{B_n\right \}$ is pro-constant. 
\end{proof}{}

\begin{proof}[Second proof of \cref{new4}]We follow the notations and the strategy from the first proof of \cref{new4}. Instead of invoking \cref{mainthm1}, by \cref{mad}, it would be enough to show that the pointed $\mathbb{G}_a^{\mathrm{perf}}$-module $r(\mathrm{dR}')$ over $\mathbb{F}_p [\epsilon]$ is isomorphic to $u^*X$ where $X$ is a pointed $\mathbb{G}_a$-module which is a deformation of $W[F].$ Here the functor $u^*$ is from \cref{pullback}.\vspace{2mm}

We note that for derived de Rham cohomology, there is a natural transformation $\mathrm{gr}^0: \mathrm{dR}(A) \to A$ obtained from taking $\mathrm{gr}^0$ of the Hodge filtration. By left Kan extension from the case of polynomial algebras, one sees that the relative derived de Rham cohomology satisfies $\mathrm{dR}_{B/A} \simeq \mathrm{dR}(B)\otimes_{\mathrm{dR}(A)} A$, where the tensor product is taken using the $\mathrm{gr}^0$ map. This isomorphism also appears in \cite[Prop. 3.11]{GL20}. Using \cref{hodgestability} and applying the unwinding functor, one obtains a natural transformation $\mathrm{dR}' \to \mathrm{id}[\epsilon]$ as functors from $\mathrm{QRSP} \to \mathrm{Alg}_{\mathbb{F}_p[\epsilon]}$ deforming the $\mathrm{gr}^0: \mathrm{dR} \to \mathrm{id}$ transformation. Using the construction as in \cref{d1} and left Kan extension, this extends to a natural transformation $\mathrm{dR}'' \to \mathrm{id}[\epsilon]$ as functors from $\mathrm{Alg}_{\mathbb{F}_p} \to \mathrm{CAlg}(D(\mathbb{F}_p [\epsilon])).$\vspace{2mm}

We will define $\mathrm{dR}'''$ on the category of arrows $(A \to B)$ of $\mathbb{F}_p$-algebras. We define $\mathrm{dR}'''(A \to B):= \mathrm{dR}'' (B) \otimes_{\mathrm{dR}'' (A)} A[\epsilon]$ where we use the map $\mathrm{dR}''(A) \to A[\epsilon].$ Now this takes values in $\mathrm{CAlg}(D(\mathbb{F}_p [\epsilon]))$ but we will only use it in case of objects $(A \to B)$ where it would give a discrete ring as output. We note that $\mathrm{dR}'''(A \to A) \simeq A[\epsilon]$ and for a QRSP algebra $S$, $\mathrm{dR}'''(S^\flat \to S) \simeq \mathrm{dR}'' (S) \simeq \mathrm{dR}'(S),$ where the last isomorphism comes from \cref{d1}. Further $\mathrm{dR}''' (\mathbb{F}_p [x] \to \mathbb{F}_p)$ is a discrete ring as it is so (derived) modulo $\epsilon$ by using \cite[Lemma 3.29]{Bha12}. It also follows that $\mathrm{dR}''' (\mathbb{F}_p [x] \to \mathbb{F}_p)$ is a flat algebra over $\mathbb{F}_p[\epsilon].$ We record two lemmas.

\begin{lemma}$\mathrm{Spec}\,\mathrm{dR}'''(\mathbb{F}_p[x]\to \mathbb{F}_p)$ has the structure of a pointed $\mathbb{G}_a$-module over $\mathbb{F}_p[\epsilon]$ and it is a deformation of $W[F]$ as a pointed $\mathbb{G}_a$-module.
\end{lemma}{}

\begin{proof}We note that in the arrow category of $\mathbb{F}_p$-algebras, the object $(\mathbb{F}_p[x] \to \mathbb{F}_p[x])$ is a coring, i.e., it corepresents the functor $(A\to B) \to A$ which is naturally valued in rings. Further the object $(\mathbb{F}_p[x] \to \mathbb{F}_p)$ is a a comodule over $(\mathbb{F}_p[x] \to \mathbb{F}_p[x])$ since $(\mathbb{F}_p[x] \to \mathbb{F}_p)$ corepresents the functor $(A \to B) \to \mathrm{Ker}(A \to B)$ which is naturally an $A$-module and further admits a map of $A$-modules $\mathrm{Ker}(A \to B) \to A.$ This provides a map $(\mathbb{F}_p[x] \to \mathbb{F}_p[x]) \to (\mathbb{F}_p[x] \to \mathbb{F}_p)$ of $(\mathbb{F}_p[x] \to \mathbb{F}_p[x])$-comodules. Applying the de Rham cohomology functor $\mathrm{dR}$ to this yields a map $\mathbb{F}_p[x] \to \mathrm{dR}(\mathbb{F}_p[x] \to \mathbb{F}_p)$ of $\mathbb{F}_p[x]$-comodules. Using the fact that $\mathrm{dR}(\mathbb{F}_p[x] \to \mathbb{F}_p) \simeq \mathbb{F}_p \langle x \rangle $ \cite[Lemma 3.29]{Bha12} and \cref{dri} we see that $\mathrm{Spec}\, \mathrm{dR}(\mathbb{F}_p[x] \to \mathbb{F}_p)$ is isomorphic to $W[F]$ as a pointed $\mathbb{G}_a$-module. Now the lemma follows from applying $\mathrm{dR}'''$ to the same diagrams and going (derived) modulo $\epsilon.$\end{proof}{}

\begin{lemma}The pullback $u^* \mathrm{Spec}\,\mathrm{dR}'''(\mathbb{F}_p[x] \to \mathbb{F}_p)$ is isomorphic to $r (\mathrm{dR}')$ as a pointed $\mathbb{G}_a^{\mathrm{perf}}$-module.
 \end{lemma}{}
 
\begin{proof}We again look at the arrow category of $\mathbb{F}_p$-algebras. The object $(\mathbb{F}_p[x^{1/p^\infty}] \to \mathbb{F}_p[x^{1/p^\infty}])$ is a coring as it corepresents the functor $(A \to B) \to A^\flat.$ The object $(\mathbb{F}_p[x^{1/p^\infty}] \to \mathbb{F}_p[x^{1/p^\infty}]/x)$ is a comodule over $(\mathbb{F}_p[x^{1/p^\infty}] \to \mathbb{F}_p[x^{1/p^\infty}])$ as it corepresents the functor $(A \to B) \to \mathrm{Ker}(A^\flat \to B).$ Further, this produces a map $(\mathbb{F}_p[x^{1/p^\infty}] \to \mathbb{F}_p[x^{1/p^\infty}]) \to (\mathbb{F}_p[x^{1/p^\infty}] \to \mathbb{F}_p[x^{1/p^\infty}]/x)$ of $(\mathbb{F}_p[x^{1/p^\infty}] \to \mathbb{F}_p[x^{1/p^\infty}])$-comodules. Lastly, we have a map $f:(\mathbb{F}_p[x^{}] \to \mathbb{F}_p[x^{}]) \to (\mathbb{F}_p[x^{1/p^\infty}] \to \mathbb{F}_p[x^{1/p^\infty}])$. By taking pushout of the map $$(\mathbb{F}_p[x^{}] \to \mathbb{F}_p[x^{}]) \to (\mathbb{F}_p[x^{}]\to \mathbb{F}_p)$$ of $(\mathbb{F}_p[x^{}]\to \mathbb{F}_p[x^{}])$-comodules along the map $f$ we obtain the map $$(\mathbb{F}_p[x^{1/p^\infty}] \to \mathbb{F}_p[x^{1/p^\infty}]) \to (\mathbb{F}_p[x^{1/p^\infty}] \to \mathbb{F}_p[x^{1/p^\infty}]/x)$$ of $(\mathbb{F}_p[x^{1/p^\infty}] \to \mathbb{F}_p[x^{1/p^\infty}])$-comodules. The rest follows from applying $\mathrm{dR}'''$ to all the diagrams. Indeed, the statement of the lemma now depends upon certain natural colimit diagrams being isomorphisms, which holds since they are known to be isomorphisms after going (derived) modulo $\epsilon.$ \end{proof}{}
Now the two lemmas above show that $r (\mathrm{dR}')$ is indeed a pullback of a deformation of $W[F]$ as a pointed $\mathbb{G}_a$-module, which finishes the proof.\end{proof}{}

\begin{remark}The proof in \cref{sec5.2} somewhat formally reduces the study of deformations of the pointed $\mathbb{G}_a^{\mathrm{perf}}$-module underlying $u^*W[F]$ by showing that any such deformation over $\mathbb{F}_p [\epsilon]/\epsilon^2$ must appear as a pullback of a deformation of $W[F]$ as a pointed $\mathbb{G}_a$-module along the map of ring schemes $u:\mathbb{G}_a^\mathrm{perf} \to \mathbb{G}_a.$ This phenomenon seems to occur more generally under suitable conditions. We formulate and sketch a proof of the following proposition which is motivated by Drinfeld's construction of taking the cone of a quasi-ideal in \cite{Dri20}. We note that the construction of taking the cone is valid for any pointed $\mathbb{G}_a$-module $X$ or any pointed $\mathbb{G}_a^\mathrm{perf}$-module $Y$. However, the cone in this generality only has the structure of a group stack and not a ring stack.
\end{remark}{}

\begin{proposition}\label{lastprop}Let $X$ be a pointed $\mathbb{G}_a$-module which is full of rank $1$ (\textit{cf.} \cref{whot}). Let $M$ be a deformation of $u^*X$ as a pointed $\mathbb{G}_a^{\mathrm{perf}}$-module over $\mathbb{F}_p[\epsilon]/\epsilon^2.$ Then $M \simeq u^*X'$ where $X'$ is a deformation of $X$ over $\mathbb{F}_p[\epsilon]/\epsilon^2$ as a pointed $\mathbb{G}_a$-module. Here $u^*$ is the functor constructed in \cref{pullback}.
 \end{proposition}{}

\begin{proof}One can take the cone of the map $M \to \mathbb{G}_a^{\mathrm{perf}}$ to obtain a flat map $f: \mathbb{G}_a^{\mathrm{perf}} \to [\mathbb{G}_a^{\mathrm{perf}}/M].$ Here $[\mathbb{G}_a^{\mathrm{perf}}/M]$ has the structure of a group stack and the map $f$ is additionally a map of group stacks. Formation of the cone commutes with base change and therefore the cone of $u^*X \to \mathbb{G}_a^\mathrm{perf}$ is given by $\mathbb{G}_a^\mathrm{perf} \to [\mathbb{G}_a^{\mathrm{perf}}/M] \times \mathrm{Spec}\, \mathbb{F}_p.$ Since $X$ is full of rank $1$, by definition $u^*X$ is full of fractional rank $1$ and hence by \cref{hodgestability}, there is a map $\alpha^{\natural} \to M  $ of pointed $\mathbb{G}_a^{\mathrm{perf}}$-modules. By taking cones, we obtain a map of group stacks $\mathbb{G}_a \to [\mathbb{G}_a^{\mathrm{perf}}/M] $ which factors the map $f$ along the map of ring schemes $\mathbb{G}_a^\mathrm{perf} \to \mathbb{G}_a.$ Since the map $\mathbb{G}_a^\mathrm{perf} \to \mathbb{G}_a$ is faithfully flat, it follows that the map $\mathbb{G}_a \to [\mathbb{G}_a^{\mathrm{perf}}/M]$ is flat. By taking the kernel of the map $\mathbb{G}_a \to [\mathbb{G}_a^{\mathrm{perf}}/M],$ we obtain a pointed $\mathbb{G}_a$-module $X'$ which is flat over $\mathrm{Spec}\, \mathbb{F}_p[\epsilon]/\epsilon^2.$ Since $M$ is the kernel of $f$ by construction, it follows that $u^*X' \simeq {M}.$ Since taking kernel commutes with base change, it follows that $X' \times \mathrm{Spec}\, \mathbb{F}_p$ is the kernel of $\mathbb{G}_a \to [\mathbb{G}_a^{\mathrm{perf}}/M] \times \mathrm{Spec}\, \mathbb{F}_p.$ Since $u^*X$ is the kernel of $\mathbb{G}_a^{\mathrm{perf}} \to [\mathbb{G}_a^{\mathrm{perf}}/M] \times \mathrm{Spec}\, \mathbb{F}_p$ it follows that $u^*X \simeq u^* (X' \times \mathrm{Spec}\, \mathbb{F}_p)$. This descends to a natural isomorphism $X \simeq X' \times \mathrm{Spec}\, \mathbb{F}_p$ by \cref{pullback}. Thus $X'$ is indeed a deformation of $X$ and it meets the necessary requirements of the proposition.
\end{proof}{}

\begin{remark}\label{table} We make an informal remark about some of the constructions that appear in this paper. Using the constructions in \cref{section3}, for any quasi-ideal $X$ in $\mathbb{G}_a$ over $\mathbb{F}_p$, one can define a ``cohomology theory" $\mathcal{U}_X$ (equipped with a Hodge filtration) as a functor $\mathcal{U}_X: \mathrm{QSyn}_{\mathbb{F}_p} \to \mathrm{CAlg}(D(\mathbb{F}_p))$. To do so, let us assume that the quasi-ideal $X$ is such that the functor $\mathrm{Un}(u^* X): \mathrm{QRSP} \to \mathrm{CAlg}(D(\mathbb{F}_p))$ is a sheaf \cite[Lemma~4.27]{BMS19}. Then, by \cite[Prop.~4.31]{BMS19}, we get an extended functor that we may call $\mathcal{U}_X: \mathrm{QSyn}_{\mathbb{F}_p} \to \mathrm{CAlg}(D(\mathbb{F}_p)).$ We point out that it would be interesting to isolate a certain class of quasi-ideals for which the sheaf property above is always true. For example, it holds true when $X = W[F]$ or zero. Even if the sheaf property is not true, one can always right Kan extend the functor $\mathrm{Un}(u^* X)$ along $\mathrm{QRSP} \to \mathrm{QSyn}_{\mathbb{F}_p}$ to obtain a functor $\mathcal{U}_X$ as above. In what follows, we let $[\mathbb{G}_a/X]$ denote the ring stack obtained by considering the cone of the quasi-ideal $X.$  \vspace{2mm} 

In \cref{Tab:Tcr} below, we make a list of rough analogies and comparisons that explain the constructions appearing in \cref{lastprop} from a cohomological perspective. The left hand side of the table consists of certain geometric objects such as quasi-ideals or stacks and some natural maps between them. The right hand side of the table consists of values of the cohomology theory $\mathcal{U}_X$ evaluated at certain $\mathbb{F}_p$-algebras and some maps between them. In principle, (whenever $u^*X$ is a nilpotent quasi-ideal) one should expect to obtain the objects on the right hand side by applying the derived global section (of the structure sheaf) functor $R\Gamma (\, \cdot \, , \mathcal{O})$ to the objects on the left hand side, i.e., $\mathcal{U}_X (\mathbb{F}_p[x]) \simeq R\Gamma ([\mathbb{G}_a/X], \mathcal{O})$, $\mathcal{U}_X (\mathbb{F}_p[x^{1/p^\infty}]/x) \simeq R\Gamma(u^* X, \mathcal{O})$ etc. 
\vspace{2mm}

In the case when $X = W[F],$ the functor $\mathcal{U}_X:\w{QSyn}_{\mathbb{F}_p} \to \mathrm{CAlg}(D(\mathbb{F}_p))$ is $S \mapsto \w{dR}(S)$ and the stack $[\mathbb{G}_a/X]$ is $(\mathbb{A}^1_{\mathbb{F}_p})^{\w{dR}}$ \cite[1.5.1]{Dri20}. The fact that the objects on the right hand side of the table are obtained by taking derived global sections of the corresponding objects on the left can be seen by noting some facts from the stacky approach to de Rham cohomology \cite[1.5.1, 3.5.1]{Dri20}, as well as some results on derived de Rham cohomology \cite[Prop. 8.12]{BMS19}. This also gives a rough comparison between the stacky approach and the
more explicit approach taken in our paper. When $X$ is the zero quasi-ideal, then the functor $\mathcal{U}_X:\w{QSyn}_{\mathbb{F}_p} \to \mathrm{CAlg}(D(\mathbb{F}_p))$ is simply $S \mapsto S$ and the stack $[\mathbb{G}_a/X]$ is simply $\mathbb{G}_a.$

\begin{center}
\captionof{table}{Quasi-ideals and cohomology theories\label{Tab:Tcr}}
\bgroup
\def\arraystretch{1.5}
\begin{tabular}{ | C{18em} | C{18em}| } 
\hline
\textbf{Quasi-ideals or stacks} & \textbf{Cohomology theory} \\ 

\hline
$[\mathbb{G}_a/X]$ & $\mathcal{U}_X (\mathbb{F}_p[x])$ \\ 
\hline
$\mathbb{G}_a \to [\mathbb{G}_a/X]$ & $\mathrm{gr}^0_{\mathrm{Hodge}}: \mathcal{U}_X (\mathbb F _p [x]) \to \mathbb{F}_p[x]$\\
\hline
$X = \mathrm{Ker} (\mathbb{G}_a \to [\mathbb{G}_a/X])$ & $\mathcal{U}_X (\mathbb{F}_p[x] \to \mathbb{F}_p) = \mathbb{F}_p \otimes_{\mathcal{U}_X (\mathbb{F}_p[x])} \mathbb{F}_p[x]$\\
\hline
$u^*X$ & $\mathcal{U}_X (\mathbb{F}_p[x^{1/p^\infty}]/x)$\\
\hline
$\alpha^{\natural} \to u^* X$& $\mathrm{gr}^0_{\mathrm{Hodge}}: \mathcal{U}_X (\mathbb{F}_p[x^{1/p^\infty}]/x) \to \mathbb{F}_p[x^{1/p^\infty}]/x$\\
\hline
\end{tabular}
\egroup
\end{center}
\end{remark}{}

\begin{remark}Antieau and Mathew asked us if the statement of \cref{mainthm} remains valid if $A$ is an animated Artinian local ring with residue field $\mathbb{F}_p$. The answer in this case does not seem to follow directly from the results in our paper. One can hope to carry out the strategy of the proof in the case where $A$ is a discrete ring, but for that one would have to talk about structures such as pointed $\mathbb{G}_a$-module or $\mathbb{G}_a^{\mathrm{perf}}$-module over an animated Artinian local ring $A.$ We hope to address this in future.
\end{remark}{}

\begin{remark}\label{lt}It seems to be an interesting question to classify all quasi-ideals in $\mathbb{G}_a$ or $\mathbb{G}_a^{\mathrm{perf}}.$ In particular, motivated by Drinfeld's construction of prismatization, it seems interesting to study the moduli stack $\mathcal{Q}$ of all quasi-ideals in $\mathbb{G}_a$. There is a canonical ring stack $\mathcal{R}$ over $\mathcal{Q}$ obtained by taking the cone of each quasi-ideal in $\mathcal{Q}.$ By using the stacky approach and the version of the unwinding explained in \cref{introex}, it seems plausible to construct a ``universal cohomology theory" $\mathcal{U}$ for varieties over $\mathbb{F}_p$ which would specialize to $\mathcal{U}_X$ from \cref{table} for a quasi-ideal $X \in \mathcal{Q}.$
\end{remark}{}

\newpage

\bibliographystyle{amsalpha}
\bibliography{main}

\providecommand{\bysame}{\leavevmode\hbox to3em{\hrulefill}\thinspace}
\providecommand{\MR}{\relax\ifhmode\unskip\space\fi MR }
\providecommand{\MRhref}[2]{%
  \href{http://www.ams.org/mathscinet-getitem?mr=#1}{#2}
}
\providecommand{\href}[2]{#2}
\begin{thebibliography}{{Sta}22}

\bibitem[Ber74]{Ber74}
P.~Berthelot, \emph{Cohomologie cristalline des sch\'{e}mas de
  caract\'{e}ristique {$p>0$}}, Lecture Notes in Mathematics, Vol. 407,
  Springer-Verlag, Berlin-New York, 1974. \MR{0384804}

\bibitem[Bha12]{Bha12}
B.~Bhatt, \emph{$p$-adic derived de {Rham} cohomology}, 2012, arXiv:1204.6560.

\bibitem[BL22]{BL}
B.~Bhatt and J.~Lurie, \emph{Absolute prismatic cohomology}, 2022,
  arXiv:2201.06120.

\bibitem[BLM21]{BLM20}
B.~Bhatt, J.~Lurie, and A.~Mathew, \emph{Revisiting the de {R}ham-{W}itt
  complex}, Ast\'{e}risque (2021), no.~424, viii+165. \MR{4275461}

\bibitem[BMS19]{BMS19}
B.~Bhatt, M.~Morrow, and P.~Scholze, \emph{Topological {H}ochschild homology
  and integral {$p$}-adic {H}odge theory}, Publ. Math. Inst. Hautes \'{E}tudes
  Sci. \textbf{129} (2019), 199--310. \MR{3949030}

\bibitem[BO78]{BO78}
P.~Berthelot and A.~Ogus, \emph{Notes on crystalline cohomology}, Princeton
  University Press, Princeton, N.J.; University of Tokyo Press, Tokyo, 1978.
  \MR{0491705}

\bibitem[BS19]{BS19}
B.~Bhatt and P.~Scholze, \emph{{Prisms and Prismatic Cohomology}}, 2019,
  arXiv:1905.08229.

\bibitem[CF66]{CF66}
P.~E. Conner and E.~E. Floyd, \emph{The relation of cobordism to
  {$K$}-theories}, Lecture Notes in Mathematics, No. 28, Springer-Verlag,
  Berlin-New York, 1966. \MR{0216511}

\bibitem[Dri18]{Dri18}
V.~Drinfeld, \emph{A stacky approach to crystals}, 2018, arXiv:1810.11853.

\bibitem[Dri21]{Dri20}
\bysame, \emph{Prismatization}, 2021, arXiv:2005.04746.

\bibitem[FF18]{FF18}
L.~Fargues and J.-M. Fontaine, \emph{Courbes et fibr\'{e}s vectoriels en
  th\'{e}orie de {H}odge {$p$}-adique}, Ast\'{e}risque (2018), no.~406,
  xiii+382, With a preface by Pierre Colmez. \MR{3917141}

\bibitem[FJ13]{FJ13}
J.-M. Fontaine and U.~Jannsen, \emph{Frobenius gauges and a new theory of
  {$p$}-torsion sheaves in characteristic $p$. {I}}, 2013, arXiv:1304.3740.

\bibitem[GL21]{GL20}
H.~Guo and S.~Li, \emph{Period sheaves via derived de {R}ham cohomology},
  Compos. Math. \textbf{157} (2021), no.~11, 2377--2406. \MR{4323988}

\bibitem[GR14]{GR14}
D.~Grinberg and V.~Riener, \emph{Hopf algebras in combinatorics}, 2014,
  arXiv:1409.8356.

\bibitem[Gro66]{Gro66}
A.~Grothendieck, \emph{On the de {R}ham cohomology of algebraic varieties},
  Inst. Hautes \'{E}tudes Sci. Publ. Math. (1966), no.~29, 95--103. \MR{199194}

\bibitem[Gro67]{Gro67}
\bysame, \emph{\'{E}l\'{e}ments de g\'{e}om\'{e}trie alg\'{e}brique. {IV}.
  \'{E}tude locale des sch\'{e}mas et des morphismes de sch\'{e}mas {IV}},
  Inst. Hautes \'{E}tudes Sci. Publ. Math. (1967), no.~32, 361. \MR{238860}

\bibitem[Gro68]{Gro68}
\bysame, \emph{Crystals and the de {R}ham cohomology of schemes}, Dix
  expos\'{e}s sur la cohomologie des sch\'{e}mas, Adv. Stud. Pure Math.,
  vol.~3, North-Holland, Amsterdam, 1968, Notes by I. Coates and O. Jussila,
  pp.~306--358. \MR{269663}

\bibitem[Ill71]{Ill71}
L.~Illusie, \emph{Complexe cotangent et d\'{e}formations. {I}}, Lecture Notes
  in Mathematics, Vol. 239, Springer-Verlag, Berlin-New York, 1971.
  \MR{0491680}

\bibitem[Ill72]{Ill72}
\bysame, \emph{Complexe cotangent et d\'{e}formations. {II}}, Lecture Notes in
  Mathematics, Vol. 283, Springer-Verlag, Berlin-New York, 1972. \MR{0491681}

\bibitem[Ill79]{Ill79}
\bysame, \emph{Complexe de de {R}ham-{W}itt et cohomologie cristalline}, Ann.
  Sci. \'{E}cole Norm. Sup. (4) \textbf{12} (1979), no.~4, 501--661.
  \MR{565469}

\bibitem[LL21]{LL}
S.~Li and T.~Liu, \emph{Comparison of prismatic cohomology and derived de
  {Rham} cohomology}, 2021, arXiv:2012.14064.

\bibitem[LM21]{LM21}
S.~Li and S.~Mondal, \emph{On endomorphisms of the de {Rham} cohomology
  functor}, 2021, arXiv:2109.04303.

\bibitem[Lur09]{Luuu}
J.~Lurie, \emph{Higher topos theory}, Annals of Mathematics Studies, vol. 170,
  Princeton University Press, Princeton, NJ, 2009.

\bibitem[Lur17]{Lur17}
J.~Lurie, \emph{Higher algebra}, available at
  \url{https://www.math.ias.edu/~lurie/papers/HA.pdf}, 2017.

\bibitem[Mat16]{Mat16}
A.~Mathew, \emph{The {G}alois group of a stable homotopy theory}, Adv. Math.
  \textbf{291} (2016), 403--541. \MR{3459022}

\bibitem[MM65]{MM65}
J.~W. Milnor and J.~C. Moore, \emph{On the structure of {H}opf algebras}, Ann.
  of Math. (2) \textbf{81} (1965), 211--264. \MR{174052}

\bibitem[MRT21]{MRT20}
T.~Moulinos, M.~Robalo, and B.~{To{\"e}n}, \emph{{A Universal HKR Theorem}},
  2021, arXiv:1906.00118.

\bibitem[MV99]{MV99}
F.~Morel and V.~Voevodsky, \emph{{${\bf A}^1$}-homotopy theory of schemes},
  Inst. Hautes \'{E}tudes Sci. Publ. Math. (1999), no.~90, 45--143 (2001).
  \MR{1813224}

\bibitem[Per76]{Per76}
D.~Perrin, \emph{Approximation des sch\'{e}mas en groupes, quasi compacts sur
  un corps}, Bull. Soc. Math. France \textbf{104} (1976), no.~3, 323--335.
  \MR{432661}

\bibitem[Rak20]{Rak20}
A.~Raksit, \emph{{Hochschild homology and the derived de Rham complex
  revisited}}, 2020, arXiv:2007.02576.

\bibitem[{Sta}22]{SP}
The {Stacks project authors}, \emph{The {Stacks} project}, available at
  \url{https://stacks.math.columbia.edu}, 2022.

\bibitem[To{\"e}20]{Toe20}
B.~To{\"e}n, \emph{{Classes caract\'eristiques des sch\'emas feuillet\'es}},
  2020, arXiv:2008.10489.

\end{thebibliography}

\end{document}